\documentclass[11pt]{article}

\usepackage[utf8]{inputenc}
\usepackage{algpseudocode}
\usepackage{amsfonts}
\usepackage{amsmath}
\usepackage{amssymb}
\usepackage{amsthm}
\usepackage{calligra}
\usepackage{cancel}
\usepackage{caption}
\usepackage{color}
\usepackage{enumerate}
\usepackage{esint}
\usepackage{geometry}
\usepackage{graphics}
\usepackage{graphicx}
\usepackage{hyperref}
\usepackage{latexsym}
\usepackage{mathrsfs}
\usepackage{pifont}
\usepackage{slashed}
\usepackage{tikz}
\usepackage{tkz-euclide}
\usepackage{times}
\usepackage{upgreek}
\usepackage{url}
\usepackage{verbatim}

\hypersetup{
    colorlinks,
    citecolor=red,
    filecolor=green,
    linkcolor=blue,
    urlcolor=black
}

\def\al{\alpha}
\def\de{\delta}
\def\ga{\gamma}
\def\ka{\kappa}
\def\la{\lambda}
\def\Om{\Omega}
\def\om{\omega}
\def\pa{\partial}

\theoremstyle{plain}
\newtheorem{proposition}{Proposition}[section]
\newtheorem{lemma}[proposition]{Lemma}
\newtheorem{theorem}[proposition]{Theorem}
\newtheorem{corollary}[proposition]{Corollary}

\theoremstyle{definition}
\newtheorem{definition}[proposition]{Definition}

\newtheorem{remark}[proposition]{Remark}

\title{Supersonic Gravitational Collapse for Non-Isentropic Gaseous Stars}

\author{Christopher Alexander\footnote{Department of Mathematics, University College London, WC1H 0AY, United Kingdom, christopher.alexander@ucl.ac.uk},~ Mahir Had\v{z}i\'{c}\footnote{Department of Mathematics, University College London, WC1H 0AY, United Kingdom, m.hadzic@ucl.ac.uk}~ ~and~ Matthew Schrecker\footnote{Department of Mathematics, University of Bath, BA2 7AY, United Kingdom, mris21@bath.ac.uk}}

\date{}

\begin{document}

\maketitle

\begin{abstract}
	We show the existence of a new class of initially smooth spherically symmetric self-similar solutions to the non-isentropic Euler--Poisson system. These solutions exhibit supersonic gravitational implosion in the sense that the density blows-up in finite time while the fluid velocity remains supersonic. In particular, they occupy a portion of the phase space that is far from the recently constructed isentropic self-similar implosion. At the heart of our proof is the presence of a two-parameter scaling invariance and the reduction of the problem to a non-autonomous system of ordinary differential equations. We use the requirement of smoothness of the flow as a selection principle that constrains the choice of scaling indices. An important consequence of our analysis is that for all the solutions we construct, the polytropic index $\ga$ is strictly bigger than $\frac{4}{3}$, which is in sharp contrast to the known results in the isentropic case.
\end{abstract}

\tableofcontents

\section{Introduction}

A self-gravitating Newtonian star is modelled using the non-isentropic Euler--Poisson system (or full Euler-Poisson system), which describes the conservation of mass, momentum and energy in the presence of non-constant entropy. The basic unknowns are the density $\varrho$, the pressure $P$, the velocity vector $\boldsymbol{u}$, the specific internal energy $e$ and the gravitational potential $\varphi$. The resulting system of partial differential equations read:
\begin{align}
	\pa_t\varrho + \nabla\cdot(\varrho\boldsymbol{u}) &= 0,\label{E:Con}\\
	\varrho (\pa_t + \boldsymbol{u}\cdot\nabla) \boldsymbol{u} + \nabla P + \varrho\nabla\varphi &= 0,\\
	\pa_t\bigg(\varrho\bigg(e+\frac{1}{2} |\boldsymbol{u}|^2\bigg)\bigg)+\nabla\cdot\bigg(\varrho \boldsymbol{u}\bigg(e+\frac{1}{2} |\boldsymbol{u}|^2+\frac{P}{\varrho}\bigg)\bigg) &= 0,\label{E:Eng}\\
	\Delta\varphi &= 4\pi \varrho.\label{E:Poi}
\end{align}
The system~\eqref{E:Con}--\eqref{E:Poi} is referred to as the \emph{non-isentropic gravitational Euler--Poisson system}. Under the assumption of an ideal gas, we introduce the equation of state
\begin{align}
	P = (\ga-1)\varrho e,
\end{align}
where the adiabatic exponent $\ga>1$. Then, for smooth solutions of~\eqref{E:Con}--\eqref{E:Eng}, the specific entropy is transported along the flow, giving us the transport equation
\begin{align}
	(\pa_t + \boldsymbol{u}\cdot\nabla)\bigg(\frac{P}{\varrho^\ga}\bigg) = 0.\label{E:Ent}
\end{align}
As we will work with smooth initial data for~\eqref{E:Con}--\eqref{E:Poi}, we will work for convenience with~\eqref{E:Ent} throughout, rather than the more cumbersome (but equivalent)~\eqref{E:Eng}. Unlike the isentropic case, where the equation of state $P=\varrho^\ga$ is fixed, we allow for the presence of non-trivial entropy, which is encoded through the \emph{dynamic} equation~\eqref{E:Ent}.

One of the central questions in the description of qualitative dynamics of self-gravitating fluids is the problem of gravitational collapse. We are interested in dynamically forming singularities characterised by the blow up of the fluid density and refer to such a process as star implosion. This question has been extensively studied in the Newtonian context in the astrophysics literature, beginning with the pioneering and independent papers of Larson and Penston~\cite{Larson69,Penston69} which identified numerically an exact self-similar collapsing solution for isothermal gas dynamics, later termed the Larson--Penston (LP) solution. A further discrete family of possible collapse profiles was found numerically by Hunter~\cite{Hunter77}, but later partial analyses of numerical mode stability in~\cite{HanawaNakayama97,MaedaHarada01} suggested that the LP solution should be the only stable solution among the LP and Hunter solutions. In the non-isentropic case, there has been much less attention, especially in the spherically symmetric case. A family of collapsing solutions with variable entropy has been numerically constructed by Lou--Shi~\cite{LouShi14} for the range of adiabatic exponent $\ga>\frac{4}{3}$ (termed general polytropic Larson--Penston solutions by the authors). 

Stellar implosion has been given a lot of attention in mathematics literature recently, but only in the isentropic context mentioned above. This includes the construction of the LP solutions~\cite{GuoHadzicJang21B} for the isothermal case $\ga=1$, the existence of self-similar imploding stars in the full supercritical regime $1<\ga<\frac{4}{3}$ in~\cite{GuoHadzicJangSchrecker22}, as well as examples of dust-like collapse in the supercritical regime~\cite{GuoHadzicJang21A}. Moreover, in the case of the Einstein--Euler system, the relativistic analogues of the Larson--Penston solution were constructed in~\cite{GuoHadzicJang23}, giving rise to spacetimes dynamically forming naked singularities from smooth initial data. We also mention~\cite{Sandine24}, in which a subclass of the above mentioned Hunter solutions has been constructed.

The purpose of this work is to rigorously construct examples of self-similar imploding singularities for the non-isentropic Euler--Poisson system~\eqref{E:Con}--\eqref{E:Poi}. It is well-known~\cite{MerleRaphaelRodnianskiSzeftel22A,GuoHadzicJang21B,GuoHadzicJangSchrecker22,GuoHadzicJang23,Hadzic23,Sandine24} that one of the most difficult challenges in the construction of such solutions is the presence of so-called sonic hypersurfaces, which correspond to the boundary of the backward sound cone emanating from the singularity at $(t,r)=(0,0)$. Such surfaces divide the space-time into the sub- and super-sonic regions and lead to a number of difficulties in the analysis.

A simple calculation in self-similar coordinates shows that a point $(t,r)$ lies on the sonic hypersurface precisely when the so-called relative velocity (defined below in~\eqref{E:omegaDef}) is equal to the sound speed of the fluid (see Definition~\ref{D:Sonic} below). This motivates the distinction of subsonic and supersonic flow depending on whether the relative velocity is greater or less than the sound speed. For the flows we construct here, we will in fact obtain solutions that are globally supersonic, except for a single sonic point at the origin. This sonic point is due to the vanishing of both sound speed and velocity at the origin, a feature of non-isentropic self-similar flows that is also seen, for example, in the G\"uderley imploding shock solution (briefly discussed below). As a consequence of this vanishing, the backwards acoustic cone will be identical to the line $\{r=0\}$.

\begin{theorem}[Main Theorem: Informal statement]\label{T:MainInformal}
	Let $\ga\in\big(\frac{19}{12},\frac{11}{6}\big)$. Then there exists a smooth, spherically symmetric, supersonic solution to~\eqref{E:Con}--\eqref{E:Poi} on $(t,x)\in(-\infty,0)\times \mathbb{R}^3$ such that, as $t\to0^-$, $\varrho(t,0)\to\infty$.
\end{theorem}

As we explain below, the presence of non-constant entropy allows the choice of the polytropic exponent $\ga>\frac{4}{3}$ to be compatible with supercriticality. The particular range $\ga\in\big(\frac{19}{12},\frac{11}{6}\big)$ included in this theorem constitutes the first band of the adiabatic indices $\ga$ for which such solutions should be expected with the basic physical property that the density is maximised at the origin (see Remark~\ref{R:gammaBands} below).

\begin{remark}
	We note that the above theorem covers the value $\ga=\frac{5}{3}$, which corresponds to the adiabatic index of atomic hydrogen gas and is thus important for astrophysical considerations (see~\cite{LouShi14}).
\end{remark}

In the fluid dynamics and astrophysics literature, self-similar solutions have been extensively studied due to their universality and the expectation that they form attractors for the dynamics of fluid flows. Such self-similar flows have typically been distinguished as of one of two types~\cite{ZeldovichRazier67}: Type I if the self-similar scaling parameters are determined from dimensional analysis and Type II otherwise.\footnote{Note that this differs from some other notions of Type I and Type II blow-up used in the PDE literature.} Our solutions are of the second kind, as the scaling is not fixed by dimensional considerations. Instead, the scaling is determined via a regularity criterion: the solution must be smooth, in particular at the sonic point. The use of regularity to identify the physically meaningful scaling solution for Type II self-similar flows has a long history. It is a key feature of the G\"uderley imploding shock solution for the non-isentropic Euler equations, as well as the recent construction of smooth imploding solutions for isentropic gas dynamics~\cite{MerleRaphaelRodnianskiSzeftel22A,MerleRaphaelRodnianskiSzeftel22B} and self-gravitating implosion~\cite{GuoHadzicJang21B,GuoHadzicJangSchrecker22,GuoHadzicJang23}. The imploding shock solution for the non-isentropic Euler equations was discovered by G\"uderley~\cite{Guderley42} in 1942 and has been studied extensively by mathematicians, physicists and engineers. These solutions were rigorously constructed in~\cite{JangLiuSchrecker25A,JangLiuSchrecker25B}. For further numerical and other mathematical results concerning the G\"uderley solutions, as well as an overview of the literature, see~\cite{JangLiuSchrecker25A,JenssenTsikkou18,Lazarus81} and the references therein. In both the shock implosion and the smooth implosion results, the parameter is found a posteriori through a shooting argument which does not yield an exact form, whereas in this paper, the scaling may still be fixed a priori, as a simple analysis at the origin demonstrates which value of scaling parameter gives smooth solutions.

\subsection{Scaling Symmetries}

For the remainder of this paper we work in spherical symmetry. We therefore assume that:
\begin{align*}
	\varrho(t,\boldsymbol{x}) &= \tilde{\rho}(t,r), & \boldsymbol{u}(t,\boldsymbol{x}) &= \tilde{u}(t,r) \frac{\boldsymbol{x}}{r}, & P(t,\boldsymbol{x}) &= \tilde{p}(t,r).
\end{align*}
The resulting system then reads:
\begin{align}
	\pa_t\tilde{\rho} + \bigg(\pa_r + \frac{2}{r}\bigg)(\tilde{\rho}\tilde{u}) &= 0,\label{E:SSCon}\\
	(\pa_t + \tilde{u}\pa_r)\tilde{u} + \frac{\pa_r \tilde p}{\tilde \rho} + \frac{M[\tilde{\rho}](r)}{r^2} &= 0,\label{E:SSMom}\\
	(\pa_t + \tilde{u}\pa_r)\bigg(\frac{\tilde{p}}{\tilde{\rho}^\ga}\bigg) &= 0,\label{E:SSEnt}
\end{align}
where we have introduced the local mass
\begin{align*}
	M[\tilde{\rho}](r) := \int^r_0 4\pi z^2\tilde{\rho}\, dz.
\end{align*}
Observe that in spherical symmetry we can easily integrate the Poisson equation~\eqref{E:Poi} over a ball of radius $r$ to conclude that
\begin{align}
	\nabla\varphi = \frac{M[\tilde{\rho}](r)}{r^2}\frac{\boldsymbol{x}}{r}\label{E:NablaPhi}
\end{align}
and therefore the formulation~\eqref{E:SSCon}--\eqref{E:SSEnt} is indeed the spherically symmetric version of the original problem.

Unlike the isentropic Euler--Poisson system, the non-isentropic version has a two-parameter family of scaling invariances. This is, of course, expected due to the presence of an additional dynamic equation~\eqref{E:SSEnt} which encodes the transport of entropy.

Assume that the triple $(\tilde{\rho},\tilde{u},\tilde{p})$ solves the Euler--Poisson system~\eqref{E:SSCon}--\eqref{E:SSEnt}. It is a matter of direct calculation to then check that for any choice of $\al\in\mathbb{R}$ the triple $(\rho, u, p)$ given by:
\begin{align}
	\tilde{\rho} &= \la^{a_1(\al)}\rho(s,y), &
	\tilde{u} &= \la^{a_2(\al)}u(s,y), &
	\tilde{p} &= \la^{a_3(\al)} p(s,y),\label{E:Scaling}
\end{align}
also solves~\eqref{E:SSCon}--\eqref{E:SSEnt}, where:
\begin{align}
	s &= \frac{t}{\la^{b(\al)}}, &
	y &= \frac{r}{\la}\label{E:syDef}
\end{align}
and:
\begin{align}
	a_1(\al) &= -\frac{2-\al}{2-\ga}, &
	a_2(\al) &= -\frac{2(\ga-1)-\al}{2(2-\ga)}, &
	a_3(\al) &= -\frac{2(\ga-\al)}{2-\ga}, &
	b(\al) &= \frac{2-\al}{2(2-\ga)}.\label{E:SSab}
\end{align}
Under this scaling, the specific entropy is scaled as
\begin{align*}
	\frac{\tilde{p}}{\tilde{\rho}^\ga} = \la^\al\frac{p}{\rho^\ga}(s,y).
\end{align*}
We observe that the total mass scales according to the law
\begin{align*}
	M[\tilde{\rho}](r) = \la^{a_1(\al)+3}M[\rho](y) = \la^{\frac{4-3\ga+\al}{2-\ga}}M[\rho](y),
\end{align*}
which allows us to introduce a natural notion of mass supercriticality in the problem. We say that the scaling~\eqref{E:Scaling}--\eqref{E:SSab} is \emph{mass supercritical} if
\begin{align}
	1 \le \ga < \frac{4+\al}{3}.\label{E:SuperCrit}
\end{align}
Therefore, when $\al>0$ supercriticality is compatible with a range of choices of the polytropic exponent $\ga$ strictly larger than $\frac{4}{3}$, that is, the regime where $\frac{4}{3}<\ga<\frac{4+\al}{3}$. This should be contrasted to the isentropic case, where supercriticality corresponds to the range~\eqref{E:SuperCrit} with $\al=0$ and therefore $\ga<\frac{4}{3}$.

\begin{remark}[The mass-critical case]\label{R:MassCritical}
	If $\al=3\ga-4$, we refer to such a choice of the pair $(\al,\ga)$ as \emph{mass-critical}. It is well-known that in the isentropic mass-critical case $(\al,\ga)=\big(0,\frac{4}{3}\big)$, there exists a finite-parameter family of collapsing and expanding compactly supported solutions discovered by Goldreich and Weber in 1980~\cite{GoldreichWeber80}. For the non-isentropic mass-critical choices $(\al,\ga)=(3\ga-4,\ga)$, the analogues of such solutions were discovered by Deng, Xiang and Yang~\cite{DengXiangYang03}. We also refer the reader to~\cite{Makino92,FuLin98,LiLouEsimbek18}. The nonlinear stability of the expanding Goldreich--Weber stars was shown in~\cite{HadzicJang18,HadzicJangLam22}.
\end{remark}

Observe that there are certain natural constraints on the choice of $\al$. To maintain an entropy bounded from above in the limit $\la\to0$, and to exhibit pressure and density blow up, we demand that $a_1,a_3<0$ and $\al>0$ in~\eqref{E:Scaling}, which together with~\eqref{E:SuperCrit} leads to the constraints:
\begin{align}
	\al &\in (3\ga-4,\ga) \subset (0,2), & \ga &> \frac{4}{3}.\label{E:Constraints}
\end{align}
Motivated by the scaling~\eqref{E:Scaling}--\eqref{E:SSab}, we look for a self-similar solution to~\eqref{E:SSCon}--\eqref{E:SSEnt} of the form:
\begin{align}
	\tilde{\rho} &= (-t)^{\frac{a_1}{b}}\rho(y), &
	\tilde{u} &= (-t)^{\frac{a_2}{b}}u(y), &
	\tilde{p} &= (-t)^{\frac{a_3}{b}}p(y),
	\label{E:SSrho}
\end{align}
where the coefficients $a_1,a_2,a_3,b$ are given by~\eqref{E:SSab} and in accordance with~\eqref{E:syDef}
\begin{align}
	y = \frac{r}{(-t)^\frac{1}{b}}.\label{E:SSyDef}
\end{align}
By analogy to~\cite{GuoHadzicJang21B,GuoHadzicJangSchrecker22,GuoHadzicJang23}, it will be particularly convenient to work with the so-called \emph{relative velocity} given by
\begin{align}
	\om(y) := (2-\ga)\frac{bu + y}{y} = \frac{(2-\al)u + 2(2-\ga)y}{2y},\label{E:omegaDef}
\end{align}
where we have used~\eqref{E:SSab}.

\begin{remark}
	The contracting Goldreich--Weber solutions mentioned in Remark~\ref{R:MassCritical} collapse homologously, that is, the self-similar fluid velocity $u$ is proportional to $y$. The quantity $y\omega$ thus corresponds to the fluid velocity relative to a frame comoving with homologous collapse, that is, a positive $\omega$ means the fluid is moving outwards (collapsing slower) relative to the homologous frame. It is for this reason $y\omega$ is also referred to as the \emph{wind velocity}~\cite{Yahil83}.
\end{remark}

In the next lemma, we formulate the non-autonomous system of ODE satisfied by the self-similar unknowns defined in~\eqref{E:SSrho}--\eqref{E:omegaDef}.

\begin{lemma}[The self-similar reduction]
	Let $(\rho,u, p)$ be the self-similar unknowns introduced through the ansatz~\eqref{E:SSrho}--\eqref{E:SSyDef}. Assume further that the triple $(\tilde{\rho}(t,\cdot),\tilde{u}(t,\cdot),\tilde{p}(t,\cdot))$ constitutes a smooth solution of the non-isentropic Euler--Poisson system~\eqref{E:SSCon}--\eqref{E:SSEnt} on some time interval $(-t,0)$, $t>0$. Then the pair $(\rho,\om)$ solves the system:
	\begin{align}
		\rho' &= \frac{\rho(yh-q)}{G},\label{E:rhoPrime1}\\
		\om' &= \frac{4-3\ga+\al-3\om}{y} - \frac{\om(yh-q)}{G},\label{E:omegaPrime1}
	\end{align}
	where $\om$ is given by~\eqref{E:omegaDef}, $p$ by
	\begin{align}
		p = \rho^{\ga}(y^3\rho\om)^{\frac{(2-\ga)\al}{4-3\ga+\al}}\label{E:p}
	\end{align}
	and:
	\begin{align}
		h(\rho,\om) &:= 2\om^2 + \Big(\ga-1-\frac{\al}{2}\Big)\om + \Big(\ga-1-\frac{\al}{2}\Big)(2-\ga) - \frac{4\pi}{4-3\ga+\al}\Big(1-\frac{\al}{2}\Big)^2\rho\om,\label{E:hDef}\\
		q(y;\rho,\om) &:= (2-\ga)\al\Big(1-\frac{\al}{2}\Big)^2\frac{p}{y\rho\om},\label{E:qDef}\\
		G(y;\rho,\om) &:= \ga\Big(1-\frac{\al}{2}\Big)^2\frac{p}{\rho} - y^2\om^2.\label{E:GDef}
	\end{align}
\end{lemma}

\begin{proof}
Substituting the self-similar variables into~\eqref{E:SSCon},~\eqref{E:SSMom} and~\eqref{E:SSEnt}, we obtain:
\begin{align}
	-\frac{a_1}{b}\rho + \frac{1}{b}y\pa_y\rho + \bigg(\pa_y + \frac{2}{y}\bigg)(\rho u) &= 0,\label{E:SSSSCon1}\\
	-\frac{a_2}{b} u + \frac{1}{b}y\pa_y u + u\pa_y u + \frac{\pa_y p}{\rho} + \frac{M[\rho](y)}{y^2} &= 0,\label{E:SSSSMom1}\\
	\left(-\frac{\al}{b} + \frac{1}{b}y\pa_y + u\pa_y\bigg)\bigg(\frac{p}{\rho^\ga}\right) &= 0,\label{E:SSSSEnt1}
\end{align}
where we have used $a_1=2a_2-2$, $b=1-a_2$ and \begin{align*}
	M[\tilde{\rho}](r) = (-t)^{\frac{a_1+3}{b}}M[\rho](y).
\end{align*}
Multiplying~\eqref{E:SSSSCon1} by $y^2$, rearranging and integrating over $[0,y]$, we conclude that
\begin{align*}
	M[\rho](y) = \frac{4\pi}{a_1+3}y^2\rho(bu + y),
\end{align*}
which by~\eqref{E:SSab} and~\eqref{E:omegaDef} gives
\begin{align}
	M[\rho](y) = \frac{4\pi}{4-3\ga+\al}y^3\rho\om.\label{E:SSSSMass}
\end{align}
In particular, the system~\eqref{E:SSSSCon1}--\eqref{E:SSSSEnt1} with~\eqref{E:SSSSMass} rearranges to give:
\begin{align}
	\pa_y(y^3\rho\om) - (4-3\ga+\al)y^2\rho &= 0,\label{E:SSSSCon2}\\
	-\Big(\ga-1-\frac{\al}{2}\Big)(2-\ga)y - \Big(2(2-\ga)-1+\frac{\al}{2}\Big)y\om + y\om^2 + y^2\om\pa_y\om&\notag\\
	+ \Big(1-\frac{\al}{2}\Big)^2\frac{\pa_yp}{\rho} + \frac{4\pi}{4-3\ga+\al}\Big(1-\frac{\al}{2}\Big)^2y\rho\om &= 0,\label{E:SSSSMom2}\\
	\frac{\pa_yp}{p} - \ga\frac{\pa_y\rho}{\rho} - \frac{(2-\ga)\al}{y\om} &= 0.\label{E:SSSSEnt2}
\end{align}
Upon rearranging~\eqref{E:SSSSCon2}, we obtain the identity
\begin{align*}
	\frac{(2-\ga)\al}{y\om} = \frac{(2-\ga)\al}{4-3\ga+\al}\bigg(\frac{3}{y} + \frac{\pa_y\rho}{\rho} + \frac{\pa_y\om}{\om}\bigg),
\end{align*}
which we use in~\eqref{E:SSSSEnt2} to solve for the self-similar pressure $p$, yielding
\begin{align}
	p &= \ka\rho^\ga (y^3\rho\om)^{\frac{(2-\ga)\al}{4-3\ga+\al}}\label{E:pkappa}
\end{align}
for some positive constant $\ka$. In order to set $\ka=1$ without loss of generality, we further rescale:
\begin{align*}
	\hat{y} &= \ka^{\frac{1}{n-2}}y, & \hat{p} &= \ka^{\frac{2}{n-2}} p\big(\ka^{-\frac{1}{n-2}}\hat{y}\big),
\end{align*}
where
\begin{align*}
	n = \frac{3(2-\ga)\al}{4-3\ga+\al}.
\end{align*}
One easily checks that the system~\eqref{E:SSSSCon2}--\eqref{E:SSSSMom2} and~\eqref{E:pkappa} remains invariant under this further rescaling and we find that~\eqref{E:pkappa} is obtained precisely with $\ka=1$, which is~\eqref{E:p}. Now substituting equation~\eqref{E:SSSSEnt2} into equation~\eqref{E:SSSSMom2} and solving for $\pa_y\rho$ and $\pa_y\om$ yields:
\begin{align}
	\pa_y\rho &= \frac{1}{G}(\rho B - y\om A),\label{E:SSSSrhoPrime}\\
	\pa_y\om &= -\frac{1}{G}\bigg(\om B - \ga\Big(1-\frac{\al}{2}\Big)^2\frac{p}{y\rho^2}A\bigg),\label{E:SSSSomegaPrime}
\end{align}
where $G$ is given by~\eqref{E:GDef} and:
\begin{align*}
	A &= (4-3\ga+\al)\rho - 3\rho\om,\\
	B &= \Big(\ga-1-\frac{\al}{2}\Big)(2-\ga)y + \Big(2(2-\ga)-1+\frac{\al}{2}\Big)y\om - y\om^2 - \frac{4\pi}{4-3\ga+\al}\Big(1-\frac{\al}{2}\Big)^2y\rho\om\\
	&- (2-\ga)\al\Big(1-\frac{\al}{2}\Big)^2\frac{p}{y\rho\om}.
\end{align*}
Upon regrouping~\eqref{E:SSSSrhoPrime}--\eqref{E:SSSSomegaPrime}, we finally obtain~\eqref{E:rhoPrime1}--\eqref{E:GDef}.
\end{proof}

\begin{remark}
	We note that for $\al=0$, equations~\eqref{E:rhoPrime1}--\eqref{E:omegaPrime1} exactly correspond to the self-similar reduction of the polytropic isentropic Euler--Poisson system from~\cite{GuoHadzicJangSchrecker22}.
\end{remark}

\begin{remark}
	It is natural to demand $\al<2$ according to~\eqref{E:SSrho} and~\eqref{E:Scaling} in order to observe implosion as $t\to0^-$.
\end{remark}

We see from the system of equations~\eqref{E:rhoPrime1}--\eqref{E:omegaPrime1} that points where $G=0$ play a crucial role. These are in fact the sonic points of the flow.

\begin{definition}[Sonic point]\label{D:Sonic}
	Let $(\rho,\om)$ be a $C^1$ solution of~\eqref{E:rhoPrime1}--\eqref{E:omegaPrime1} on a relatively open interval $I\subset [0,\infty)$. We say that a point $y_*\in I$ is a \emph{sonic point} if
	\begin{align*}
		G(y_*;\rho(y_*),\om(y_*)) = \ga\Big(1-\frac{\al}{2}\Big)^2\frac{p(y_*)}{\rho(y_*)} - y_*^2\om(y_*)^2 = 0.
	\end{align*}
	If $G(y_*;\rho(y_*),\om(y_*))>0$, we say the flow is \emph{subsonic} at $y_*$, while if $G(y_*;\rho(y_*),\om(y_*))<0$, we say the flow is \emph{supersonic} at $y_*$.
\end{definition}

A simple calculation based on this formula shows that at a sonic point, the curve $\{y=y_*\}$ corresponding to
\begin{align*}
	r(t) = y_*(-t)^{\frac{2(2-\ga)}{2-\al}}
\end{align*}
satisfies the relation
\begin{align*}
	\dot{r}(t) = \tilde{u}(t,r(t))-c_s(t,r(t)),
\end{align*}
where we recall that the sound speed \begin{align*}
	c_s = \sqrt{\ga\frac{\tilde{p}}{\tilde{\rho}}}.
\end{align*}
That is, $\{y=y_*\}$ corresponds to the backwards acoustic cone from the origin.

\begin{definition}[The far-field solution]\label{D:FarField}
	It can be easily checked that an explicit solution to~\eqref{E:rhoPrime1}--\eqref{E:omegaPrime1} is given by the triplet:
	\begin{align*}
		\rho_f &= \bar{\rho}_f y^{-\frac{2-\al}{2-\ga}}, &
		\om_f &= 2-\ga, &
		p_f &= \frac{2\pi(2-\ga)^2\bar{\rho}_f^2}{(\ga-\al)(4-3\ga+\al)}y^{-\frac{2(\ga-\al)}{2-\ga}},
	\end{align*}
	where
	\begin{align*}
		\bar{\rho}_f = (2-\ga)^{-\frac{\al}{3\ga-4}}\bigg(\frac{2\pi(2-\ga)^2}{(\ga-\al)(4-3\ga+\al)}\bigg)^{\frac{4-3\ga+\al}{(3\ga-4)(2-\ga)}}.
	\end{align*}
	We refer to the above triplet as the \emph{far-field} solution of~\eqref{E:rhoPrime1}--\eqref{E:omegaPrime1}.
\end{definition}

The far-field solution from Definition~\ref{D:FarField} represents a time-independent steady state solution
\begin{align*}
	(\rho_f,u_f,p_f) = \bigg(\bar{\rho}_fr^{-\frac{2-\al}{2-\ga}},0,\frac{2\pi (2-\ga)^2\bar{\rho}_f^2}{(\ga-\al)(4-3\ga+\al)} r^{-\frac{2(\ga-\al)}{2-\ga}}\bigg)
\end{align*}
of the non-isentropic Euler--Poisson system~\eqref{E:SSCon}--\eqref{E:SSEnt}. It has the unphysical feature that it blows up at the origin, but on the other hand it correctly models the asymptotic behaviour in $y$ as $y\to\infty$.

\subsection{Smoothness as a Selection Principle}

We emphasise that the smoothness of the self-similar profile across the sonic surface is a key requirement for stability (or even finite codimension stability) of imploding solutions as solutions of the Euler flow~\cite{MerleRaphaelRodnianskiSzeftel22A,MerleRaphaelRodnianskiSzeftel22B}. In the context of Euler flows with gravity, smoothness is used as a criterion in the physics literature to single out possibly stable imploding flows (see for example the discussion in~\cite{HaradaMaeda01,OriPiran90}). We are thus interested in finding smooth solutions to the reduced ODE system~\eqref{E:rhoPrime1}--\eqref{E:omegaPrime1}. It is therefore the case that at $y=0$ we must have
\begin{align}
	\om(0) = \om_0 = \frac{4-3\ga+\al}{3},\label{E:omega0}
\end{align}
and to model implosion at the origin, we also demand
\begin{align}
	\rho(0) = \rho_0 > 0.\label{E:rho0}
\end{align}

The following lemma crucially highlights how the requirement of smoothness of the self-similar solution imposes a selection principle that leads to the requirement $\ga>\frac{4}{3}$ and simultaneously fixes the remaining scaling freedom in the problem.

\begin{lemma}\label{L:SP}
	If $\al\in(0,2)$ and $1\leq\ga<\frac{4}{3}$, there is no smooth solution to~\eqref{E:rhoPrime1}--\eqref{E:omegaPrime1} satisfying the initial conditions~\eqref{E:omega0}--\eqref{E:rho0}. If, on the other hand, $\ga>\frac{4}{3}$ and~\eqref{E:Constraints} holds, then a necessary condition for the existence of a smooth solution to~\eqref{E:rhoPrime1}--\eqref{E:omegaPrime1} is
	\begin{align}
		n := n(\ga,\al) &= \frac{3(2-\ga)\al}{4-3\ga+\al}\in(2\mathbb{Z})\cap (2,\infty), & \al &\in (3\ga-4,\ga).\label{E:nDef}
	\end{align}
	Moreover, any such solution must satisfy the initial conditions:
	\begin{align}
		\rho(0) &= \frac{1}{6\pi}, & \om(0) &= \frac{4-3\ga+\al}{3}\label{E:IC}.
	\end{align}
\end{lemma}

\begin{proof}
Assume first that $\ga\in\big(1,\frac{4}{3}\big)$. Then, from~\eqref{E:p} and~\eqref{E:GDef}, we have 
\begin{align}
	G(y;\rho,\om) = \ga\Big(1-\frac{\al}{2}\Big)^2\frac{p}{\rho}-y^2\om^2 = \ga\Big(1-\frac{\al}{2}\Big)^2\rho^{\ga-1} (y^3\rho\om)^{\frac{(2-\ga)\al}{4-3\ga+\al}}-y^2\om^2.\label{E:GDefAlt}
\end{align}
For $\ga\in\big(1,\frac{4}{3}\big)$, $\al\in(0,2)$, one checks that
\begin{align*}
	\frac{3(2-\ga)\al}{4-3\ga+\al} \in (0,2)
\end{align*}
so that as $y\to0$, 
\begin{align}
	G\underset{y\to0}\sim \ga\Big(1-\frac{\al}{2}\Big)^2\frac{p}{\rho}\label{E:AsympG0}
\end{align}
due to~\eqref{E:omega0}--\eqref{E:rho0}. Comparing now to the ODE for $\rho$, we recall from~\eqref{E:qDef} that
\begin{align*}
	q = (2-\ga)\al\Big(1-\frac{\al}{2}\Big)^2\frac{p}{y\rho\om}\underset{y\to0}\sim y^\beta
\end{align*}
where
\begin{align}
	\beta := \beta(\ga,\al) = \frac{3(2-\ga)\al}{4-3\ga+\al}-1 = n(\ga,\al)-1,
\end{align}
also recalling~\eqref{E:nDef}. Note that $\beta\in(-1,1)$ and since $h$ defined by~\eqref{E:hDef} is order 1 near $y=0$, we conclude that
\begin{align*}
	yh-q\underset{y\to0}\sim -q.
\end{align*}
Therefore, using this, ~\eqref{E:rhoPrime1} and~\eqref{E:AsympG0} we infer that to the leading order
\begin{align*}
	\rho\underset{y\to0}\sim y^{-\frac{3 (2-\ga) \al}{\ga(4-3\ga+\al)}},
\end{align*}
which is in contradiction to the requirement~\eqref{E:rho0}, since the exponent above is strictly negative.

Assume now that $\ga>\frac{4}{3}$ and that~\eqref{E:Constraints} holds. It is straightforward to check that
\begin{align*}
	n(\ga,\al) \in \Big(\frac{3\ga}{2},\infty\Big)\subset(2,\infty),
\end{align*}
and therefore by~\eqref{E:GDefAlt} we conclude
\begin{align*}
	G\underset{y\to0}\sim -y^2\om^2.
\end{align*}
Just as above, we conclude that
\begin{align*}
	q = (2-\ga)\al\Big(1-\frac{\al}{2}\Big)^2\frac{p}{y\rho\om}\underset{y\to0}\sim y^\beta,
\end{align*}
where
\begin{align*}
	\beta(\ga,\al) \in \Big(\frac{3\ga}{2}-1,\infty\Big)\subset(1,\infty).
\end{align*}
Thus, as $y\to0$ we conclude that
\begin{align*}
	\rho' = \frac{\rho(yh-q)}{G}\underset{y\to0}\sim -\frac{\rho(0)h(0)}{y\om(0)^2}+O_{y\to0}\big(y^{\beta-2}\big).
\end{align*}
Considering the $\om$ equation, we see that
\begin{align*}
	\om'\underset{y\to0}\sim \frac{4-3\ga+\al-3\om(0)}{y} + \frac{h(0)}{y\om(0)}+O_{y\to0}\big(y^{\beta-2}\big).
\end{align*}
In order for both $\rho$ and $\om$ to be $C^1$ at the origin, we therefore require the two constraints:
\begin{align*}
	h(\rho(0),\om(0)) &= 0, & \om(0) &= \frac{4-3\ga+\al}{3}.
\end{align*}
Solving the first of these equations subject to the second, this is easily seen to require
\begin{align*}
	\rho(0) = \frac{1}{6\pi}.
\end{align*}
Thus, we see the necessity of~\eqref{E:omega0}--\eqref{E:rho0} with $\rho_0=\frac{1}{6\pi}$, and so we have proved~\eqref{E:IC}.

Assuming now that $\rho$ and $\om$ are indeed $C^1$ at the origin, clearly then $h=O_{y\to0}(y)$, and so if $n\in(2,3)$ so that $\beta\in(1,2)$, then
\begin{align*}
	yh-q\underset{y\to0}\sim -(2-\ga)\al\Big(1-\frac{\al}{2}\Big)^2\rho_0^{\ga+\frac{n}{3}-1}\om_0^{\frac{n}{3}-1}y^\beta.
\end{align*}
In particular, from the ODE for $\rho$, we see that
\begin{align*}
	\rho\underset{y\to0}\sim y^{\beta-1}
\end{align*}
and hence $\rho$ cannot be $C^2$. Thus, for a smooth solution, we require $n\geq3$.

Expanding
\begin{align*}
	\frac{1}{G} = -\frac{1}{y^2\om^2}+\frac{\ga(1-\frac{\al}{2})^2\frac{p}{\rho}}{y^2\om^2(\ga(1-\frac{\al}{2})^2\frac{p}{\rho}-y^2\om^2)},
\end{align*}
it is easy to see that
\begin{align}
	\rho' = -\frac{\rho h}{y\om^2} + \frac{\rho q}{y^2\om^2} + O_{y\to0}\big(y^{n-2}\big),\label{E:rhoPrimeO}
\end{align}
where we have used that $|yh-q|=O_{y\to0}\big(y^2\big)$ (at least) due to $n\geq3$ and $h(0)=0$ along with smoothness of $\rho$ and $\om$.

Suppose now that $n\not\in\mathbb{Z}$, $n>3$. Then, due to smoothness and $h(0)=0$, clearly $-\frac{\rho h}{y\om^2}$ admits a regular Taylor expansion, while
\begin{align*}
	\frac{\rho q}{y^2\om^2}\underset{y\to0}\sim\frac{\rho_0q_0}{\om_0^2}y^{n-3}
\end{align*}
and where
\begin{align*}
	q_0 = \lim_{y\to0}\frac{q(y)}{y^{n-1}} > 0.
\end{align*}
As this term is strictly more singular than the final term in~\eqref{E:rhoPrimeO}, we deduce that $\rho'$ is not smooth, a contradiction to smoothness of $\rho$. Thus, we require $n\in\mathbb{Z}$, $n\geq4$.

Moreover, we may easily verify that $n$ must in fact be even. This follows from the simple observation that as $\rho$ and $\om$ are the radial representations of smooth scalar functions on $\mathbb{R}^3$, all of their odd derivatives at the origin (and hence odd Taylor coefficients) must vanish and they admit an even extension to the real line. 

The middle term in~\eqref{E:rhoPrimeO} is thus
\begin{align*}
	\frac{\rho q}{y^2\om^2} = (2-\ga)\al\Big(1-\frac{\al}{2}\Big)^2y^{n-3}\rho^{\ga+\frac{n}{3}}\om^{\frac{n}{3}-3},
\end{align*}
which is lower order in $y$ than the final term, but also even, leading to a contradiction. Thus $\rho'$ is a sum of (non-trivial) even and odd terms, a contradiction to the fact that $\rho$ is even.
\end{proof}

\subsection{Main Theorem and Methodology}

Lemma~\ref{L:SP} leads to an important conclusion that we should look for imploding solutions in the subset $\ga\in\big(\frac{4}{3},2\big)$ of the supercritical range of polytropic indices $\ga$ and with values of $\al$ such that $n(\ga,\al)$ is even. Condition~\eqref{E:nDef} therefore fixes the remaining scaling freedom in the problem and we next formulate the main theorem of this work.

\begin{theorem}[Main Theorem]\label{T:Main}
	Let $\ga\in\big(\frac{19}{12},\frac{11}{6}\big)$, $n(\ga,\al)=4$. Then there exists a global, real-analytic solution $(\rho,\om)$ of the self-similar Euler--Poisson system~\eqref{E:rhoPrime1}--\eqref{E:omegaPrime1} such that the only sonic point is at the origin, the flow is globally supersonic and the solution satisfies the monotonicity conditions for all $y>0$,
	\begin{align}
		\rho'(y) < 0 < \om'(y).
	\end{align}
	Moreover, there exists a constant $\rho_\infty>0$ such that, as $y\to\infty$, $(\rho,\om)$ satisfy the asymptotics:
	\begin{align}
		\rho(y) &= \rho_\infty y^{-\frac{2-\al}{2-\ga}} + O\Big(y^{-\frac{3}{2}\frac{2-\al}{2-\ga}}\Big), & \om(y) &= 2-\ga + O\Big(y^{-\frac{1}{2}\frac{2-\al}{2-\ga}}\Big).\label{E:rhoomegaAsymp}
	\end{align}
\end{theorem}

\begin{remark}
	With the stated asymptotics, it is simple to check the behaviour of the mass and energy associated with the imploding solution localised to a ball of fixed radius $R$. Using~\eqref{E:rhoomegaAsymp}, we see that:
	\begin{align*}
		M(t,R) &= \int_0^R\tilde\rho(t,r)r^2\, dr \underset{t\to0^-}\sim R^{4-3\ga+\al},\\
		E(t,R) &= \frac{1}{2}\int_0^R\bigg(\tilde\rho|\tilde{u}|^2+\frac{2}{\ga-1}\tilde{p}-|\nabla\varphi|^2\bigg)r^2\, dr \underset{t\to0^-}\sim R^{5-2\frac{2-\al}{2-\ga}},
	\end{align*}
	where we recall relation~\eqref{E:NablaPhi}. The above relations are consistent with the scaling invariance in the problem. In particular, the mass and energy remain locally finite in any finite region on approach to the collapse time.
\end{remark}

\begin{remark}[The band structure]\label{R:gammaBands}
	We see from the analysis of a formal Taylor series at the origin in Section~\ref{S:Taylor}, that smooth solutions may be constructed at $y=0$ for some $\al\in2\mathbb{Z}_+$ when $\ga>\frac{4}{3}$ and $\ga\neq \frac{4}{3}+\frac{1}{2m}$, $m\in\mathbb{Z}_+$. These critical values $\frac{4}{3}+\frac{1}{2m}$ separate bands of possible solutions across which the qualitative properties (that is, sign conditions for certain derivatives) of the solution change. For example, for $\ga>\frac{11}{6}=\frac{4}{3}+\frac{1}{2}$, the density is locally increasing near the origin, and so such a solution, if it existed globally, would describe a star with a \emph{well} of lower density at the centre, invaded by a higher density region from outside. We therefore choose to restrict ourselves to the more likely scenario in which the density takes its highest value at the origin, and so treat the range $\big(\frac{19}{12},\frac{11}{6}\big)$.
\end{remark}

As indicated by the remark above, to construct the supersonic collapse solutions to~\eqref{E:rhoPrime1}--\eqref{E:omegaPrime1}, we begin by constructing a formal Taylor series at the origin, which is also a sonic point for the flow. The substitution of a formal Taylor expansion into the ODE system yields a recurrence relation for the coefficients of $\rho$ and $\om$. Via combinatorial estimates akin to those developed in~\cite{GuoHadzicJang21B,GuoHadzicJangSchrecker22}, we obtain growth bounds on these coefficients which enable us to prove the convergence of the series locally around $y=0$, and hence obtain a local, analytic solution, which is supersonic for small $y>0$. The details are contained in Section~\ref{S:Taylor} and the local existence statement is formulated in Theorem~\ref{T:LWP}.

The second and key step is then to extend this solution to the full interval $y\in[0,\infty)$. As the obstructions to continuation of the solution are the blow-up of either $\rho$ or $\om$, the vanishing of $\rho$, or the occurrence of a second sonic point, to exclude each of these possibilities, we make a series of bootstrap bounds. Using these bootstrap assumptions, we use nonlinear invariances of the solutions along the flow to propagate these globally in $y$. The key structural difference between the solutions constructed in Theorem~\ref{T:Main} and those constructed in~\cite{GuoHadzicJang21B,GuoHadzicJangSchrecker22} is that the sonic line coincides with the centre of symmetry $\{(t,r)=(t,0)\}$. Therefore the solution is supported in the region external to the boundary of the backward sound cone. As a consequence, the sonic point is at $y=0$, which effectively removes a characteristic scale present in the isentropic setting~\cite{GuoHadzicJang21B,GuoHadzicJangSchrecker22,GuoHadzicJang23,MerleRaphaelRodnianskiSzeftel22B}, which is the size of the non-trivial opening angle of the backward sound cone. As a consequence of this, we formulate our bootstrap assumption working with $\rho(0)-\rho(y)$ and $\om(y)-\om(0)$, both of which are shown to remain strictly positive for all $y>0$. That this is so locally around $y=0$ is a consequence of the precise knowledge of the Taylor expansion of the solution shown in Section~\ref{S:Taylor}. The crux of the argument is to propagate a ``sandwich" bound that compares $\rho(0)-\rho(y)$ to $\om(y)-\om(0)$ from below and above in a suitable way (see bootstrap assumptions~\eqref{E:B4}--\eqref{E:B5}). A careful argument detailing the dynamic propagation of these bootstrap bounds is presented in Theorem~\ref{T:GE}. Such an argument consists of significant new ingredients in comparison to earlier constructions of self-similar imploding flows. The proof that $\lim_{y\to\infty}\rho(y)=0$ is relatively easy and it is the content of Proposition~\ref{P:rhoLim}, whereas the proof of the asymptotic relation $\lim_{y\to\infty}\om(y)=2-\ga$ requires several new ideas and subtle bootstrap arguments tailored to the problem. The latter statement is the content of Proposition~\ref{P:omegaLim}, where in one of the preceding lemmas we use interval arithmetic in a very soft way. Together, Theorem~\ref{T:GE}, Proposition~\ref{P:rhoLim}, Proposition~\ref{P:omegaLim} and Lemma~\ref{L:Asymptotics} constitute the proof of Theorem~\ref{T:Main}.

\medskip 

{\bf Acknowledgments.}
CA and MH are supported by the EPSRC Early Career Fellowship EP/S02218X/1. MS is supported by the EPSRC Postdoctoral Research Fellowship EP/W001888/1.

\section{Local Existence}
\label{S:Taylor}

As discussed above, we begin the construction of the supersonic collapse solutions of Theorem~\ref{T:Main} by first establishing the local existence of solutions close to the self-similar origin, $y=0$. In this section, we establish this local existence by constructing convergent Taylor series for the density and relative velocity that give analytic solutions of the ODE system~\eqref{E:rhoPrime1}--\eqref{E:omegaPrime1} on their domain of convergence. This is the content of the next theorem.

\begin{theorem}[Local existence]\label{T:LWP}
	Let $\ga\in\big(\frac{19}{12},\frac{11}{6}\big)$ and $n\in2\mathbb{Z}\cap[4,\infty)$, where we recall~\eqref{E:nDef}. Then there exists a $\nu>0$ and sequences of coefficients $(\rho_N)_{N=0}^\infty$ and $(\om_N)_{N=0}^\infty$ such that the series:
	\begin{align*}
		\rho(y) &= \sum_{N=0}^\infty \rho_Ny^N, & \om(y) &= \sum_{N=0}^\infty \om_Ny^N,
	\end{align*}
	converge absolutely on the interval $|y|<\nu$ and the functions $\rho(y)$ and $\om(y)$ are real analytic solutions of equations~\eqref{E:rhoPrime1}--\eqref{E:omegaPrime1}.
\end{theorem}

The proof of Theorem~\ref{T:LWP} is presented in Section~\ref{S:LWPProof}, in Section~\ref{SS:Taylor} we derive the recurrence relations satisfied by the Taylor coefficients of $(\rho,\om)$ near $y=0$ and in Section~\ref{SS:Induction} we explain the key inductive steps necessary for the series convergence argument.

\subsection{Taylor Coefficients}
\label{SS:Taylor}

For smooth solutions at the origin $y=0$, we recall from~\eqref{E:IC} that:
\begin{align*}
	\rho_0 &= \frac{1}{6\pi}, & \om_0 &= \frac{4-3\ga+\al}{3}.
\end{align*}
Moreover, from~\eqref{E:nDef} we obtain the identities:
\begin{align}
	\al &= \frac{n(3\ga-4)}{n+3\ga-6}, & \om_0 &= \frac{(3\ga-4)(2-\ga)}{n+3\ga-6}.\label{E:alpha(n)}
\end{align}
We therefore first establish recurrence relations to derive the higher order Taylor coefficients of any smooth solution from the lower order coefficients. We seek the Taylor series of smooth solutions in the form:
\begin{align*}
	\rho(y) &= \sum_{N=0}^{\infty}\rho_{N(n-2)}y^{N(n-2)}, & \om(y) &= \sum_{N=0}^{\infty}\om_{N(n-2)}y^{N(n-2)}.
\end{align*}
To facilitate the computation of the higher order coefficients, we define
\begin{align*}
	Q(y) := \frac{q}{y}(y) &= \sum_{N=0}^\infty Q_{N(n-2)}y^{N(n-2)}.
\end{align*}
Due to the structure of the solutions, we define a new variable
\begin{align*}
	x = y^{n-2}
\end{align*}
and express $\rho$, $\om$ and $Q$ as functions of $x$ through:
\begin{align}
	\bar{\rho}(x) &= \sum_{M=0}^\infty \bar{\rho}_{M}x^M, & \bar{\om}(x) &= \sum_{M=0}^\infty
	\bar{\om}_{M}x^M, & \bar{Q}(x) := \frac{q}{y}(x) &= \sum_{M=0}^\infty \bar{Q}_Mx^M.\label{E:TaylorxForm}
\end{align}
We note that $\bar{\rho}_0=\rho_0$, $\bar{\om}_0=\om_0$ and $\bar{\rho}_M=\bar{\om}_M=\bar{Q}_{M+1}=0$ for $M<0$. Moreover, we introduce the following notation to facilitate the manipulation of products of Taylor series:
\begin{equation}
\begin{aligned}
	(\bar{\rho}\bar{\om})_M &= \sum_{i+j=M}\bar{\rho}_i\bar{\om}_j, & &\qquad & (\bar{\om}^2)_M &= \sum_{i+j=M}\bar{\om}_i\bar{\om}_j,\\
	(\bar{\rho}^2\bar{\om})_M &= \sum_{i+j+k=M}\bar{\rho}_i\bar{\rho}_j\bar{\om}_k, & &\qquad & (\bar{\rho}\bar{\om}^2)_M &= \sum_{i+j+k=M}\bar{\rho}_i\bar{\om}_j\bar{\om}_k,\label{E:Round}
\end{aligned}
\end{equation}
and
\begin{equation}
\begin{aligned}
	[\bar{\rho}\bar{\om}]_M &= \sum_{\substack{i+j=M\\ i,j\neq M}}\bar{\rho}_i\bar{\om}_j, & [\bar{\om}^2]_M &= \sum_{\substack{i+j=M\\ i,j\neq M}}\bar{\om}_i\bar{\om}_j,\\
	[\bar{\rho}^2\bar{\om}]_M &= \sum_{\substack{i+j+k=M\\ i,j,k\neq M}}\bar{\rho}_i\bar{\rho}_j\bar{\om}_k, & [\bar{\rho}\bar{\om}^2]_M &= \sum_{\substack{i+j+k=M\\ i,j,k\neq M}}\bar{\rho}_i\bar{\om}_j\bar{\om}_k, &[\bar{\om}^3]_M &= \sum_{\substack{i+j+k=M\\ i,j,k\neq M}}\bar{\om}_i\bar{\om}_j\bar{\om}_k.\label{E:Square}
\end{aligned}
\end{equation}

We are now in a position to derive the main recurrence relation between the Taylor coefficients of the series~\eqref{E:TaylorxForm}.

\begin{lemma}\label{L:MatrixForm}
	Let $\ga\in\big(\frac{19}{12},\frac{11}{6}\big)$ and $n\in2\mathbb{Z}\cap[4,\infty)$. Suppose that $(\rho,\om,Q)$ defined by~\eqref{E:TaylorxForm} are local, smooth solutions of~\eqref{E:rhoPrime1}--\eqref{E:omegaPrime1}. Then, for any $M\geq1$, the coefficients $\bar{\rho}_M$ and $\bar{\om}_M$ are determined from the coefficients $(\bar{\rho}_0,\ldots,\bar{\rho}_{M-1})$ and $(\bar{\om}_0,\ldots,\bar{\om}_{M-1})$ by the relation
	\begin{align}
		A_M\bigg(\begin{array}{c}
			\bar{\rho}_M\\
			\bar{\om}_M
		\end{array}\bigg) =
		\bigg(\begin{array}{c}
			\mathcal{F}_M\\
			\mathcal{G}_M
		\end{array}\bigg),\label{E:Matrix}
	\end{align}
	where $A_M=\begin{pmatrix} A_M^{11} & A_M^{12}\\ A_M^{21} & A_M^{22}\end{pmatrix}$ has coefficients:
	\begin{equation}
	\begin{aligned}
		A_M^{11} &= M(n-2)\bar{\om}_0^2 - \frac{2}{9}\Big(1-\frac{\al}{2}\Big)^2,\\
		A_M^{12} &= 4\bar{\rho}_0\bar{\om}_0 + \Big(\ga-1-\frac{\al}{2}\Big)\bar{\rho}_0 - \frac{2}{9}\Big(1-\frac{\al}{2}\Big)^2\frac{\bar{\rho}_0}{\bar{\om}_0},\\
		A_M^{21} &= -\frac{2}{9}\Big(1-\frac{\al}{2}\Big)^2\frac{\bar{\om}_0}{\bar{\rho}_0},\\
		A_M^{22} &= -M(n-2)\bar{\om}_0^2 + \bar{\om}_0^2 + \Big(\ga-1-\frac{\al}{2}\Big)\bar{\om}_0 - \frac{2}{9}\Big(1-\frac{\al}{2}\Big)^2,\label{E:AMDef}
	\end{aligned}
	\end{equation}
	and:
	\begin{align}
		\mathcal{F}_M &= -2[\bar{\rho}\bar{\om}^2]_M - \Big(\ga-1-\frac{\al}{2}\Big)[\bar{\rho}\bar{\om}]_M + \frac{2}{9}\Big(1-\frac{\al}{2}\Big)^2\frac{[\bar{\rho}^2\bar{\om}]_M}{\bar{\rho}_0\bar{\om}_0}\label{E:FMDef}\\
		&+ \sum_{i+j+k=M}(i+1)\frac{(n-2)\ga}{(2-\ga)\al}\bar{\rho}_{i+1}\bar{\om}_{j-1}\bar{Q}_k + \sum_{i+j=M}\bar{\rho}_i\bar{Q}_j - \sum_{\substack{i+j=M\\ i< M-1}}(i+1)(n-2)\bar{\rho}_{i+1}(\bar{\om}^2)_{j-1},\notag\\
		\mathcal{G}_M &= -3\bar{\om}_0[\bar{\om}^2]_M + [\bar{\om}^3]_M - \Big(\ga-1-\frac{\al}{2}\Big)[\bar{\om}^2]_M + \frac{2}{9}\Big(1-\frac{\al}{2}\Big)^2\frac{[\bar{\rho}\bar{\om}^2]_M}{\bar{\rho}_0\bar{\om}_0}\label{E:GMDef}\\
		&- \sum_{i+j+k=M}(i+1)\frac{(n-2)\ga}{(2-\ga)\al}\bar{\om}_{i+1}\bar{\om}_{j-1}\bar{Q}_k + \sum_{i+j=M}\bigg(\bar{\om}_i\bar{Q}_j - \frac{3\ga}{(2-\ga)\al}\big((\bar{\om}^2)_i - \bar{\om}_0\bar{\om}_i\big)\bar{Q}_j\bigg)\notag\\
		&+ \sum_{\substack{i+j=M\\ i< M-1}}(i+1)(n-2)\bar{\om}_{i+1}(\bar{\om}^2)_{j-1}.\notag
	\end{align}
\end{lemma}

\begin{proof}
We begin by re-writing~\eqref{E:rhoPrime1}--\eqref{E:omegaPrime1} in the $x$ variable as:
\begin{align}
	\frac{(n-2)\ga}{(2-\ga)\al}x\bar{\om}\bar{\rho}'\bar{Q} + \bar{\rho}\bar{Q} - (n-2)x\bar{\om}^2\bar{\rho}' &= \bar{\rho}\bigg[2\bar{\om}^2 + \Big(\ga-1-\frac{\al}{2}\Big)(\bar{\om}+2-\ga) - \frac{2}{9}\Big(1-\frac{\al}{2}\Big)^2\frac{\bar{\rho}\bar{\om}}{\bar{\rho}_0\bar{\om}_0}\bigg],\label{E:rhoPrime2}\\
	-\frac{(n-2)\ga}{(2-\ga)\al}x\bar{\om}\bar{\om}'\bar{Q} + \bar{\om}\bar{Q} + (n-2)x\bar{\om}^2\bar{\om}' &= \bar{\om}\bigg[2\bar{\om}^2 + \Big(\ga-1-\frac{\al}{2}\Big)(\bar{\om}+2-\ga) - \frac{2}{9}\Big(1-\frac{\al}{2}\Big)^2\frac{\bar{\rho}\bar{\om}}{\bar{\rho}_0\bar{\om}_0}\bigg]\notag\\
	&- 3(\bar{\om}_0-\bar{\om})\bigg(\frac{\ga}{(2-\ga)\al}\bar{\om}\bar{Q} - \bar{\om}^2\bigg).\label{E:omegaPrime2}
\end{align}
Substituting~\eqref{E:TaylorxForm} into~\eqref{E:rhoPrime2} yields
\begin{multline*}
	\sum_{M=0}^{\infty}\bigg[\sum_{i+j+k=M}(i+1)\frac{(n-2)\ga}{(2-\ga)\al}\bar{\rho}_{i+1}\bar{\om}_{j-1}\bar{Q}_k + \sum_{i+j=M}\Big(\bar{\rho}_i\bar{Q}_j - (i+1)(n-2)\bar{\rho}_{i+1}(\bar{\om}^2)_{j-1}\Big)\bigg]x^M\\
	= \sum_{M=0}^{\infty}\bigg[2(\bar{\rho}\bar{\om}^2)_M + \Big(\ga-1-\frac{\al}{2}\Big)\big((\bar{\rho}\bar{\om})_M + (2-\ga)\bar{\rho}_M\big) - \frac{2}{9}\Big(1-\frac{\al}{2}\Big)^2\frac{(\bar{\rho}^2\bar{\om})_M}{\bar{\rho}_0\bar{\om}_0}\bigg]x^M,
\end{multline*}
where we have used the identities:
\begin{align*}
	f'(x) &= \sum_{M=0}^\infty(M+1)f_{M+1}x^M, & xf(x) &= \sum_{M=0}^\infty f_Mx^{M+1}.
\end{align*}
Comparing coefficients of equal order, we thus obtain 
\begin{multline}
	\sum_{i+j+k=M}(i+1)\frac{(n-2)\ga}{(2-\ga)\al}\bar{\rho}_{i+1}\bar{\om}_{j-1}\bar{Q}_k + \sum_{i+j=M}\Big(\bar{\rho}_i\bar{Q}_j - (i+1)(n-2)\bar{\rho}_{i+1}(\bar{\om}^2)_{j-1}\Big)\\
	= 2(\bar{\rho}\bar{\om}^2)_M + \Big(\ga-1-\frac{\al}{2}\Big)\big((\bar{\rho}\bar{\om})_M + (2-\ga)\bar{\rho}_M\big) - \frac{2}{9}\Big(1-\frac{\al}{2}\Big)^2\frac{(\bar{\rho}^2\bar{\om})_M}{\bar{\rho}_0\bar{\om}_0}.\label{E:rhoPrime3}
\end{multline}
Similarly, substituting~\eqref{E:TaylorxForm} into~\eqref{E:omegaPrime2}, we group terms and obtain
\begin{align}
	&- \sum_{i+j+k=M}(i+1)\frac{(n-2)\ga}{(2-\ga)\al}\bar{\om}_{i+1}\bar{\om}_{j-1}\bar{Q}_k\notag\\
	&+ \sum_{i+j=M}\bigg(\bar{\om}_i\bar{Q}_j + (i+1)(n-2)\bar{\om}_{i+1}(\bar{\om}^2)_{j-1} - \frac{3\ga}{(2-\ga)\al}\big((\bar{\om}^2)_i - \bar{\om}_0\bar{\om}_i\big)\bar{Q}_j\bigg)\label{E:omegaPrime3}\\
	&= 3\bar{\om}_0(\bar{\om}^2)_M - (\bar{\om}^3)_M + \Big(\ga-1-\frac{\al}{2}\Big)\big((\bar{\om}^2)_M+(2-\ga)\bar{\om}_M\big) - \frac{2}{9}\Big(1-\frac{\al}{2}\Big)^2\frac{(\bar{\rho}\bar{\om}^2)_M}{\bar{\rho}_0\bar{\om}_0}.\notag
\end{align}
In order to identify the highest order coefficients, we observe from~\eqref{E:IC} the simple identity
\begin{align}
	h(\rho_0,\om_0) = 0.\label{E:h(0)=0}
\end{align}
Now, for $M>0$, we isolate the highest order coefficients of~\eqref{E:rhoPrime3} (recalling $\bar{Q}_0=0$ and~\eqref{E:h(0)=0}) and rearrange to obtain
\begin{equation}
\begin{aligned}
	&M(n-2)\bar{\om}_0^2\bar{\rho}_M + 4\bar{\rho}_0\bar{\om}_0\bar{\om}_M + \Big(\ga-1-\frac{\al}{2}\Big)\bar{\rho}_0\bar{\om}_M - \frac{2}{9}\Big(1-\frac{\al}{2}\Big)^2\frac{\bar{\rho}_0\bar{\om}_0\bar{\rho}_M + \bar{\rho}_0^2\bar{\om}_M}{\bar{\rho}_0\bar{\om}_0}\\
	&= -2[\bar{\rho}\bar{\om}^2]_M - \Big(\ga-1-\frac{\al}{2}\Big)[\bar{\rho}\bar{\om}]_M + \frac{2}{9}\Big(1-\frac{\al}{2}\Big)^2\frac{[\bar{\rho}^2\bar{\om}]_M}{\bar{\rho}_0\bar{\om}_0}\\
	&+ \sum_{i+j+k=M}(i+1)\frac{(n-2)\ga}{(2-\ga)\al}\bar{\rho}_{i+1}\bar{\om}_{j-1}\bar{Q}_k + \sum_{i+j=M}\bar{\rho}_i\bar{Q}_j - \sum_{\substack{i+j=M\\ i< M-1}}(i+1)(n-2)\bar{\rho}_{i+1}(\bar{\om}^2)_{j-1}.\label{E:rhoPrime4}
\end{aligned}
\end{equation}
Similarly, from~\eqref{E:omegaPrime3}
\begin{equation}
\begin{aligned}
	&-M(n-2)\bar{\om}_0^2\bar{\om}_M + \bar{\om}_0^2\bar{\om}_M + \Big(\ga-1-\frac{\al}{2}\Big)\bar{\om}_0\bar{\om}_M - \frac{2}{9}\Big(1-\frac{\al}{2}\Big)^2\frac{\bar{\rho}_0\bar{\om}_0\bar{\om}_M+\bar{\om}_0^2\bar{\rho}_M}{\bar{\rho}_0\bar{\om}_0}\\
	&= - 3\bar{\om}_0[\bar{\om}^2]_M + [\bar{\om}^3]_M - \Big(\ga-1-\frac{\al}{2}\Big)[\bar{\om}^2]_M + \frac{2}{9}\Big(1-\frac{\al}{2}\Big)^2\frac{[\bar{\rho}\bar{\om}^2]_M}{\bar{\rho}_0\bar{\om}_0}\\
	&- \sum_{i+j+k=M}(i+1)\frac{(n-2)\ga}{(2-\ga)\al}\bar{\om}_{i+1}\bar{\om}_{j-1}\bar{Q}_k + \sum_{i+j=M}\bigg(\bar{\om}_i\bar{Q}_j - \frac{3\ga}{(2-\ga)\al}\big((\bar{\om}^2)_i - \bar{\om}_0\bar{\om}_i\big)\bar{Q}_j\bigg)\\
	&+ \sum_{\substack{i+j=M\\ i< M-1}}(i+1)(n-2)\bar{\om}_{i+1}(\bar{\om}^2)_{j-1}.\label{E:omegaPrime4}
\end{aligned}
\end{equation}
Finally, letting $\mathcal{F}_M$ and $\mathcal{G}_M$ denote the right hand sides of~\eqref{E:rhoPrime4} and~\eqref{E:omegaPrime4} respectively, we obtain the claimed identity~\eqref{E:Matrix}.
\end{proof}

With these relations, we are able to collect the explicit forms of the first coefficients.

\begin{lemma}\label{L:Taylor}
	Let $\ga\in\big(\frac{19}{12},\frac{11}{6}\big)$ and $n\in2\mathbb{Z}\cap[4,\infty)$. Suppose that $(\bar{\rho},\bar{\om},\bar{Q})$ defined by~\eqref{E:TaylorxForm} are local, smooth solutions of~\eqref{E:rhoPrime1}--\eqref{E:omegaPrime1}. Then the zeroth- and first-order coefficients are given explicitly by:
	\begin{equation}
	\begin{aligned}
		\bar{\rho}_0 &= \frac{1}{6\pi}, & &\qquad & \bar{\om}_0 &= \frac{(3\ga-4)(2-\ga)}{n+3\ga-6},\\
		\bar{\rho}_1 &= -\frac{3(n+1)n}{2(\ga-1)(11-6\ga)}p_0, & &\qquad & \bar{\om}_1 &= \frac{3n(n-2)}{2(\ga-1)(11-6\ga)}\frac{\bar{\om}_0}{\bar{\rho}_0}p_0,\label{E:rho1omega1}
	\end{aligned}
	\end{equation}
	where we define
	\begin{align*}
		p_0 = \bar{\rho}_0^{\ga}(\bar{\rho}_0\bar{\om}_0)^{\frac{n}{3}}.
	\end{align*}
	In particular, we have the simple relation
	\begin{align}
		\bar{\rho}_1 = -\frac{n+1}{n-2}\frac{\bar{\rho}_0}{\bar{\om}_0}\bar{\om}_1.\label{E:rho1omega1Alt}
	\end{align}
	Moreover, the second-order coefficients satisfy
	\begin{align}
		\bar{\rho}_2 &= \frac{9n(2n-1)S(n,\ga)}{4(\ga-1)^2(6\ga-7)(12\ga-19)(11-6\ga)^2}{\frac{p_0^2}{\bar{\rho}_0}},\label{E:rho2}\\
		\bar{\om}_2 &= -\frac{9n(n-2)T(n,\ga)}{4(\ga-1)^2(6\ga-7)(12\ga-19)(11-6\ga)^2}\frac{\bar{\om}_0}{\bar{\rho}_0}{\frac{p_0^2}{\bar{\rho}_0}},\label{E:omega2}
	\end{align}
	where
	\begin{align*}
		S(n,\ga) &= 6\ga(\ga-1)(11-6\ga) - 3n(12\ga^2 - 34\ga + 21) + n^2(36\ga^3 - 30\ga^2 - 132\ga + 133),\\
		T(n,\ga) &= 12\ga(\ga-1)(11-6\ga) - n(144\ga^2 - 402\ga + 259) + n^2(72\ga^3 - 132\ga^2 - 66\ga + 133).
	\end{align*}
\end{lemma}

\begin{proof}
The expressions for $\bar{\rho}_0$ and $\bar{\om}_0$ follow directly from~\eqref{E:IC}. To obtain $\bar{\rho}_1$ and $\bar{\om}_1$, we first note that these expressions directly imply
\begin{align}
	\bar{Q}_1 &= (2-\ga)\al\Big(1-\frac{\al}{2}\Big)^2\bar{\rho}_0^\ga(\bar{\rho}_0\bar{\om}_0)^{\frac{n}{3}-1} = n\Big(1-\frac{\al}{2}\Big)^2\frac{p_0}{\bar{\rho}_0}\label{E:Q1}
\end{align}
Moreover, we see easily from~\eqref{E:FMDef}--\eqref{E:GMDef} that:
\begin{align*}
	\mathcal{F}_1 &= \Big(\ga-1-\frac{\al}{2}\Big)(2-\ga)\bar{\rho}_0 + \bar{\rho}_0\bar{Q}_1, & \mathcal{G}_1 &= \Big(\ga-1-\frac{\al}{2}\Big)(2-\ga)\bar{\om}_0 + \bar{\om}_0\bar{Q}_1.
\end{align*}
Thus,~\eqref{E:Matrix} for $M=1$ reduces to
\begin{multline}
	(n-2)\bar{\om}_0^2\bar{\rho}_1 + 4\bar{\rho}_0\bar{\om}_0\bar{\om}_1 + \Big(\ga-1-\frac{\al}{2}\Big)\bar{\rho}_0\bar{\om}_1 - \frac{2}{9}\Big(1-\frac{\al}{2}\Big)^2\bigg(\bar{\rho}_1+\frac{\bar{\rho}_0}{\bar{\om}_0}\bar{\om}_1\bigg)\\ = \Big(\ga-1-\frac{\al}{2}\Big)(2-\ga)\bar{\rho}_0 + \bar{\rho}_0\bar{Q}_1\label{E:rho1Prime}
\end{multline}
and
\begin{multline}
	-(n-3)\bar{\om}_0^2\bar{\om}_1 + \Big(\ga-1-\frac{\al}{2}\Big)\bar{\om}_0\bar{\om}_1 - \frac{2}{9}\Big(1-\frac{\al}{2}\Big)^2\bigg(\frac{\bar{\om}_0}{\bar{\rho}_0}\bar{\rho}_1+\bar{\om}_1\bigg)\\ = \Big(\ga-1-\frac{\al}{2}\Big)(2-\ga)\bar{\om}_0 + \bar{\om}_0\bar{Q}_1\label{E:omega1Prime}
\end{multline}
respectively. Multiplying the second equation by $\bar{\rho}_0$, dividing by $\bar{\om}_0$ and substituting this into the first equation, we thus obtain~\eqref{E:rho1omega1Alt}. Substituting~\eqref{E:rho1omega1Alt} back into~\eqref{E:rho1Prime} then yields~\eqref{E:rho1omega1}. For $M=2$, we find that:
\begin{align*}
	\mathcal{F}_2 &= -2(\bar{\rho}_0\bar{\om}_1^2 + 2\bar{\om}_0\bar{\rho}_1\bar{\om}_1) - \Big(\ga-1-\frac{\al}{2}\Big)\bar{\rho}_1\bar{\om}_1 + \frac{2}{9}\Big(1-\frac{\al}{2}\Big)^2\frac{\bar{\om}_0\bar{\rho}_1^2+2\bar{\rho}_0\bar{\rho}_1\bar{\om}_1}{\bar{\rho}_0\bar{\om}_0}\\
	&+ \frac{(n-2)\ga}{(2-\ga)\al}\bar{\om}_0\bar{\rho}_1\bar{Q}_1 + \bar{\rho}_1\bar{Q}_1 + \bar{\rho}_0\bar{Q}_2 - 2(n-2)\bar{\om}_0\bar{\rho}_1\bar{\om}_1,\\
	\mathcal{G}_2 &= -\Big(\ga-1-\frac{\al}{2}\Big)\bar{\om}_1^2 + \frac{2}{9}\Big(1-\frac{\al}{2}\Big)^2\frac{\bar{\rho}_0\bar{\om}_1^2 + 2\bar{\om}_0\bar{\rho}_1\bar{\om}_1}{\bar{\rho}_0\bar{\om}_0}\\
	&- \frac{(n+1)\ga}{(2-\ga)\al}\bar{\om}_0\bar{\om}_1\bar{Q}_1 + \bar{\om}_1\bar{Q}_1 + \bar{\om}_0\bar{Q}_2 + 2(n-2)\bar{\om}_0\bar{\om}_1^2,
\end{align*}
which can be simplified to:
\begin{align*}
	\mathcal{F}_2 &= -2(\bar{\rho}_0\bar{\om}_1^2 + 2\bar{\om}_0\bar{\rho}_1\bar{\om}_1) - \Big(\ga-1-\frac{\al}{2}\Big)\bar{\rho}_1\bar{\om}_1 + \frac{2}{9}\Big(1-\frac{\al}{2}\Big)^2\bigg(\frac{\bar{\rho}_1}{\bar{\rho}_0}+2\frac{\bar{\om}_1}{\bar{\om}_0}\bigg)\bar{\rho}_1\\
	&+ \bigg(1+\frac{(n-2)\ga}{n}\bigg)\bar{\rho}_1\bar{Q}_1 + \bar{\rho}_0\bar{Q}_2 - 2(n-2)\bar{\om}_0\bar{\rho}_1\bar{\om}_1,\\
	\mathcal{G}_2 &= -\Big(\ga-1-\frac{\al}{2}\Big)\bar{\om}_1^2 + \frac{2}{9}\Big(1-\frac{\al}{2}\Big)^2\bigg(\frac{\bar{\om}_1}{\bar{\om}_0}+2\frac{\bar{\rho}_1}{\bar{\rho}_0}\bigg)\bar{\om}_1\\
	&+ \bigg(1-\frac{(n+1)\ga}{n}\bigg)\bar{\om}_1\bar{Q}_1 + \bar{\om}_0\bar{Q}_2 + 2(n-2)\bar{\om}_0\bar{\om}_1^2.
\end{align*}
Now $\bar{Q}_1$ is given by~\eqref{E:Q1} and $\bar{Q}_2$ is the second-order terms of
\begin{align*}
	\bar{Q} &= n\bar{\om}_0\Big(1-\frac{\al}{2}\Big)^2x\rho^\ga(\rho\om)^{\frac{n}{3}-1}\\
	&= n\bar{\om}_0\Big(1-\frac{\al}{2}\Big)^2x(\bar{\rho}_0+\bar{\rho}_1x)^{\ga}\big((\bar{\rho}_0+\bar{\rho}_1x)(\bar{\om}_0+\bar{\om}_1x)\big)^{\frac{n}{3}-1} + O_{x\to0}(x^3)\\
	&= n\Big(1-\frac{\al}{2}\Big)^2\frac{p_0}{\rho_0}x\bigg(1+\ga\frac{\bar{\rho}_1}{\bar{\rho}_0}x\bigg)\bigg(1+\bigg(\frac{n}{3}-1\bigg)\bigg(\frac{\bar{\rho}_1}{\rho_0}+\frac{\bar{\om}_1}{\om_0}\bigg)x\bigg) + O_{x\to0}(x^3)\\
	&= \bar{Q}_1x + \bar{Q}_1\bigg(\ga\frac{\bar{\rho}_1}{\bar{\rho}_0}+\bigg(\frac{n}{3}-1\bigg)\bigg(\frac{\bar{\rho}_1}{\rho_0}+\frac{\bar{\om}_1}{\om_0}\bigg)\bigg)x^2 +O_{x\to0}(x^3),
\end{align*}
that is,
\begin{align*}
	\bar{Q}_2 = \bar{Q}_1\bigg(\ga\frac{\bar{\rho}_1}{\bar{\rho}_0}+\bigg(\frac{n}{3}-1\bigg)\bigg(\frac{\bar{\rho}_1}{\rho_0}+\frac{\bar{\om}_1}{\om_0}\bigg)\bigg).
\end{align*}
We have from~\eqref{E:Matrix} that:
\begin{align*}
	\bar{\rho}_2 &= \frac{1}{\det A_2}\big(A^{22}_2\mathcal{F}_2-A^{12}_2\mathcal{G}_2\big), &
	\bar{\om}_2 &= \frac{1}{\det A_2}\big(A^{11}_2\mathcal{G}_2-A^{21}_2\mathcal{F}_2\big),
\end{align*}
and where from the proof of Lemma~\ref{L:rhoMomegaMBounds}
\begin{align*}
	\det A_2 = -\frac{(n-2)^2(6\ga-7)(12\ga-19)}{2(3\ga-4)^2}\bar{\om}_0^4.
\end{align*}
Noting that
\begin{align*}
	1-\frac{\al}{2} = \frac{3(n-2)}{2(3\ga-4)}\bar{\om}_0,
\end{align*}
a tedious calculation reduces~\eqref{E:Matrix} to~\eqref{E:rho2} and~\eqref{E:omega2}.
\end{proof}

In order to solve~\eqref{E:Matrix} for the coefficients of order $M\geq2$, we require the invertibility of the matrix $A_M$ for $M\geq2$. This is the content of the next lemma. 

\begin{lemma}\label{L:rhoMomegaMBounds}
	Let $\ga\in\big(\frac{19}{12},\frac{11}{6}\big)$ and $n\in2\mathbb{Z}\cap[4,\infty)$. Then the matrix $A_M$ is invertible and there exists a constant $D=D(n,\ga)>0$, independent of $M$, such that for all integers $M>1$, we have the bounds:
	\begin{align}
		|\bar{\rho}_M| &\leq \frac{D}{M}\Big(|\mathcal{F}_M| + \frac{1}{M}|\mathcal{G}_M|\Big),\label{E:rhoMBound}\\
		|\bar{\om}_M| &\leq \frac{D}{M}\Big(|\mathcal{G}_M| + \frac{1}{M}|\mathcal{F}_M|\Big)\label{E:omegaMBound}.
	\end{align}
\end{lemma}

\begin{proof}
We first observe from~\eqref{E:AMDef} that the matrix $A_M$ has determinant
\begin{align*}
	\det A_M = -\bar{\om}_0^2\bigg(M^2(n-2)^2\bar{\om}_0^2 - M(n-2)\Big(\bar{\om}_0 + \Big(\ga-1-\frac{\al}{2}\Big)\Big)\bar{\om}_0 - \frac{6}{9}\Big(1-\frac{\al}{2}\Big)^2\bigg).
\end{align*}
Making explicit the dependence of $\bar{\om}_0$ and $\al$ on $n$ and $\ga$ from~\eqref{E:IC}, this yields
\begin{align}
	\det A_M = -\frac{(n-2)^2(3\ga-4)^2(2-\ga)^4}{2(n+3\ga-6)^4}\big(M(3\ga-4)+1\big)\big(2M(3\ga-4)-3\big)\label{E:det}.
\end{align}
Thus, providing $M>1$ and $\ga\in\big(\frac{19}{12},\frac{11}{6}\big)$, we have
\begin{align*}
	\frac{(n-2)^2(12\ga-19)}{2(3\ga-4)}\bar{\om}_0^4M^2 \leq |\det A_M| \leq (n-2)^2\bar{\om}_0^4M^2.
\end{align*}
Moreover, from~\eqref{E:AMDef}, each matrix coefficient $A_M^{ij}$, with $i,j\in\{1,2\}$, is easily seen to satisfy
\begin{align*}
	\big|A^{ij}_M\big|\leq CM
\end{align*}
for some $C(n,\ga)>0$. Now, from~\eqref{E:Matrix}, we deduce:
\begin{align*}
	\bar{\rho}_M &= \frac{1}{\det A_M}\big(A^{22}_M\mathcal{F}_M-A^{12}_M\mathcal{G}_M\big), &
	\bar{\om}_M &= \frac{1}{\det A_M}\big(A^{11}_M\mathcal{G}_M-A^{21}_M\mathcal{F}_M\big).
\end{align*}
and so we conclude that there exists a constant $D>0$ such that~\eqref{E:rhoMBound} and~\eqref{E:omegaMBound} hold.
\end{proof}

\begin{remark}\label{R:det}
	(i) From the expression~\eqref{E:det}, we see that for $n\geq4$ and $\ga\in\big(\frac{4}{3},2\big)$, $\det A_M$ vanishes whenever $\ga=\frac{4}{3}+\frac{1}{2M}$. In fact, each of these thresholds marks a qualitative change in the behaviour of the solutions with, for example, local solutions around $y=0$ having density monotone increasing for $\ga>\frac{11}{6}$ and monotone decreasing for $\ga\in\big(\frac{4}{3},\frac{11}{6}\big)$. The sequence of values $\ga_m=\frac{4}{3}+\frac{1}{2m}$ therefore demarcates a series of bands $(\ga_{m+1},\ga_m)$ with qualitatively different behaviour. The solutions constructed in Theorem~\ref{T:Main} are those corresponding to the first such band.\\
	(ii) It follows immediately from~\eqref{E:FMDef}--\eqref{E:GMDef} and~\eqref{E:rhoMBound}--\eqref{E:omegaMBound} that, for $\ga\in\big(\frac{19}{12},\frac{11}{6}\big)$, given $M_0\geq3$ and $\beta\in(1,2)$, we may find $C_0=C_0(M_0,\beta,n,\ga)\geq1$ such that for $1<M\leq M_0$, we have the bounds:
	\begin{align}
		|\rho_M| &\leq \frac{C^{M-\beta}}{M^3}, & |\om_M| &\leq \frac{C^{M-\beta}}{M^3}
	\end{align}
	for all $C\geq C_0$.
\end{remark}

\subsection{Estimates on Taylor Coefficients}\label{SS:Induction}

In order to use~\eqref{E:Matrix} to establish quantitative bounds on the higher order Taylor coefficients, we require some preliminary results to control key nonlinear quantities arising in $\mathcal{F}_M$ and $\mathcal{G}_M$.

\begin{lemma}\label{L:Quadratic}
	Let $M\geq2$ and suppose there exist constants $C\geq1$ and $\beta\in(1,2)$ such that, for $m=2,\ldots,M$, we have:
	\begin{align*}
		|\bar{\rho}_m| &\leq \frac{C^{m-\beta}}{m^3}, & |\bar{\om}_m| &\leq \frac{C^{m-\beta}}{m^3}.
	\end{align*}
	Then there exists a constant $D>0$, independent of $C$ and $\beta$, such that 
	\begin{align}
		\big|(\bar{\rho}\bar{\om})_{M}\big| + \big|(\bar{\om}^2)_{M}\big| &\leq D\frac{C^{M-\beta}}{M^3}.\label{E:Quadratic}
	\end{align}
\end{lemma}

\begin{proof}
We show~\eqref{E:Quadratic} by proving the bound for $(\bar{\om}^2)_{M}$, as the other estimate follows similarly. Note first that from Remark~\ref{R:det}(ii), taking $M_0=4$, we may always choose $D$ such that~\eqref{E:Quadratic} holds for $M=2,3$. In the case $M\geq4$, we split the sum and apply the assumption to bound
\begin{align*}
	\big|(\bar{\om}^2)_{M}\big| &\leq \sum_{i=0}^{M}|\bar{\om}_i||\bar{\om}_{M-i}|\\
	&= 2|\bar{\om}_0||\bar{\om}_{M}|+2|\bar{\om}_1||\bar{\om}_{M-1}|+\sum_{i=2}^{M-2}|\bar{\om}_i||\bar{\om}_{M-i}|\\
	&\leq D_1\bigg(\frac{C^{M-\beta}}{M^3}+\frac{C^{M-1-\beta}}{(M-1)^3}\bigg)+\sum_{i=2}^{M-2}\frac{C^{i-\beta}}{i^3}\frac{C^{M-i-\beta}}{(M-i)^3}\\
	&\leq D_1\frac{C^{M-\beta}}{M^3}\bigg(1+\frac{C^{-1}M^3}{(M-1)^3}+\sum_{i=2}^{M-2}\frac{C^{-\beta}}{i^3(M-i)^3}\bigg)\\
	&\leq D\frac{C^{M-\beta}}{M^3}
\end{align*}
as claimed, where the constant $D_1>0$ may be different in each line and where in the last inequality we have used \cite[Lemma B.1]{GuoHadzicJangSchrecker22} to estimate
\begin{align*}
	\sum_{i=2}^{M-2}\frac{1}{i^3(M-i)^3} \leq \frac{D_2}{M^3}
\end{align*}
for some constant $D_2>0$.
\end{proof}

\begin{lemma}\label{L:BracketBounds}
	Let $M>2$ and suppose there exist constants $C\geq1$ and $\beta\in(1,2)$ such that, for $m=2,\ldots,M-1$, we have:
	\begin{align*}
		|\bar{\rho}_m| &\leq \frac{C^{m-\beta}}{m^3}, & |\bar{\om}_m| &\leq \frac{C^{m-\beta}}{m^3}.
	\end{align*}
	Then there exists a constant $D>0$, independent of $C$ and $\beta$, such that:
	\begin{align}
		\big|[\bar{\rho}\bar{\om}]_M\big| + \big|[\bar{\om}^2]_M\big| &\leq D\frac{C^{M-1-\beta}}{M^3},\label{E:rhoMBracket}\\
		\big|[\bar{\rho}^2\bar{\om}]_M\big| + \big|\bar{\rho}\bar{\om}^2]_M\big| + \big|[\bar{\om}^3]_M\big| &\leq D\frac{C^{M-1-\beta}}{M^3},\label{E:omegaMBracket}
	\end{align}
	where we recall the notation~\eqref{E:Square}.
\end{lemma}

\begin{proof}
We first we prove~\eqref{E:rhoMBracket}. We estimate
\begin{align*}
	\big|[\bar{\rho}\bar{\om}]_M\big| &\leq \sum_{\substack{i+j=M\\ i,j\neq M}}|\bar{\rho}_i||\bar{\om}_j| = |\bar{\rho}_1||\bar{\om}_{M-1}| + |\bar{\om}_1||\bar{\rho}_{M-1}| + \sum_{\substack{i+j=M\\ i,j>1}}|\bar{\rho}_i||\bar{\om}_j|\\
	&\leq (|\bar{\rho}_1|+|\bar{\om}_1|)\frac{C^{M-1-\beta}}{(M-1)^3} + \sum_{\substack{i+j=M\\ i,j>1}}\frac{C^{i-\beta}}{i^3}\frac{C^{j-\beta}}{j^3}\\
	&= (|\bar{\rho}_1|+|\bar{\om}_1|)\frac{C^{M-1-\beta}}{(M-1)^3} + \sum_{\substack{i+j=M\\ i,j>1}}\frac{C^{M-2\beta}}{i^3j^3}\\
	&\leq (|\bar{\rho}_1|+|\bar{\om}_1|)\frac{C^{M-1-\beta}}{(M-1)^3} + D_1\frac{C^{M-2\beta}}{M^3}
\end{align*}
for some constant $D_1>0$, where we have applied \cite[Lemma B.1]{GuoHadzicJangSchrecker22} to bound
\begin{align*}
	\sum_{i+j=M}\frac{1}{i^3j^3} \leq \frac{D_1}{M^3}.
\end{align*}
The bounds for the remaining terms in~\eqref{E:rhoMBracket} follow similarly.

To estimate the cubic terms in~\eqref{E:omegaMBracket}, we first treat the case $M=3$. 
Since $\bar{\rho}_0$, $\bar{\om}_0$, $\bar{\rho}_1$ and $\bar{\om}_1$ are each bounded by a constant depending only on $n$ and $\ga$, we have
\begin{align*}
	\big|[\bar{\rho}\bar{\om}^2]_3\big| &\leq \sum_{\substack{i+j+k=3\\ i,j,k\neq 3}}|\bar{\rho}_i||\bar{\om}_j||\bar{\om}_k| = |\bar{\rho}_1||\bar{\om}_1|^2 + 2\bar{\rho}_0|\bar{\om}_1||\bar{\om}_2| + 2\bar{\om}_0|\bar{\rho}_1||\bar{\om}_2| + 2\bar{\om}_0|\bar{\om}_1||\bar{\rho}_2|\\
	&\leq |\bar{\rho}_1||\bar{\om}_1|^2 + 2(\bar{\rho}_0|\bar{\om}_1| + \bar{\om}_0|\bar{\rho}_1| + \bar{\om}_0|\bar{\om}_1|)\frac{C^{2-\beta}}{2^3}\\
	&\leq D_2\frac{C^{3-1-\beta}}{3^3}
\end{align*}
for some constant $D_2>0$. Similar estimates hold for $\big|[\bar{\rho}^2\bar{\om}]_3\big|$ and $\big|[\bar{\om}^3]_3\big|$.

When $M\geq4$, we again separate terms and obtain
\begin{align*}
	&\big|[\bar{\rho}\bar{\om}^2]_M\big| \leq \sum_{\substack{i+j+k=M\\ i,j,k\neq M}}|\bar{\rho}_i||\bar{\om}_j||\bar{\om}_k|\\
	&= (2\bar{\rho}_0|\bar{\om}_1|+|\bar{\rho}_1||\bar{\om}_0|)|\bar{\om}_{M-1}| + 2|\bar{\om}_0||\bar
	\bar{\om}_1||\bar{\rho}_{M-1}| + 2|\bar{\rho}_1||\bar{\om}_1||\bar{\om}_{M-2}| + |\bar{\om}_1|^2|\bar{\rho}_{M-2}|\\
	&+ \bar{\rho}_0\sum_{\substack{i+j=M\\ i,j>1}}|\bar{\om}_i||\bar{\om}_j|+|\bar{\om}_0|\sum_{\substack{i+j=M\\ i,j>1}}|\bar{\rho}_i||\bar{\om}_j|+|\bar{\rho}_1|\sum_{\substack{i+j = M-1\\ i,j>1}}|\bar{\om}_i||\bar{\om}_j|+|\bar{\om}_1|\sum_{\substack{i+j=M-1\\ i,j>1}}|\bar{\rho}_i||\bar{\om}_j|\\
	&+ \sum_{\substack{i+j+k=M\\ i,j,k>1}}|\bar{\rho}_i||\bar{\om}_j||\bar{\om}_k|\\
	&\leq (2\bar{\rho}_0|\bar{\om}_1|+|\bar{\rho}_1||\bar{\om}_0|+2|\bar{\om}_0||\bar{\om}_1|)\frac{C^{M-1-\beta}}{(M-1)^3} + \big(2|\bar{\rho}_1||\bar{\om}_1|+|\bar{\om}_1|^2\big)\frac{C^{M-2-\beta}}{(M-2)^3}\\
	&+ (|\bar{\rho}_0|+|\bar{\om}_0|)\sum_{\substack{i+j=M\\ i,j>1}}\frac{C^{i-\beta}}{i^3}\frac{C^{j-\beta}}{j^3}+(|\bar{\rho}_1|+|\bar{\om}_1|)\sum_{\substack{i+j=M-1\\ i,j>1}}\frac{C^{i-\beta}}{i^3}\frac{C^{j-\beta}}{j^3}\\
	&+ \sum_{\substack{i+j+k=M\\ i,j,k>1}}\frac{C^{i-\beta}}{i^3}\frac{C^{j-\beta}}{j^3}\frac{C^{k-\beta}}{k^3}\\
	&\leq D_3\bigg(\frac{C^{M-1-\beta}}{(M-1)^3}+\frac{C^{M-2-\beta}}{(M-2)^3}+\sum_{\substack{i+j=M\\ i,j>1}}\frac{C^{M-2\beta}}{i^3j^3}+\sum_{\substack{i+j=M-1\\ i,j>1}}\frac{C^{M-1-2\beta}}{i^3j^3}+\sum_{\substack{i+j+k=M\\ i,j,k>1}}\frac{C^{M-3\beta}}{i^3j^3k^3}\bigg)\\
	&\leq D_3\bigg(\frac{C^{M-1-\beta}}{(M-1)^3}+\frac{C^{M-2-\beta}}{(M-2)^3}+\frac{C^{M-2\beta}}{M^3}+\frac{C^{M-1-2\beta}}{M^3}+\frac{C^{M-3\beta}}{M^3}\bigg),
\end{align*}
for some constant $D_3>0$ (possibly different in each line), where we have again applied \cite[Lemma B.1]{GuoHadzicJangSchrecker22} to control the sums in the penultimate line. As $\beta\in(1,2)$, this implies the first bound of~\eqref{E:omegaMBracket}. Since the other combinations of $\bar{\rho}$ and $\bar{\om}$ follow analogously, we conclude the proof.
\end{proof}

In order to establish combinatorial estimates on the Taylor coefficients of $\bar{Q}$ that allow us to show the convergence of the formal series, we first collect some useful identities for the Taylor coefficients of certain key nonlinear quantities using the Fa\`a di Bruno formula.

\begin{lemma}\label{L:QSeries}
	The coefficients of the Taylor series for $\bar{Q}$ are given by
	\begin{align}
		\bar{Q}_M = (2-\ga)\al\Big(1-\frac{\al}{2}\Big)^2\sum_{\substack{i+j=M\\i>0}}\bar{P}_{i-1}\bar{W}_j,\label{E:QBar}
	\end{align}
	where, for $M\geq1$, we have:
	\begin{align}
		\bar{P}_M &= \sum_{\la\in\Lambda_M}\frac{M!\ga(\ga-1)\cdots(\ga+1-|\la|)}{\la_1!\cdots\la_M!}\bar{\rho}_0^{\ga-|\la|}\prod_{k=1}^{M}\bar{\rho}_k^{\la_k},\label{E:FdB1}\\
		\bar{W}_M &= \sum_{\la\in\Lambda_M}\frac{M!\big(\frac{n}{3}-1\big)\big(\frac{n}{3}-2\big)\cdots\big(\frac{n}{3}-|\la|\big)}{\la_1!\cdots\la_M!}(\bar{\rho}_0\bar{\om}_0)^{\frac{n}{3}-1-|\la|}\prod_{k=1}^{M}(\bar{\rho}\bar{\om})_k^{\la_k},\label{E:FdB2}
	\end{align}
	with
	\begin{align*}
		\Lambda_M &= \bigg\{\la=(\la_1,\dots,\la_M)\in\mathbb{N}_0^M : \sum_{k=1}^M k\la_k=M\bigg\}, & |\la| &= \la_1+\dots+\la_M.
	\end{align*}
\end{lemma}

\begin{proof}
Clearly $\bar{Q}_M$ is defined by
\begin{align*}
	\bar{Q}_M = \Big(1-\frac{\al}{2}\Big)^2\frac{(2-\ga)\al}{M!}\frac{d^M}{dx^M}\Big( x\bar{\rho}^\ga(\bar{\rho}\bar{\om})^{\frac{n}{3}-1}\Big)\bigg|_{x=0}
\end{align*}
and we recall the Faà di Bruno formula
\begin{align*}
	\frac{d^M}{dx^M}\big(f(g(x))\big) = \sum_{\la\in\Lambda_M}\frac{M!}{\la_1!\cdots\la_M!}f^{(|\la|)}\big(g(x)\big)\prod_{k=1}^{M}\bigg(\frac{g^{(k)}(x)}{k!}\bigg)^{\la_k}.
\end{align*}
Applying the Faà di Bruno formula separately to $\bar{\rho}^\ga$ and $(\bar{\rho}\bar{\om})^{\frac{n}{3}-1}$, we obtain~\eqref{E:FdB1} and~\eqref{E:FdB2} respectively. Then, since
\begin{align*}
	\bar{Q} = (2-\ga)\al\Big(1-\frac{\al}{2}\Big)^2 x\bar{\rho}^\ga(\bar{\rho}\bar{\om})^{\frac{n}{3}-1} = \sum_{M=0}^\infty\bar{Q}_Mx^M,
\end{align*}
the claimed expression~\eqref{E:QBar} follows.
\end{proof}

With the identities of the previous lemma, we now establish combinatorial bounds on the coefficients $\bar{P}_M$ and $\bar{W}_M$.

\begin{lemma}\label{L:PWBounds}
	Let $M>1$ and $\beta\in(1,2)$. Then there exists a constant $C_0\geq1$ such that, if:
	\begin{align*}
		|\bar{\rho}_m| &\leq \frac{C_0^{m-\beta}}{m^3}, &
		|\bar{\om}_m| &\leq \frac{C_0^{m-\beta}}{m^3},
	\end{align*}
	for each integer $1<m\leq M$, then there exists a constant $D>0$, independent of $C_0$ and $\beta$, such that:
	\begin{align}
		\big|\bar{P}_M\big| &\leq D\bigg(\frac{C_0^{M-\beta}}{M^3} + \frac{C_0^{M-2}}{M^2}\bigg),\\
		\big|\bar{W}_M\big| &\leq D\bigg(\frac{C_0^{M-\beta}}{M^3} + \frac{C_0^{M-2}}{M^2}\bigg).
	\end{align}
\end{lemma}

\begin{proof}
The proof follows similarly to \cite[Lemma B.5]{GuoHadzicJangSchrecker22}, noting that
\begin{align*}
	\bigg|\bigg(\frac{n}{3}-1\bigg)\bigg(\frac{n}{3}-2\bigg)\cdots\bigg(\frac{n}{3}-m\bigg)\bigg| \leq 3\bigg(\frac{n}{3}-1\bigg)(m-1)!
\end{align*}
holds for all even integers $n\geq4$.
\end{proof}

\begin{lemma}\label{L:QBound}
	Let $M>2$ and suppose that, for $C_0$ defined by Lemma~\ref{L:PWBounds}, there exist constants $C\geq C_0$ and $\beta\in(1,2)$ such that:
	\begin{align*}
		|\bar{\rho}_m| &\leq \frac{C^{m-\beta}}{m^3}, &
		|\bar{\om}_m| &\leq \frac{C^{m-\beta}}{m^3},
	\end{align*}
	for each integer $1<m<M$. Then there exists a constant $D>0$, independent of $C$ and $\beta$, such that
	\begin{align}
		|\bar{Q}_M| &\leq D\frac{C^{M-1-2\beta}}{M^3}.
	\end{align}
\end{lemma}

\begin{proof}
We see that
\begin{align*}
	|\bar{Q}_M| \leq (2-\ga)\al\Big(1-\frac{\al}{2}\Big)^2\sum_{i+j=M}|\bar{P}_{i-1}||\bar{W}_j|
\end{align*}
and by Lemma~\ref{L:PWBounds}, we have that there exists a constant $D_1>0$ such that, for all $M>1$, we have:
\begin{align*}
	\big|\bar{P}_M\big| &\leq D_1\bigg(\frac{C^{M-\beta}}{M^3} + \frac{C^{M-2}}{M^2}\bigg), &
	\big|\bar{W}_M\big| &\leq D_1\bigg(\frac{C^{M-\beta}}{M^3} + \frac{C^{M-2}}{M^2}\bigg).
\end{align*}
We first show the desired inequality for $\bar{Q}_3$. Since $\bar{P}_1$ and $\bar{W}_1$ are continuous functions of $\bar{\rho}_0$, $\bar{\om}_0$, $\bar{\rho}_1$ and $\bar{\om}_1$, which are all bounded only in terms of $n$ and $\ga$, then for $M=3$ we have
\begin{align*}
	\sum_{i+j=3}|\bar{P}_{i-1}||\bar{W}_j| &= |\bar{P}_0||\bar{W}_2| + |\bar{P}_1||\bar{W}_1| + |\bar{W}_0||\bar{P}_2|\\
	&\leq |\bar{P}_1||\bar{W}_1| + D_1\big(|\bar{P}_0|+|\bar{W}_0|\big)\bigg(\frac{C^{2-\beta}}{2^3} + \frac{C}{2^2}\bigg)\\
	&\leq D_2\frac{C^{3-1-2\beta}}{3^3},
\end{align*}
for some constant $D_2>0$.

Now, for $M>3$ we have
\begin{align*}
	\sum_{i+j=M}|\bar{P}_{i-1}||\bar{W}_j| &= |\bar{P}_0||\bar{W}_{M-1}| + |\bar{P}_{M-1}||\bar{W}_0| + |\bar{P}_{M-2}||\bar{W}_1| + |\bar{P}_1||\bar{W}_{M-2}| + \sum_{\substack{i+j=M\\ i-1,j>1}}|\bar{P}_{i-1}||\bar{W}_j|\\
	&\leq \big(|\bar{P}_0|+|\bar{W}_0|\big)\frac{C^{M-1-\beta}}{(M-1)^3} + \big(|\bar{P}_1|+|\bar{W}_1|\big)\frac{C^{M-2-\beta}}{(M-2)^3} + \sum_{\substack{i+j=M\\ i-1,j>1}}|\bar{P}_{i-1}||\bar{W}_j|,
\end{align*}
where by \cite[Lemma B.1]{GuoHadzicJangSchrecker22}
\begin{align*}
	\sum_{\substack{i+j=M\\ i-1,j>1}}|\bar{P}_{i-1}||\bar{W}_j| &\leq D_1^2\sum_{\substack{i+j=M\\ i-1,j>1}}\bigg(\frac{C^{i-1-\beta}}{(i-1)^3} + \frac{C^{i-3}}{(i-1)^2}\bigg)\bigg(\frac{C^{j-\beta}}{j^3} + \frac{C^{j-2}}{j^2}\bigg)\\
	&= D_1^2\sum_{\substack{i+j=M-1\\ i,j>1}}\bigg(\frac{C^{i-\beta}}{i^3} + \frac{C^{i-2}}{i^2}\bigg)\bigg(\frac{C^{j-\beta}}{j^3} + \frac{C^{j-2}}{j^2}\bigg)\\
	&= D_1^2\sum_{\substack{i+j=M-1\\ i,j>1}}\frac{C^{M-1-2\beta}}{i^3j^3} + \frac{C^{M-\beta-3}}{i^3j^2} + \frac{C^{M-\beta-3}}{i^2j^3} + \frac{C^{M-5}}{i^2j^2}\\
	&\leq D_3\bigg(\frac{C^{M-1-2\beta}}{M^3} + \frac{2C^{M-\beta-3}}{M^2} + \frac{C^{M-5}}{M^2}\bigg),
\end{align*}
for some constant $D_3>0$. We therefore conclude the proof.
\end{proof}

With Lemmas~\ref{L:BracketBounds} and~\ref{L:QBound}, we can now close inductive estimates on the source terms $\mathcal{F}_M$ and $\mathcal{G}_M$. This is the content of the next lemma.

\begin{lemma}\label{L:FMGMBounds}
	Let $M>2$ and suppose that, for $C_0$ defined by Lemma~\ref{L:PWBounds}, there exist constants $C\geq C_0$ and $\beta\in(1,2)$ such that:
	\begin{align}\label{E:Inductive1}
		|\bar{\rho}_m| &\leq \frac{C^{m-\beta}}{m^3}, &
		|\bar{\om}_m| &\leq \frac{C^{m-\beta}}{m^3},
	\end{align}
	for each integer $1<m<M$. Then there exists a constant $D>0$, independent of $C$ and $\beta$, such that:
	\begin{align}
		|\mathcal{F}_M| &\leq D\frac{C^{M-1-\beta}}{M^2},\\
		|\mathcal{G}_M| &\leq D\frac{C^{M-1-\beta}}{M^2}.
	\end{align}
\end{lemma}

\begin{proof}
We prove the claimed estimate for $\mathcal{G}_M$ as the estimate for $\mathcal{F}_M$ follows similarly. Observe first that from Remark~\ref{R:det}(ii) with choice $M_0=5$ and~\eqref{E:GMDef}, the bound follows directly for $2\leq M\leq6$ by choosing $D$ sufficiently large as $\mathcal{G}_M$ is a continuous function of the coefficients $\rho_m$ and $\om_m$ with order $m\leq M-1$. It remains therefore to estimate in the case $M\geq7$.

We estimate directly from the definition of $\mathcal{G}_M$,~\eqref{E:GMDef},
\begin{equation}
\begin{aligned}
	\mathcal{G}_M &\leq 3\big|\bar{\om}_0[\bar{\om}^2]_M\big| + \big|[\bar{\om}^3]_M\big| - \Big|\Big(\ga-1-\frac{\al}{2}\Big)[\bar{\om}^2]_M\Big| + \bigg|\frac{2}{9}\Big(1-\frac{\al}{2}\Big)^2\frac{[\bar{\rho}\bar{\om}^2]_M}{\bar{\rho}_0\bar{\om}_0}\bigg|\\
	&+ \sum_{\substack{i+j=M\\ i< M-1}}(i+1)(n-2)\big|\bar{\om}_{i+1}(\bar{\om}^2)_{j-1}\big|+ \sum_{i+j+k=M}(i+1)\frac{(n-2)\ga}{(2-\ga)\al}\big|\bar{\om}_{i+1}\bar{\om}_{j-1}\bar{Q}_k\big|\\
	&+ \sum_{i=0}^{M-1}\big|\bar{\om}_i\bar{Q}_{M-i}\big| + \sum_{i=1}^{M-1} \frac{3\ga}{(2-\ga)\al}\big|(\bar{\om}^2)_i - \bar{\om}_0\bar{\om}_i\big|\big|\bar{Q}_{M-i}\big|,\label{E:GMBound}
\end{aligned}
\end{equation}
where we have used that $\bar{Q}_0=0$.
Applying Lemma~\ref{L:BracketBounds} three times, we see that the first line is bounded as
\begin{align}
	3\big|\bar{\om}_0[\bar{\om}^2]_M\big| + \big|[\bar{\om}^3]_M\big| - \Big|\Big(\ga-1-\frac{\al}{2}\Big)[\bar{\om}^2]_M\Big| + \bigg|\frac{2}{9}\Big(1-\frac{\al}{2}\Big)^2\frac{[\bar{\rho}\bar{\om}^2]_M}{\bar{\rho}_0\bar{\om}_0}\bigg| \leq D_1\frac{C^{M-1-\beta}}{M^3}.\label{E:GM1}
\end{align}
for some constant $D_1>0$. Considering next the first sum on the second line of~\eqref{E:GMBound}, we estimate this (recalling $M\geq7$) as
\begin{align*}
	\sum_{\substack{i+j=M\\ i<M-1}}(i+1)|\bar{\om}_{i+1}|\big|(\bar{\om}^2)_{j-1}\big| &= (M-1)|\bar{\om}_{M-1}|\big|(\bar{\om}^2)_1\big| + (M-2)|\bar{\om}_{M-2}|\big|(\bar{\om}^2)_2\big|+ |\bar{\om}_1|\big|(\bar{\om}^2)_{M-1}\big|\\
	&+ \sum_{\substack{i+j=M\\ 0<i<M-3}}(i+1)|\bar{\om}_{i+1}|\big|(\bar{\om}^2)_{j-1}\big|.
\end{align*}
Applying~\eqref{E:Inductive1} and Lemma~\ref{L:Quadratic}, we deduce the estimate
\begin{align}
	\sum_{\substack{i+j=M\\ i< M-1}}(i+1)|\bar{\om}_{i+1}|\big|(\bar{\om}^2)_{j-1}\big| &\leq D_2\frac{C^{M-1-\beta}}{(M-1)^2}+\sum_{i=1}^{M-4}\frac{C^{i+1-\beta}}{(i+1)^2}\frac{C^{j-1-\beta}}{(M-i-1)^3}\notag\\
	&\leq D_2\frac{C^{M-1-\beta}}{(M-1)^2}\bigg(1+\sum_{i=2}^{M-3}\frac{C^{1-\beta}}{i^2(M-i)^3}\bigg)\notag\\
	&\leq D_2\frac{C^{M-1-\beta}}{(M-1)^2},\label{E:GM2}
\end{align}
for some constant $D_2>0$ (possibly different on each line) and where we have applied \cite[Lemma B.1]{GuoHadzicJangSchrecker22} to estimate the final sum.

Considering next the second sum on the second line of~\eqref{E:GMBound}, we note that $\bar{P}_0$, $\bar{P}_1$, $\bar{W}_0$ and $\bar{W}_1$ are functions of $\bar{\rho}_0$, $\bar{\om}_0$, $\bar{\rho}_1$ and $\bar{\om}_1$ only, that is, not any higher order coefficients. Because of this, $\bar{Q}_1$ and $\bar{Q}_2$ are also only a function of these coefficients and thus can be bounded by a constant. We therefore rewrite the sum to isolate terms containing $\bar{\om}_k$ (for $k=0,1$) and $\bar{Q}_j$ (for $j=1,2$), obtaining
\begin{align*}
	&\sum_{i+j+k=M}(i+1)|\bar{\om}_{i+1}||\bar{\om}_{j-1}||\bar{Q}_k|\\
	&\leq \bar{\om}_0|\bar{\om}_1||\bar{Q}_{M-1}| + (|\bar{\om}_1||\bar{\om}_1|+2|\bar{\om}_2||\bar{\om}_0|)|\bar{Q}_{M-2}|\\
	&+ |\bar{Q}_2|\bigg((M-2)|\bar{\om}_{M-2}||\bar{\om}_0|+(M-3)|\bar{\om}_{M-3}||\bar{\om}_1|+|\bar{\om}_{M-3}||\bar{\om}_1|+\sum_{i=1}^{M-5}(i+1)|\bar{\om}_{i+1}||\bar{\om}_{M-i-1}|\bigg)\\
	&+ |\bar{Q}_1|\bigg((M-1)|\bar{\om}_{M-1}||\bar{\om}_0|+(M-2)|\bar{\om}_{M-2}||\bar{\om}_1|+|\bar{\om}_{M-2}||\bar{\om}_1|+\sum_{i=1}^{M-4}(i+1)|\bar{\om}_{i+1}||\bar{\om}_{M-i-1}|\bigg)\\
	&+ \sum_{k=3}^{M-3}|\bar{Q}_k|\Big(|\bar{\om}_1||\bar{\om}_{M-k-1}|+(M-k)|\bar{\om}_{M-k}||\bar{\om}_0|+(M-k-1)|\bar{\om}_{M-k-1}||\bar{\om}_1|\Big)\\
	&+ \sum_{k=3}^{M-3}|\bar{Q}_k|\sum_{i=1}^{M-k-3}(i+1)|\bar{\om}_{i+1}||\bar{\om}_{M-k-i-1}|.
\end{align*}
Applying the inductive assumption~\eqref{E:Inductive1} and Lemma~\ref{L:QBound}, we therefore obtain the bound
\begin{align*}
	\sum_{i+j+k=M}&(i+1)|\bar{\om}_{i+1}||\bar{\om}_{j-1}||\bar{Q}_k|\\
	&\leq D_3\frac{C^{M-1-\beta}}{(M-1)^2}+D_3\sum_{i=1}^{M-4}\frac{C^{i+1-\beta}}{(i+1)^2}\frac{C^{M-i-1-\beta}}{(M-i-1)^3}\\
	&+ \sum_{k=3}^{M-3}\frac{C^{k-1-2\beta}}{k^3}\bigg(\frac{C^{M-k-\beta}}{(M-k)^2}+\sum_{i=1}^{M-k-3}\frac{C^{i+1-\beta}}{(i+1)^2}\frac{C^{M-k-i-1-\beta}}{(M-k-i-1)^3}\bigg),
\end{align*}
for some constant $D_3>0$, where the last sum is non-empty only if $M-k-3\geq1$, that is, when $k\leq M-4$. From \cite[Lemma B.1]{GuoHadzicJangSchrecker22}, we estimate each of these sums, obtaining
\begin{align}
	\sum_{i+j+k=M}&(i+1)|\bar{\om}_{i+1}||\bar{\om}_{j-1}||\bar{Q}_k|\notag\\
	&\leq D_4\bigg(1+\frac{1}{C^{\beta-1}}\bigg)\frac{C^{M-1-\beta}}{(M-1)^2}+\sum_{k=3}^{M-3}\frac{C^{k-1-2\beta}}{k^3}\bigg(\frac{C^{M-k-\beta}}{(M-k)^2}+\frac{C^{M-k-2\beta}}{(M-k)^2}\bigg)\notag\\
	&\leq D_4\frac{C^{M-1-\beta}}{(M-1)^2}.\label{E:GM3}
\end{align}
for some constant $D_4>0$, possibly different on each line. To estimate the final two terms in~\eqref{E:GMBound}, we apply similar but simpler arguments using~\eqref{E:Inductive1} and Lemma~\ref{L:QBound} (as the estimates are purely quadratic), to conclude the bound
\begin{align}
	\sum_{i=0}^{M-1}\big|\bar{\om}_i\bar{Q}_{M-i}\big| + \sum_{i=1}^{M-1} \frac{3\ga}{(2-\ga)\al}\big|(\bar{\om}^2)_i - \bar{\om}_0\bar{\om}_i\big|\big|\bar{Q}_{M-i}\big|
	\leq D_5\frac{C^{M-1-\beta}}{(M-1)^2}\label{E:GM4}
\end{align}
for some constant $D_5>0$. Combining~\eqref{E:GM1}--\eqref{E:GM4}, we conclude the proof.
\end{proof}

\subsection{Proof of Local Existence Theorem}\label{S:LWPProof}

\begin{proof}[Proof of Theorem~\ref{T:LWP}]
Recalling the formal expansion~\eqref{E:TaylorxForm}, it is clear that the local convergence of the infinite series for $(\rho,\om)$ follows by showing that:
\begin{align}
	|\bar{\rho}_M| &\leq \frac{C^{M-\beta}}{M^3}, &
	|\bar{\om}_M| &\leq \frac{C^{M-\beta}}{M^3}.\label{E:Inductive2}
\end{align}
for some constants $C\geq1$, $\beta\in(1,2)$ and all $M\geq2$. We proceed by induction.

In the case $M=2$, we see from Lemma~\ref{L:Taylor} that $\bar{\rho}_2$ and $\bar{\om}_2$ are continuous functions of $\ga$ and $n$ for $\ga\in\big(\frac{19}{12},\frac{11}{6}\big)$ and thus bounded by a constant $D>0$. Therefore, given $\beta\in(1,2)$, we may choose $C\geq C_0$ sufficiently large, where $C_0$ is as in Lemma~\ref{L:PWBounds}, so that:
\begin{align*}
	|\bar{\rho}_2| &\leq \frac{C^{2-\beta}}{8}, &
	|\bar{\om}_2| &\leq \frac{C^{2-\beta}}{8}.
\end{align*}
Suppose now that $M\geq3$ and that estimates~\eqref{E:Inductive2} hold for $M=2,\ldots,M-1$. By Lemma~\ref{L:rhoMomegaMBounds}, we obtain the estimates:
\begin{align*}
	|\bar{\rho}_M| &\leq \frac{D}{M}\Big(|\mathcal{F}_M| + \frac{1}{M}|\mathcal{G}_M|\Big), &
	|\bar{\om}_M| &\leq \frac{D}{M}\Big(|\mathcal{G}_M| + \frac{1}{M}|\mathcal{F}_M|\Big).
\end{align*}
By the inductive assumption, we see that the assumptions of Lemma~\ref{L:FMGMBounds} hold, and thus we deduce:
\begin{align*}
	|\bar{\rho}_M| &\leq \frac{D}{C}\frac{C^{M-\beta}}{M}, &
	|\bar{\om}_M| &\leq \frac{D}{C}\frac{C^{M-\beta}}{M}.
\end{align*}
By choosing $C\geq\min\{C_0,D\}$, we conclude that the estimate holds for all $M\geq3$. Therefore the series~\eqref{E:TaylorxForm} converge absolutely for $|x|<\frac{1}{C}$.

As the coefficients in the series~\eqref{E:TaylorxForm} satisfy the recurrence relation~\eqref{E:Matrix}, within the radius of convergence, $|x|<\frac{1}{C}$, term by term differentiation and direct substitution of~\eqref{E:TaylorxForm} into the ODE system~\eqref{E:rhoPrime1}--\eqref{E:p} shows that the obtained functions are indeed analytic solutions.
\end{proof}

\section{Global Existence}

\subsection{Existence of Supersonic Collapse}

The aim of this section is to extend the local solution given by Theorem~\ref{T:LWP} globally to the right for all $y>0$. The core of this argument is a series of bootstrap assumptions, which we propagate via dynamical invariances of the flow. These assumptions are most conveniently stated in terms of quantities measuring the deviation of density and relative velocity from their initial values. To that end, we define:
\begin{align*}
	R(y) &:= \rho_0-\rho(y), & \Om(y) &:= \om(y)-\om_0.
\end{align*}

To prove global existence, we will show that on the maximal interval of existence the following inequalities are dynamically propagated:
\begin{align}
	\rho' &< 0,\label{E:B1}\\
	\rho &> 0,\label{E:B2}\\
	\om &> \om_0,\label{E:B3}\\
	y^2\om R &> m_0\frac{p}{\rho},\label{E:B4}\\
	\frac{n+1}{n-2}\rho_0\Om &> \om_0R,\label{E:B5}
\end{align}
where
\begin{align}
	m_0 = \frac{3(n+1)n}{2(\ga-1)(11-6\ga)}\rho_0\om_0\label{E:m0Def}
\end{align}
and we recall:
\begin{align*}
	\rho_0 &= \frac{1}{6\pi}, &
	\om_0 &= \frac{(3\ga-4)(2-\ga)}{n+3\ga-6}.
\end{align*}

\begin{lemma}\label{L:BootstrapInitial}
	Let $\ga\in\big(\frac{19}{12},\frac{11}{6}\big)$and $(\rho,\om)$ be the unique local-in-$y$ solution to the initial value problem~\eqref{E:rhoPrime1}--\eqref{E:omegaPrime1} given by Theorem~\ref{T:LWP}. Then there exists a $\de>0$ such that~\eqref{E:B1}--\eqref{E:B5} hold for $y\in(0,\de)$.
\end{lemma}

\begin{proof}
Starting with~\eqref{E:B1}--\eqref{E:B3}, we know from Theorem~\ref{T:LWP} and Lemma~\ref{L:Taylor} that:
\begin{align*}
	\rho(y) &= \rho_0 + \rho_{n-2}y^{n-2} + O_{y\to0}\big(y^{2(n-2)}\big),\\
	\om(y) &= \om_0 + \om_{n-2}y^{n-2} + O_{y\to0}\big(y^{2(n-2)}\big),
\end{align*}
where:
\begin{align*}
	\rho_{n-2} &= \bar{\rho}_1 < 0, & \om_{n-2} &= \bar{\om}_1 > 0.
\end{align*}
Therefore~\eqref{E:B1},~\eqref{E:B2} and~\eqref{E:B3} hold on an interval $(0,\de_1)$ for some $\de_1>0$. For~\eqref{E:B5}, we first observe from~\eqref{E:rho1omega1Alt} that
\begin{align*}
	\frac{n+1}{n-2}\rho_0\Om - \om_0R = \bigg(\frac{n+1}{n-2}\rho_0\bar{\om}_2 + \om_0\bar{\rho}_2\bigg)y^{2(n-2)} + o_{y\to0}\big(y^{2(n-2)}\big).
\end{align*}
Now, from~\eqref{E:rho2}--\eqref{E:omega2}, we see
\begin{align*}
	\frac{n+1}{n-2}\rho_0\bar{\om}_2 + \om_0\bar{\rho}_2 = \frac{9n(2n-1)S(n,\ga)-9n(n+1)T(n,\ga)}{4(\ga-1)^2(6\ga-7)(12\ga-19)(11-6\ga)^2}{\frac{\om_0}{\rho_0}}p_0^2 > 0,
\end{align*}
where the positivity follows from
\begin{align*}
	9n(2n-1)S(n,\ga)-9n(n+1)T(n,\ga) &= -18\ga(\ga-1)(11-6\ga) + 2n(90\ga^2 - 252\ga + 161)\\
	&- n^2(108\ga^3 - 234\ga^2 + 133) + n^3(6\ga-7)(12\ga-19) > 0
\end{align*}
for all $n\geq4$ and $\ga\in\big(\frac{19}{12},\frac{11}{6}\big)$. Thus there exists a $\de_2>0$ such that~\eqref{E:B5} holds on $(0,\de_2)$.

Finally, for~\eqref{E:B4} we see that
\begin{align*}
	&\frac{1}{y^2\om}\bigg(y^2\om R - m_0\frac{p}{\rho}\bigg)\\
	&= (\rho_0-\rho)-m_0y^{n-2}\rho^\ga(\rho\om)^{\frac{n}{3}-1}\\
	&= -\bar{\rho}_1y^{n-2} - \bar{\rho}_2y^{2(n-2)}\\
	&- m_0\rho_0^\ga(\rho_0\om_0)^{\frac{n}{3}-1}y^{n-2}\bigg(1+\ga\frac{\bar{\rho}_1}{\bar{\rho}_0}y^{n-2}\bigg)\bigg(1+\bigg(\frac{n}{3}-1\bigg)\bigg(\frac{\bar{\rho}_1}{\rho_0}+\frac{\bar{\om}_1}{\om_0}\bigg)y^{n-2}\bigg) + o_{y\to0}\big(y^{2(n-2)}\big)\\
	&= -\bar{\rho}_2y^{2(n-2)} - m_0\rho_0^\ga(\rho_0\om_0)^{\frac{n}{3}-1}\bigg(\ga\frac{\bar{\rho}_1}{\rho_0} + \bigg(\frac{n}{3}-1\bigg)\bigg(\frac{\bar{\rho}_1}{\rho_0}+\frac{\bar{\om}_1}{\om_0}\bigg)\bigg)y^{2(n-2)} + o_{y\to0}\big(y^{2(n-2)}\big)\\
	&= \left(-\bar{\rho}_2 + \bigg(\frac{3(n+1)n}{2(\ga-1)(11-6\ga)}\bigg)^2\bigg(\ga+\frac{n-3}{n+1}\bigg)\frac{p_0^2}{\rho_0}\right)y^{2(n-2)} + o_{y\to0}\big(y^{2(n-2)}\big),
\end{align*}
where we have applied~\eqref{E:rho1omega1} to see the vanishing of terms at order $y^{n-2}$, and it can be checked that
\begin{align*}
	-\bar{\rho}_2 + \bigg(\frac{3(n+1)n}{2(\ga-1)(11-6\ga)}\bigg)^2\bigg(\ga+\frac{n-3}{n+1}\bigg)\frac{p_0^2}{\rho_0} > 0
\end{align*}
for all $n\geq4$ and $\ga\in\big(\frac{19}{12},\frac{11}{6}\big)$. Thus there exists a $\de_3>0$ such that~\eqref{E:B4} holds on $(0,\de_3)$. Therefore we have that~\eqref{E:B1}--\eqref{E:B5} hold on the interval $(0,\de)$ for $\de=\min\{\de_1,\de_2,\de_3\}$.
\end{proof}

Before we prove that the assumptions~\eqref{E:B1}--\eqref{E:B5} are dynamically propagated, we collect several identities that will play an important role in the proof.

\begin{lemma}\label{L:Bootstrap}
	Consider the unique local-in-$y$ solution to the initial value problem~\eqref{E:rhoPrime1}--\eqref{E:omegaPrime1} given by Theorem~\ref{T:LWP} extended to its maximal interval of existence $[0,T)$. Then the following identities hold:
	\begin{align}
		\frac{d}{dy}\bigg(\frac{n+1}{n-2}\rho_0\Om-\om_0R\bigg) 
		&= -\bigg(\frac{\rho'}{\rho}+\frac{3}{y}\bigg)\bigg(\frac{n+1}{n-2}\rho_0\Om-\om_0R\bigg)\notag\\
		&+ \frac{3\om_0}{(n-2)y}\left(-\frac{y\rho'}{\rho}\bigg(\rho_0+\frac{2}{3}(n-2)R\bigg)-(n-2)R\right),\label{E:B5ID}\\
		\frac{d}{dy}\bigg(y^2\om R - m_0\frac{p}{\rho}\bigg)
		&= 2y\om R + y^2\om\bigg(\frac{3\om_0-3\om}{y\om}-\frac{\rho'}{\rho}\bigg) R - y^2\rho\om\frac{\rho'}{\rho}\notag\\
		&- m_0\bigg((\ga-1)\frac{\rho'}{\rho} + \frac{n}{3}\bigg(\frac{3\om_0-3\om}{y\om}\bigg) + \frac{n}{y}\bigg)\frac{p}{\rho},\label{E:B4ID}\\
		-y\frac{\rho'}{\rho} 
		&= -\frac{y^2}{G}\big(\mathcal{A}\Om + \mathcal{B}R + \mathcal{C}\Om^2 + \mathcal{D}R\Om + \mathcal{E}\big),\label{E:rhoPrimeID}
	\end{align}
	where:
	\begin{align}
		\mathcal{A} &:= 4\om_0 + \frac{(2-\ga)(n-6\ga+6)}{2(n+3\ga-6)} - \frac{4\pi}{3\om_0}\bigg(\frac{3(n-2)(2-\ga)}{2(n+3\ga-6)}\bigg)^2\rho_0,\\
		\mathcal{B} &:= \frac{4\pi}{3\om_0}\bigg(\frac{3(n-2)(2-\ga)}{2(n+3\ga-6)}\bigg)^2\om_0 - \frac{n\om_0}{m_0}\bigg(\frac{3(n-2)(2-\ga)}{2(n+3\ga-6)}\bigg)^2,\\
		\mathcal{C} &:= 2,\label{E:CDef}\\
		\mathcal{D} &:= \frac{4\pi}{3\om_0}\bigg(\frac{3(n-2)(2-\ga)}{2(n+3\ga-6)}\bigg)^2,\label{E:DDef}\\
		\mathcal{E} &:= \mathcal{E}(y;\rho,\om) = \bigg(\frac{3(n-2)(2-\ga)}{2(n+3\ga-6)}\bigg)^2\frac{n\om_0}{m_0y^2\om}\bigg(y^2\om R - m_0\frac{p}{\rho}\bigg).\label{E:EDef}
	\end{align}
	Moreover, the expression $\mathcal{A}$ is positive if and only if 
	\begin{align}
		\ga \geq \frac{10+n}{9},\label{E:10+n}
	\end{align}
	and furthermore, for any $\ga>\max\big\{\frac{19}{12},\frac{10+n}{9}\big\}$, as long as~\eqref{E:B5} holds weakly, we have the bound
	\begin{align}
		\mathcal{A}\Om + \mathcal{B}R \geq \frac{\om_0^2}{\rho_0}(n-2)R\label{E:AB}
	\end{align}
	with equality if and only if $\frac{n+1}{n-2}\rho_0\Om=\om_0R$.
\end{lemma}

\begin{remark}
	Unlike the requirement that $\ga\in\big(\frac{19}{12},\frac{11}{6}\big)$, which determines the local $(y\to0)$ sign of $\frac{n+1}{n-2}\rho_0\Om=\om_0R$, it is less clear what the significance of condition~\eqref{E:10+n} is. For $n=4$ this condition allows for $\ga$ in the full range $\big(\frac{19}{12},\frac{11}{6}\big)$, for $n=6$ this range is restricted and for any even integer $n\geq8$ there are clearly no values of $\ga<2$ such that $\ga\geq\frac{10+n}{9}$. Even though condition~\eqref{E:10+n} is essential to the proof of Lemma~\ref{L:Bootstrap}, this condition is likely to be a technicality of the method of proof, rather than a fundamental obstruction to the existence of smooth solutions for larger values of $n$.
\end{remark}

\begin{proof}
To prove~\eqref{E:B5ID} we use~\eqref{E:rhoPrime1}--\eqref{E:omegaPrime1} to get
\begin{align*}
	\frac{d}{dy}\bigg(\frac{n+1}{n-2}\rho_0\Om-\om_0R\bigg) &= \frac{n+1}{n-2}\rho_0\om' + \om_0\rho'\\
	&= \frac{n+1}{n-2}\rho_0\om\bigg(\frac{3\om_0-3\om}{y\om}-\frac{\rho'}{\rho}\bigg) + \om_0\rho\frac{\rho'}{\rho}\\
	&= \frac{1}{y}\left(-\bigg(\frac{n+1}{n-2}\rho_0\om - \om_0\rho\bigg)y\frac{\rho'}{\rho} - 3\frac{n+1}{n-2}\rho_0\Om\right).
\end{align*}
We now rewrite
\begin{align*}
	\frac{n+1}{n-2}\rho_0\om = \frac{n+1}{n-2}\rho_0(\om-\om_0) - \om_0R + \frac{n+1}{n-2}\rho_0\om_0 + \om_0R
\end{align*}
and similarly for the term $\frac{n+1}{n-2}\rho_0\Om$. A direct manipulation then gives~\eqref{E:B5ID}. The proof of~\eqref{E:B4ID} is a direct calculation based on~\eqref{E:rhoPrime1}--\eqref{E:omegaPrime1}.

To prove~\eqref{E:rhoPrimeID} we use~\eqref{E:rhoPrime1} to obtain
\begin{align}
	-y\frac{\rho'}{\rho} 
	&= -\frac{y(yh-q)}{G}\notag\\
	&= -\frac{y^2}{G}\bigg(2\om^2 + \Big(\ga-1-\frac{\al}{2}\Big)\om + \Big(\ga-1-\frac{\al}{2}\Big)(2-\ga) - \frac{4\pi}{3\om_0}\Big(1-\frac{\al}{2}\Big)^2\rho\om\bigg)\notag\\
	&+ \frac{y^2}{G}n\om_0\Big(1-\frac{\al}{2}\Big)^2\frac{p}{y^2\rho\om}\notag\\
	&= -\frac{y^2}{G}\bigg(2(\om+\om_0) + \Big(\ga-1-\frac{\al}{2}\Big) - \frac{4\pi}{3\om_0}\Big(1-\frac{\al}{2}\Big)^2\rho_0\bigg)\Om\notag\\
	&- \frac{y^2}{G}\frac{4\pi}{3\om_0}\Big(1-\frac{\al}{2}\Big)^2\om R + \frac{y^2}{G}n\om_0\Big(1-\frac{\al}{2}\Big)^2\frac{p}{y^2\rho\om}.\label{E:rhoPrimeIDAlt}
\end{align}
Recall that
\begin{align*}
	\al = \frac{n(3\ga-4)}{n+3\ga-6},
\end{align*}
so that:
\begin{align*}
	\ga-1-\frac{\al}{2} &= \frac{(2-\ga)(n-6\ga+6)}{2(n+3\ga-6)}, & 1-\frac{\al}{2} &= \frac{3(n-2)(2-\ga)}{2(n+3\ga-6)}.
\end{align*}
Then, upon regrouping~\eqref{E:rhoPrimeIDAlt} as a polynomial of $ \Om$ and $R$, we obtain~\eqref{E:rhoPrimeID}--\eqref{E:EDef}.

To prove the claim about the positivity of $\mathcal{A}$, we first observe that due to
\begin{align*}
	\om_0 = \frac{4-3\ga+\al}{3} = \frac{(3\ga-4)(2-\ga)}{n+3\ga-6},
\end{align*}
we easily check that
\begin{align*}
	\mathcal{A} = \frac{2-\ga}{2(3\ga-4)(n+3\ga-6)} \big(8(3\ga-4)^2 + (3\ga-4)(n-6\ga+6)- (n-2)^2\big).
\end{align*}
The expression in parentheses may be expressed as a quadratic polynomial in $z=3\ga-4$ of the form $q(3\ga-4)$ where
\begin{align*}
	q(z) = 6z^2 + (n-2)z - (n-2)^2 = \big(3z-(n-2)\big)(2z+n-2).
\end{align*}
As $\ga\in\big(\frac{19}{12},\frac{11}{6}\big)$, and thus $3\ga-4>0$, we see $\mathcal{A}\geq0$ if and only if $3\ga-4\geq\frac{n-2}{3}$. The latter condition is equivalent to~\eqref{E:10+n}.

For the proof of~\eqref{E:AB}, we first note that since $\ga>\max\big\{\frac{19}{12},\frac{10+n}{9}\big\}$ by the above, we know that $\mathcal{A}\geq0$. Therefore, applying the weak version of bootstrap assumption~\eqref{E:B5} in the form
\begin{align*}
	\Om \geq \frac{n-2}{n+1}\frac{\om_0}{\rho_0}R,
\end{align*}
gives us the lower bound
\begin{align*}
	\mathcal{A}\Om + \mathcal{B}R \geq \bigg( \frac{n-2}{n+1}\frac{\om_0}{\rho_0}\mathcal{A} + \mathcal{B}\bigg)R.
\end{align*}
Since
\begin{align*}
	\mathcal{A}\Om + \mathcal{B}R &= \frac{n-2}{n+1}\frac{\om_0}{\rho_0}\mathcal{A}R + \mathcal{B}R\\
	&= \frac{n-2}{n+1}\frac{\om_0}{\rho_0}\bigg(4\om_0 + \frac{(2-\ga)(n-6\ga+6)}{2(n+3\ga-6)} - \frac{4\pi}{3\om_0}\bigg(\frac{3(n-2)(2-\ga)}{2(n+3\ga-6)}\bigg)^2\rho_0\bigg)R\\
	&+ \bigg(\frac{4\pi}{3\om_0}\bigg(\frac{3(n-2)(2-\ga)}{2(n+3\ga-6)}\bigg)^2\om_0 - \frac{n\om_0}{m_0}\bigg(\frac{3(n-2)(2-\ga)}{2(n+3\ga-6)}\bigg)^2\bigg)R\\
	&= (n-2)\frac{\om_0^2}{\rho_0}\bigg(\frac{4}{n+1} + \frac{n-6\ga+6}{2(n+1)(3\ga-4)} - \frac{(n-2)^2}{2(n+1)(3\ga-4)^2}\bigg)R\\
	&+ (n-2)\frac{\om_0^2}{\rho_0}\bigg(\frac{n-2}{2(3\ga-4)^2} - \frac{3(n-2)(\ga-1)(11-6\ga)}{2(n+1)(3\ga-4)^2}\bigg)R\\
	&= (n-2)\frac{\om_0^2}{\rho_0}R,
\end{align*}
we obtain~\eqref{E:AB} and the proof is complete.
\end{proof}

\begin{lemma}
	For any $x\in[0,1]$, any $\ga\in\big(\frac{4}{3},2\big)$ and any even $n\geq4$, we have the inequality
	\begin{align}
		-\frac{(n-2)^2\ga}{n+1}x^2 + (n-2)\bigg(\frac{2}{3}n-\ga\bigg)x + n \geq 0.\label{E:QuadInq}
	\end{align}
\end{lemma}

\begin{proof}
The proof is elementary. By the quadratic formula it is easy to check that one root is strictly negative. By a direct calculation, one checks that the quadratic polynomial is positive at $x=0$ and $x=1$, and therefore bound~\eqref{E:QuadInq} must hold.
\end{proof}

\begin{lemma}\label{L:Free}
	Consider the unique local-in-$y$ solution to the initial value problem~\eqref{E:rhoPrime1}--\eqref{E:omegaPrime1} given by Theorem~\ref{T:LWP} extended smoothly to its maximal interval of existence $[0,T)$. If~\eqref{E:B2} and~\eqref{E:B3} hold on $[0,T)$, then	for all $y\in[0,T)$, we have
	\begin{align}
		\om_0R > \rho_0\Om - R\Om.\label{E:Free}
	\end{align}
\end{lemma}

\begin{proof}
We first note that, by Lemma~\ref{L:BootstrapInitial}, there exists a $\de>0$ such that~\eqref{E:B5} holds on $[0,\de]$ and therefore so does the weaker inequality~\eqref{E:Free}. We next observe
\begin{align*}
	\om_0R - \rho_0\Om + R\Om = -\rho\om + \rho_0\om_0,
\end{align*}
so that, computing the derivative,
\begin{align*}
	\frac{d}{dy}\big(\om_0R-\rho_0\Om+R\Om\big) = -(\rho\om)' = -\rho\om\bigg(\frac{\rho'}{\rho}+\frac{\om'}{\om}\bigg) = -\rho\om\bigg(\frac{3\om_0-3\om}{y\om}\bigg) \geq 0,
\end{align*}
where we have applied~\eqref{E:B2}--\eqref{E:B3} in the last inequality. As~\eqref{E:Free} holds at $\de>0$, it therefore holds on $[0,T)$.
\end{proof}

\begin{lemma}\label{L:Supersonic}
	Let $n=4,6$ and $\ga>\max\big\{\frac{19}{12},\frac{10+n}{9}\big\}$. Consider the unique local-in-$y$ solution to the initial value problem~\eqref{E:rhoPrime1}--\eqref{E:omegaPrime1} given by Theorem~\ref{T:LWP} extended smoothly to its maximal interval of existence $[0,T)$. If the bootstrap assumptions~\eqref{E:B1}--\eqref{E:B5} hold weakly on $[0,T)$, then the flow remains uniformly supersonic. In particular, for any $\de>0$, there exists a $g_0(n,\ga,\de)>0$ such that, for $y\in[\de,T)$,
	\begin{align}
		G(y;\rho,\om) \leq -g_0y^2\om.\label{E:GUpper}
	\end{align}
\end{lemma}

\begin{proof}
We expand $G$ from~\eqref{E:GDef} and apply the weak inequality~\eqref{E:B4} to obtain 
\begin{align}
	-G &= y^2\om^2 - \ga\Big(1-\frac{\al}{2}\Big)^2\frac{p}{\rho}\notag\\
	&\geq y^2\om^2 - \frac{\ga}{m_0}\Big(1-\frac{\al}{2}\Big)^2y^2\om R\notag\\
	&= y^2\om^2 - \frac{2\ga(\ga-1)(11-6\ga)}{3(n+1)n\om_0\rho_0}\Big(1-\frac{\al}{2}\Big)^2y^2\om R\notag\\
	&= y^2\om \bigg(\om - \frac{3(n-2)^2\ga(\ga-1)(11-6\ga)}{2(n+1)n(3\ga-4)^2}\frac{\om_0}{\rho_0}R\bigg)\notag\\
	&\geq \frac{2n-1}{n+1}\bigg(1- \frac{3(n-2)^2\ga(\ga-1)(11-6\ga)}{2(2n-1)n(3\ga-4)^2}\bigg)\frac{\om_0}{\rho_0}y^2\om R,\label{E:-GBound}
\end{align}
where we have used the bound
\begin{align*}
	\om = \om_0 + \Om \geq \frac{2n-1}{n+1}\frac{\om_0}{\rho_0}R,
\end{align*}
which follows from the bootstrap bound~\eqref{E:B5} and the fact that $R\leq\rho_0$ due to~\eqref{E:B1}. It is straightforward to check that
\begin{align*}
	1- \frac{3(n-2)^2\ga(\ga-1)(11-6\ga)}{2n(2n-1)(3\ga-4)^2} > 0
\end{align*}
for $n=4,6$ and $\ga>\max\big\{\frac{19}{12},\frac{10+n}{9}\big\}$. To conclude the proof, we note that by Lemma~\ref{L:BootstrapInitial} and~\eqref{E:B1}, given $\de>0$, there exists a $\varrho_0(\de)>0$ such that for $y\in[\de,T)$, we have $R\geq\varrho_0$. This implies the existence of $g_0$ as stated in~\eqref{E:GUpper}.
\end{proof}

\begin{lemma}[Upper bound on $\om$]\label{L:omegaUpper}
	Let $n=4,6$ and $\ga>\max\big\{\frac{19}{12},\frac{10+n}{9}\big\}$. Consider the unique local-in-$y$ solution to the initial value problem~\eqref{E:rhoPrime1}--\eqref{E:omegaPrime1} given by Theorem~\ref{T:LWP} extended smoothly to its maximal interval of existence $[0,T)$. If the bootstrap assumptions~\eqref{E:B1}--\eqref{E:B5} hold weakly on $[0,T)$, then there exists a constant $\om_{\text{upper}}<\infty$, depending only on $\ga$ and $n$, such that
	\begin{align*}
		\sup_{y\in[0,T)}\om(y) < \om_{\text{upper}}.
	\end{align*}
\end{lemma}

\begin{proof}
From the weak forms of~\eqref{E:B2}--\eqref{E:B3}, we see that $\rho,\om\geq0$. Thus $q\geq0$ and so
\begin{align*}
	yh-q \leq y\bigg(2\om^2+\Big(\ga-1-\frac{\al}{2}\Big)\om + \Big(\ga-1-\frac{\al}{2}\Big)(2-\ga)\bigg).
\end{align*}
We then use~\eqref{E:omegaPrime1},~\eqref{E:hDef}--\eqref{E:qDef} and $G<0$ from Lemma~\ref{L:Supersonic} to obtain the upper bound
\begin{align}
	\om' < -\frac{3(\om-\om_0)}{y} - \frac{y\om}{G}\bigg(2\om^2 + \Big(\ga-1-\frac{\al}{2}\Big)\om + \Big(\ga-1-\frac{\al}{2}\Big)(2-\ga)\bigg).\label{E:omegaPrimeUpper}
\end{align}
On the other hand, from the penultimate line of~\eqref{E:-GBound} we have:
\begin{align*}
	-G &> y^2\om \big(\om-c_1(n,\ga)\big), & c_1(n,\ga) &:= \frac{3(n-2)^2\ga(\ga-1)(11-6\ga)}{2(n+1)n(3\ga-4)^2}\om_0,
\end{align*}
where we have used the bound $R\leq\rho_0$. Inserting this into~\eqref{E:omegaPrimeUpper}, we conclude
\begin{align*}
	y\om' < -3(\om-\om_0) + \frac{1}{\om-c_1(n,\ga)}\bigg(2\om^2 + \Big(\ga-1-\frac{\al}{2}\Big)\om + \Big(\ga-1-\frac{\al}{2}\Big)(2-\ga)\bigg),
\end{align*}
which then easily implies that there exists a $(n,\ga)$-dependent finite upper bound on $\om(y)$ for all $y\in[0,T)$.
\end{proof}

\begin{theorem}\label{T:GE}
	Let $n=4,6$ and $\ga>\max\big\{\frac{19}{12},\frac{10+n}{9}\big\}$. Consider the unique local-in-$y$ solution to the initial value problem~\eqref{E:rhoPrime1}--\eqref{E:omegaPrime1} given by Theorem~\ref{T:LWP} on its maximal interval of existence $[0,T)$. Let $0<\mathcal{T}\leq T$ be the maximal time $y=\mathcal{T}$, such that~\eqref{E:B1}--\eqref{E:B5} hold on $[0,\mathcal{T})$, that is,
	\begin{align*}
		\mathcal{T} := \sup_{y\in[0,T)}\{\text{\emph{bounds} } \eqref{E:B1}-\eqref{E:B5} \text{ \emph{hold on} } [0,y)\}.
	\end{align*}
	Then $T=\mathcal{T}=\infty$.
\end{theorem}

\begin{proof}
\emph{Step 1: Show~\eqref{E:B2}--\eqref{E:B3}.}\\
From~\eqref{E:omegaPrime1}, we see that for $y\in(0,\mathcal{T})$,
\begin{align*}
	(y^3\Om)' = -y^3\om\frac{\rho'}{\rho} \geq 0,
\end{align*}
where $\rho'\leq 0$ on $(0,\mathcal{T})$ by definition of $\mathcal{T}$ and~\eqref{E:B1}. Thus, from Lemma~\ref{L:BootstrapInitial}, it easily follows that
\begin{align*}
	y^3\Om \geq \de^3\Om(\de) > 0
\end{align*}
for $\de>0$ sufficiently small. Furthermore, from~\eqref{E:rhoPrime1}, we see that
\begin{align*}
	(\log\rho)' = \frac{yh-q}{G}.
\end{align*}
By definition of $\mathcal{T}$ and Lemma~\ref{L:Supersonic}, given $\de>0$, there exists an $\epsilon>0$ such that $G\leq-\epsilon y^2<0$ on $[\de,\mathcal{T})$. Then, as $\om(y)\in[\om_0,\om_{\text{upper}}]$ from Lemma~\ref{L:omegaUpper} and $\rho(y)\in[0,\rho_0]$, we see that there exists an $M>0$ such that $|h|\leq M$ on $\big[\frac{1}{2}\mathcal{T},\mathcal{T}\big]$ and $q\geq0$. In particular, 
\begin{align*}
	(\log\rho)' \geq -\frac{M}{\epsilon y},
\end{align*}
and therefore, for any finite $\de\leq y_0\leq\mathcal{T}$, there exists a $\bar{\varrho}(y_0)>0$ such that $\rho\geq\bar{\varrho}>0$ on $[\de,y_0]$.\\

\noindent\emph{Step 2: Prove that~\eqref{E:B5} is propagated by the flow.}\\
To this end, let us assume by way of contradiction that the strict inequality in~\eqref{E:B5} fails at $y=\mathcal{T}<\infty$, that is,
\begin{align}
	\bigg[\frac{n+1}{n-2}\rho_0\Om-\om_0R\bigg]_{y=\mathcal{T}} = 0.\label{E:B5Contra1}
\end{align}
From the minimality of $\mathcal{T}$ and the identity~\eqref{E:B5ID}, it follows that the right-most term on the right-hand side of~\eqref{E:B5ID} is non-positive at $y=\mathcal{T}$. Multiplying by $-\frac{G}{y^2}$ we conclude that
\begin{align}
	\bigg[\bigg(\rho_0+\frac{2}{3}(n-2)R\bigg)\left(-\frac{G}{y^2}\right)\left(-y\frac{\rho'}{\rho}\right) - (n-2)\left(-\frac{G}{y^2}\right)R \bigg]_{y=\mathcal{T}} \leq 0.\label{E:B5Contra2}
\end{align}
Using~\eqref{E:GDef} and expanding $\om^2$ in the form
\begin{align*}
	\om^2 = \Om^2 + 2\Om + \om_0^2,
\end{align*}
we rewrite the second summand in~\eqref{E:B5Contra2} in the form
\begin{align}
	\left(-\frac{G}{y^2}\right)R &= \om^2R - \ga\Big(1-\frac{\al}{2}\Big)^2\frac{p}{y^2\rho} R\notag\\
	&= \om_0^2R - \ga\Big(1-\frac{\al}{2}\Big)^2\frac{p}{y^2\rho} R + 2\om_0R \Om + R\Om^2.\label{E:B5Contra3}
\end{align}
We also note that, due to~\eqref{E:B5Contra1}, from~\eqref{E:AB}, at $y=\mathcal{T}$,
\begin{align*}
	\mathcal{A}\Om + \mathcal{B}R = (n-2)\frac{\om_0^2}{\rho_0}R.
\end{align*}
Therefore, using~\eqref{E:rhoPrimeID},~\eqref{E:B5Contra3} and recalling the identity
\begin{align*}
	1-\frac{\al}{2} = \frac{3(n-2)(2-\ga)}{2(n+3\ga-6)},
\end{align*}
we may rewrite the left-hand side of~\eqref{E:B5Contra2} in the form
\begin{align}
	&\bigg[\bigg(\rho_0+\frac{2}{3}(n-2)R\bigg)\left(-\frac{G}{y^2}\right)\left(-y\frac{\rho'}{\rho}\right) - \left(-\frac{G}{y^2}\right)(n-2)R \bigg]_{y=\mathcal{T}}\notag\\
	&= \bigg(\rho_0+\frac{2}{3}(n-2)R\bigg)\bigg((n-2)\frac{\om_0^2}{\rho_0}R + \mathcal{C}\Om^2 + \mathcal{D}R\Om + \mathcal{E}\bigg)\notag\\
	&- (n-2)\om_0^2R + (n-2)\ga\Big(1-\frac{\al}{2}\Big)^2\frac{p}{y^2\rho}R - 2(n-2)\om_0R\Om - (n-2)R\Om^2\notag\\
	&= \bigg(\rho_0+\frac{2}{3}(n-2)R\bigg)\Big(1-\frac{\al}{2}\Big)^2\frac{n\om_0}{m_0y^2\om}\bigg(y^2\om R - m_0\frac{p}{\rho}\bigg) + (n-2)\ga\Big(1-\frac{\al}{2}\Big)^2\frac{p}{y^2\rho} R\notag\\
	&+ \frac{2}{3}(n-2)^2\frac{\om_0^2}{\rho_0}R^2 + \rho_0\mathcal{C}\Om^2 + \big(\rho_0\mathcal{D}- 2(n-2)\om_0\big)R\Om\notag\\
	&+ (n-2)\bigg(\frac{2}{3}\mathcal{C}-1\bigg) R \Om^2+\frac{2}{3} (n-2)\mathcal{D}R^2 \Om,\label{E:B5Contra4}
\end{align}
where all functions are evaluated at $y=\mathcal{T}$. To simplify the notation for the analysis of this expression, we name the first line of the right hand side by defining
\begin{align*}
	K := \bigg(\rho_0+\frac{2}{3}(n-2)R\bigg)\Big(1-\frac{\al}{2}\Big)^2\frac{n\om_0}{m_0y^2\om}\bigg(y^2\om R-m_0\frac{p}{\rho}\bigg) + (n-2)\ga\Big(1-\frac{\al}{2}\Big)^2\frac{p}{y^2\rho}R.
\end{align*}
We rewrite the expression $K$ in the form
\begin{equation}
\begin{aligned}
	K &= \Big(1-\frac{\al}{2}\Big)^2\frac{1}{m_0y^2\om}\bigg(n\om_0\bigg(\rho_0+\frac{2}{3}(n-2)R\bigg)- (n-2)\ga\om R\bigg)\bigg(y^2\om R - m_0\frac{p}{\rho}\bigg)\\
	&+ \Big(1-\frac{\al}{2}\Big)^2\frac{(n-2)\ga}{m_0}\om R^2.\label{E:K1}
\end{aligned}
\end{equation}
By~\eqref{E:B5Contra1}, we have 
\begin{align}
	\Om = \frac{n-2}{n+1}\frac{\om_0}{\rho_0} R,\quad y=\mathcal{T}.\label{E:B5Contra1Alt}
\end{align}
Using this and the decomposition $\om=\om_0+\Om$, we conclude that 
\begin{align*}
	&n\om_0\bigg(\rho_0+\frac{2}{3}(n-2)R\bigg) - (n-2)\ga\om R\\
	&= n\om_0\rho_0 + (n-2)\om_0\bigg(\frac{2}{3}n-\ga\bigg)R - \frac{(n-2)^2\ga}{n+1}\frac{\om_0}{\rho_0}R^2\\
	&= \om_0\rho_0\left(-\frac{(n-2)^2\ga}{n+1}\frac{R^2}{\rho_0^2} + (n-2)\bigg(\frac{2}{3}n-\ga\bigg) \frac{R}{\rho_0} + n\right) \geq 0,\quad y=\mathcal{T},
\end{align*}
where we have crucially used the bound~\eqref{E:QuadInq} and the fact that $\frac{R}{\rho_0}\in[0, 1]$ since $0<\rho\leq\rho_0$ by the bootstrap assumptions~\eqref{E:B1}--\eqref{E:B2}. We feed this back into~\eqref{E:K1}, again decomposing $\om=\om_0+\Om$, and together with the bootstrap assumption~\eqref{E:B4}, to conclude that 
\begin{align}
	K \geq \Big(1-\frac{\al}{2}\Big)^2\frac{(n-2)\ga}{m_0}\om_0R^2 + \Big(1-\frac{\al}{2}\Big)^2\frac{(n-2)\ga}{m_0}\Om R^2.\label{E:K2}
\end{align}
Using~\eqref{E:K2} in~\eqref{E:B5Contra4}, we obtain the bound
\begin{align}
	&\bigg[\bigg(\rho_0+\frac{2}{3}(n-2)R\bigg)\left(-\frac{G}{y^2}\right)\left(-y\frac{\rho'}{\rho}\right) - (n-2)\left(-\frac{G}{y^2}\right)R \bigg]_{y=\mathcal{T}}\notag\\
	&\geq \bigg((n-2)\ga\Big(1-\frac{\al}{2}\Big)^2\frac{\om_0}{m_0}+\frac{2}{3}(n-2)^2\frac{\om_0^2}{\rho_0}\bigg)R^2 + \rho_0\mathcal{C}\Om^2 + \big(\rho_0\mathcal{D}-2(n-2)\om_0\big)R\Om\label{E:B5Contra5}\\
	&+ (n-2)\bigg(\frac{2}{3}\mathcal{C}-1\bigg)R\Om^2 + \bigg(\frac{2}{3}(n-2)\mathcal{D} + \Big(1-\frac{\al}{2}\Big)^2\frac{(n-2)\ga}{m_0}\bigg)R^2\Om,\quad y=\mathcal{T}.\notag
\end{align}
Since $\mathcal{C}=2$ and $\mathcal{D}>0$ by~\eqref{E:CDef}--\eqref{E:DDef}, the third-order polynomials in~\eqref{E:B5Contra5} are all strictly positive. Therefore, using~\eqref{E:B5Contra1Alt} we have the lower bound
\begin{align*}
	&\bigg[\bigg(\rho_0+\frac{2}{3}(n-2)R\bigg)\left(-\frac{G}{y^2}\right)\left(-y\frac{\rho'}{\rho}\right) - (n-2)\left(-\frac{G}{y^2}\right)R \bigg]_{y=\mathcal{T}}\\
	&\geq \bigg((n-2)\ga\Big(1-\frac{\al}{2}\Big)^2\frac{\om_0}{m_0}+\frac{2}{3}(n-2)^2\frac{\om_0^2}{\rho_0} + 2 \frac{(n-2)^2}{(n+1)^2}\frac{\om_0^2}{\rho_0} + \big(\rho_0\mathcal{D}-2(n-2)\om_0\big)\frac{n-2}{n+1}\frac{\om_0}{\rho_0}\bigg)R^2\\
	&= \bigg((n-2)\ga\Big(1-\frac{\al}{2}\Big)^2\frac{\om_0}{m_0} + \frac{n-2}{n+1}\om_0\mathcal{D} + (n-2)^2\bigg(\frac{2}{3}-\frac{2n}{(n+1)^2}\bigg)\frac{\om_0^2}{\rho_0}\bigg)R^2 > 0,\quad y=\mathcal{T},
\end{align*}
which gives contradiction to~\eqref{E:B5Contra2}.\\

\noindent\emph{Step 3: Prove that~\eqref{E:B4} is propagated by the flow.}\\
Similar to Step 2, let us assume for contradiction that the strict inequality in~\eqref{E:B4} fails at $y=\mathcal{T}<\infty$, that is,
\begin{align}
	\bigg[y^2\om R - m_0\frac{p}{\rho}\bigg]_{y=\mathcal{T}} = 0.\label{E:B4Contra1}
\end{align}
Using~\eqref{E:B4Contra1} in~\eqref{E:B4ID} we obtain
\begin{align}
	\frac{d}{dy}\bigg(y^2\om R - m_0\frac{p}{\rho}\bigg) 
	&= 2y\om R + y^2\om\bigg(\frac{3\om_0-3\om}{y\om}-\frac{\rho'}{\rho}\bigg)R - y^2\rho\om\frac{\rho'}{\rho}\notag\\
	&- y^2\om\bigg((\ga-1)\frac{\rho'}{\rho} + \frac{n}{3}\bigg(\frac{3\om_0-3\om}{y\om}\bigg) + \frac{n}{y}\bigg) R\notag\\
	&= -\rho_0y^2\om\frac{\rho'}{\rho} - (n-2)y\om R\notag\\
	&+ y\bigg((n-3)\Om-(\ga-1)y\om\frac{\rho'}{\rho}\bigg)R,\quad y=\mathcal{T}.\label{E:B4Contra2}
\end{align}
From the minimality of $\mathcal{T}$, the right-hand side of~\eqref{E:B4Contra2} is non-positive, which upon multiplication by $-\frac{G}{y^3\om}$ implies the bound
\begin{align*}
	\mathcal{M} := \left(-\frac{G}{y^2}\right)\bigg(\big(\rho_0+(\ga-1)R\big)\left(-y\frac{\rho'}{\rho}\right) - (n-2)R + \frac{n-3}{\om}R\Om\bigg) \leq 0,\quad y=\mathcal{T}.
\end{align*}
Note that by~\eqref{E:GDef} and~\eqref{E:B4Contra1} we have the identity
\begin{align}
	-\frac{G}{y^2} &= \om^2 - \frac{\ga}{m_0}\Big(1-\frac{\al}{2}\Big)^2\om R\notag\\
	&= \om_0^2 + 2\om_0\Om - \frac{\ga}{m_0}\Big(1-\frac{\al}{2}\Big)^2\om_0R + \Om^2 - \frac{\ga}{m_0}\Big(1-\frac{\al}{2}\Big)^2\Om R\label{E:B4Contra3}.
\end{align}
We now use~\eqref{E:rhoPrimeID} and~\eqref{E:B4Contra3} to expand $\mathcal{M}$ into polynomial powers of $R$ and $\Om$. We thus obtain
\begin{align}
	\mathcal{M} &= \big(\rho_0+(\ga-1)R\big) \big(\mathcal{A}\Om+\mathcal{B}R+\mathcal{C}\Om^2+\mathcal{D}R\Om+\mathcal{E}\big)\notag\\
	&- (n-2)R\bigg(\om_0^2+2\om_0\Om-\frac{\ga}{m_0}\Big(1-\frac{\al}{2}\Big)^2\om_0R + \Om^2 - \frac{\ga}{m_0}\Big(1-\frac{\al}{2}\Big)^2\Om R\bigg)\notag\\
	&+ (n-3)R\Om\bigg(\om_0+\Om-\frac{\ga}{m_0}\Big(1-\frac{\al}{2}\Big)^2R\bigg)\notag\\
	&= Q_1 + Q_2 + Q_3 + \big(\rho_0+(\ga-1)R\big)\mathcal{E},\label{E:M}
\end{align}
where:
\begin{align*}
	Q_1 &:= \rho_0\big(\mathcal{A}\Om+\mathcal{B}R\big) - (n-2)\om_0^2R,\\
	Q_2 &:= (\ga-1)R\big(\mathcal{A}\Om+\mathcal{B}R\big) + \rho_0\big(\mathcal{C}\Om^2+\mathcal{D}R\Om\big) - 2(n-2)\om_0R\Om\\
	&+ (n-2)\ga\Big(1-\frac{\al}{2}\Big)^2\frac{\om_0}{m_0}R^2 + (n-3)\om_0R\Om,\\
	Q_3 &:= (\ga-1)R\big(\mathcal{C}\Om^2+\mathcal{D}R\Om\big) - (n-2)R\bigg(\Om^2 - \frac{\ga}{m_0}\Big(1-\frac{\al}{2}\Big)^2\Om R\bigg)\\
	&+ (n-3)R\Om^2 - \Big(1-\frac{\al}{2}\Big)^2\frac{(n-3)\ga}{m_0}R^2\Om.
\end{align*}
We now use~\eqref{E:AB} to infer that for any $\ga>\max\big\{\frac{19}{12},\frac{10+n}{9}\big\}$, $Q_1\geq0$. Using~\eqref{E:AB} and regrouping the terms, we have the bound
\begin{align}
	Q_2 > \bigg((n-2)(\ga-1)\frac{\om_0^2}{\rho_0}+(n-2)\ga\Big(1-\frac{\al}{2}\Big)^2\frac{\om_0}{m_0}\bigg)R^2 + \big(\rho_0\mathcal{D}-(n-1)\om_0\big)R\Om + 2\rho_0\Om^2.\label{E:Q2}
\end{align}
Noticing that
\begin{align*}
	\bigg(\frac{n+1}{n-2}\rho_0\Om-\om_0R\bigg)^2 \geq 0,
\end{align*}
we have
\begin{align*}
	2\rho_0\Om^2 \geq 4\frac{n-2}{n+1}\om_0R\Om - 2\frac{(n-2)^2}{(n+1)^2}\frac{\om^2_0}{\rho_0}R^2.
\end{align*}
Applying this inequality in~\eqref{E:Q2}, along with \begin{align*}
	\ga(n-2)\Big(1-\frac{\al}{2}\Big)^2\frac{\om_0}{m_0}R^2 \geq 0,
\end{align*}
we obtain
\begin{align*}
	Q_2 > \bigg((n-2)(\ga-1) - 2\frac{(n-2)^2}{(n+1)^2}\bigg)\frac{\om_0^2}{\rho_0}R^2 + \bigg(\frac{(n-2)^2}{2(3\ga-4)^2} + 4\frac{n-2}{n+1} - (n-1)\bigg)\om_0R\Om,
\end{align*}
where we have used the fact that
\begin{align}
	\rho_0\mathcal{D} = \frac{2}{9\om_0}\bigg(\frac{3(n-2)(2-\ga)}{2(n+3\ga-6)}\bigg)^2 = \frac{(n-2)^2}{2(3\ga-4)^2}\om_0.\label{E:DSimple}
\end{align}
We next apply the inequality of Lemma~\ref{L:Free} to further bound this from below as
\begin{align}
	Q_2 > q_2(n,\ga) \om_0R\Om - \bigg((n-2)(\ga-1) - 2\frac{(n-2)^2}{(n+1)^2}\bigg)\frac{\om_0}{\rho_0}R^2\Om,\label{E:Q2Alt}
\end{align}
where
\begin{align*}
	q_2(n,\ga) = (n-2)(\ga-1) - 2\frac{(n-2)^2}{(n+1)^2} + \frac{(n-2)^2}{2(3\ga-4)^2} + 4\frac{n-2}{n+1} - (n-1).
\end{align*}
To see that $q_2$ is positive, we first differentiate with respect to $\ga$ to obtain
\begin{align*}
	\frac{\pa q_2}{\pa\ga}(n,\ga) = (n-2)\bigg(1-\frac{3(n-2)}{(3\ga-4)^3}\bigg) < 0
\end{align*}
for all $n\geq4$ and $\ga\in\big(\frac{19}{12},\frac{11}{6}\big)$. Similarly, differentiating $q_2$ with respect to $n$, we obtain
\begin{align*}
	\frac{\pa q_2}{\pa n}(n,\ga) = \ga-2 + \frac{36}{(n+1)^3} + \frac{n-2}{(4-3\ga)^2} > 0
\end{align*}
for all $n\geq4$ and $\ga\in\big(\frac{19}{12},\frac{11}{6}\big)$. Thus, for $n\geq4$ and $\ga\in\big(\frac{19}{12},\frac{11}{6}\big)$, we have
\begin{align*}
	q_2(n,\ga) \geq q_2\Big(\frac{11}{6},4\Big) > 0.
\end{align*}
Before combining $Q_2$ with $Q_3$, we first apply~\eqref{E:DSimple} to simplify $Q_3$ as
\begin{align*}
	Q_3 = \big(2(\ga-1)-1\big)R\Om^2 + \frac{(n-2)^2(\ga-1)}{2(3\ga-4)^2}\frac{\om_0}{\rho_0}R^2\Om + \frac{\ga}{m_0}\Big(1-\frac{\al}{2}\Big)^2R^2\Om.
\end{align*} 
Now, combining this identity and~\eqref{E:Q2Alt} with $\mathcal{E}\geq0$ and $Q_1\geq0$, and recalling~\eqref{E:M}, we have 
\begin{align*}
	\mathcal{M} &= Q_1 + Q_2 + Q_3 + \big(\rho_0+(\ga-1)R\big)\mathcal{E}\\
	&> \bigg((n-2)(\ga-1) - 2\frac{(n-2)^2}{(n+1)^2} + \frac{(n-2)^2}{2(3\ga-4)^2} + 4\frac{n-2}{n+1} - (n-1)\bigg)\om_0R\Om\\
	&- \bigg((n-2)(\ga-1) - 2\frac{(n-2)^2}{(n+1)^2}\bigg)\frac{\om_0}{\rho_0}R^2\Om\\
	&+ \big(2(\ga-1)-1\big)R\Om^2 + \frac{(n-2)^2(\ga-1)}{2(3\ga-4)^2}\frac{\om_0}{\rho_0}R^2\Om + \frac{\ga}{m_0}\Big(1-\frac{\al}{2}\Big)^2R^2\Om\\
	&> \bigg(\frac{(n-2)^2}{2(3\ga-4)^2} + 4\frac{n-2}{n+1} - (n-1) + \frac{n-2}{n+1}\big(2(\ga-1)-1\big) + \frac{(n-2)^2(\ga-1)}{2(3\ga-4)^2}\bigg)\frac{\om_0}{\rho_0}R^2\Om,
\end{align*}
where we have used $q_2\geq0$, the bound
\begin{align*}
	\om_0R\Om \geq \frac{\om_0}{\rho_0}R^2\Om
\end{align*}
and~\eqref{E:B5}, that is,
\begin{align*}
	R\Om^2 \geq \frac{n-2}{n+1}\frac{\om_0}{\rho_0} R^2\Om.
\end{align*}
We set
\begin{align*}
	q_3(n,\ga) = \frac{(n-2)^2}{2(3\ga-4)^2} + 4\frac{n-2}{n+1} - (n-1) + \frac{n-2}{n+1}\big(2(\ga-1)-1\big) + \frac{(n-2)^2(\ga-1)}{2(3\ga-4)^2}.
\end{align*}
It remains to show that this quantity is positive. Differentiating $q_3$ with respect to $\ga$ we obtain
\begin{align*}
	\frac{\pa q_3}{\pa\ga}(n,\ga) = \frac{1}{2}(n-2)\bigg(\frac{4}{n+1}-\frac{n-2}{(3\ga-4)^2}-\frac{8(n-2)}{(3\ga-4)^3}\bigg) < 0
\end{align*}
for all $n\geq4$ and $\ga\in\big(\frac{19}{12},\frac{11}{6}\big)$. Similarly, differentiating $q_3$ with respect to $n$ we obtain
\begin{align*}
	\frac{\pa q_3}{\pa n}(n,\ga) = \bigg(1 + \frac{4}{3\ga-4}\bigg)\frac{n-2}{3(3\ga-4)}+ \frac{6\ga+3}{(n+1)^2} - 1 > 0
\end{align*}
for all $n\geq4$ and $\ga\in\big(\frac{19}{12},\frac{11}{6}\big)$. Thus, for $n\geq4$ and $\ga\in\big(\frac{19}{12},\frac{11}{6}\big)$, we have
\begin{align*}
	q_3(n,\ga) \geq q_3\Big(\frac{11}{6},4\Big) > 0.
\end{align*}
We therefore obtain $\mathcal{M}>0$ at $y=\mathcal{T}$ and hence a contradiction.
\end{proof}

\begin{corollary}\label{C:rhoPrimeLower}
	Let $n=4,6$ and $\ga>\max\big\{\frac{19}{12},\frac{10+n}{9}\big\}$. Consider the unique global-in-$y$ solution to the initial value problem~\eqref{E:rhoPrime1}--\eqref{E:omegaPrime1} given by Theorem~\ref{T:GE}. Then for all $y>0$ the following bound holds
	\begin{align}
		\left(-\frac{G}{y^2}\right)\left(-y\frac{\rho'}{\rho}\right) > (n-2)\frac{\om_0^2}{\rho_0}R + 2 \Om^2 + \frac{4\pi}{3\om_0}\Big(1-\frac{\al}{2}\Big)^2 R \Om
	\end{align}
\end{corollary}

\begin{proof}
The claim is a trivial consequence of the identity~\eqref{E:rhoPrimeID} and the lower bound~\eqref{E:AB}.
\end{proof}

\subsection{Monotonicity, Sharp Bounds and Limiting Values of $\rho$ and $\om$}

In the next proposition we show that the global-in-$y$ self-similar density $\rho(y)$ generated by Theorem~\ref{T:GE} converges to zero as $y\to\infty$, which is consistent with the behaviour of the density associated with the exact far-field solution (see Definition~\ref{D:FarField}).

\begin{proposition}[Asymptotic value of $\rho$]\label{P:rhoLim}
	Consider the unique global-in-$y$ solution to the initial value problem~\eqref{E:rhoPrime1}--\eqref{E:omegaPrime1} given by Theorem~\ref{T:GE}. Then 
	\begin{align}
		\lim_{y\to\infty}\rho(y) = 0.\label{E:rhoLim}
	\end{align}
\end{proposition}

\begin{proof}
Since $\rho'(y)<0$ for all $y>0$ by Theorem~\ref{T:GE}, we assume for contradiction that
\begin{align}
	\lim_{y\to\infty}\rho(y) = \rho_\infty > 0\label{E:rhoLimContra}
\end{align}
for some $0<\rho_\infty<\rho_0$. From Corollary~\ref{C:rhoPrimeLower} we have
\begin{align*}
	\rho' < -\frac{\rho}{y}\left(-\frac{y^2}{G}\right)\bigg((n-2)\frac{\om_0^2}{\rho_0}R + 2\Om^2 + \frac{4\pi}{3\om_0}\Big(1-\frac{\al}{2}\Big)^2 R \Om\bigg).
\end{align*}
Thus, for $y$ sufficiently large so that
\begin{align*}
	\rho(y) < \frac{\rho_\infty+\rho_0}{2},
\end{align*}
we have
\begin{align*}
	\rho' &< -\frac{\rho_\infty}{y}\bigg(\frac{1}{\om_{\text{upper}}^2}\bigg)\bigg((n-2)\frac{\om_0^2}{\rho_0}\frac{\rho_0-\rho_\infty}{2} + 2(\om-\om_0)^2 + \frac{4\pi}{3\om_0}\Big(1-\frac{\al}{2}\Big)^2\frac{\rho_0-\rho_\infty}{2}(\om-\om_0)\bigg)\\
	&< -\frac{\beta}{y},
\end{align*}
for some constant $\beta>0$. Therefore $(y^\beta\rho)'<0$, which is in contradiction with the claim~\eqref{E:rhoLimContra}. This proves~\eqref{E:rhoLim}.
\end{proof}

\begin{proposition}[Monotonicity]\label{P:Mon}
	Consider the unique global-in-$y$ solution to the initial value problem~\eqref{E:rhoPrime1}--\eqref{E:p} given by Theorem~\ref{T:GE}. Let 
	\begin{align*}
		y_0 = \sup\{y > 0\ |\ \om(\tilde{y}) < 2-\ga\ \text{ \emph{for all} }\ \tilde{y}\in(0,y)\},
	\end{align*}
	then $\om'(y) > 0$ for all $y\in(0,y_0)$.
\end{proposition}

\begin{proof}
It follows from Theorem~\ref{T:LWP} that there exists a $\de\in(0,y_0)$ such that $\om'(y)>0$ for $y\in(0,\de)$. Now suppose for contradiction that there exists a $\hat{y}\in(0,y_0)$ such that $\om'(\hat{y})=0$. We will prove for contradiction that $\om''(\hat{y})>0$.

First, we note that
\begin{align}
	\om' &= -\frac{3(\om-\om_0)}{y} - \frac{\om(yh-q)}{G}\notag\\
	&= -\frac{1}{G}\left(-\frac{3(\om-\om_0)}{y}\bigg(y^2\om^2 - \ga\Big(1-\frac{\al}{2}\Big)^2\frac{p}{\rho}\bigg) + y\om h - \om q\right)\notag\\
	&= -\frac{y\om}{G}\left(h - 3(\om-\om_0)\om + \big(3\ga(\om-\om_0)-n\om_0\big)\Big(1-\frac{\al}{2}\Big)^2\frac{p}{y^2\rho\om}\right)\notag\\
	&= -\frac{y\om}{G}\big(A(\om) + B(\rho,\om) + P(\rho,\om,p)\big),\label{E:omegaPrimeID}
\end{align}
where:
\begin{align}
	A(\om) &:= 2\om^2+\Big(\ga-1-\frac{\al}{2}\Big)\om+\Big(\ga-1-\frac{\al}{2}\Big)(2-\ga) - 3(\om-\om_0)\om,\label{E:ADef}\\
	B(\rho,\om) &:= -\frac{2}{9}\Big(1-\frac{\al}{2}\Big)^2\frac{\rho\om}{\rho_0\om_0},\label{E:BDef}\\
	P(\rho,\om,p) &:= \Big(1-\frac{\al}{2}\Big)^2f(\om)\frac{p}{y^2\rho},\label{E:PDef}
\end{align}
and
\begin{align*}
	f(\om) := \frac{3\ga(\om-\om_0)-n\om_0}{\om}.
\end{align*}
Since $G(y)<0$ for all $y>0$, we see that $\om'(\hat{y})=0$ implies
\begin{align}
	A(\hat{\om}) + B(\hat{\rho},\hat{\om}) + P(\hat{\rho},\hat{\om},\hat{p}) = 0,\label{E:ABP}
\end{align}
where to simplify notation, we define
\begin{align}
	\hat{\om} &:= \om(\hat{y}), & \hat{\rho} &:= \rho(\hat{y}), & \hat{p} &:= p(\hat{y}).
\end{align}
Thus
\begin{align*}
	\om''(\hat{y}) = -\frac{\hat{y}\hat{\om}}{G(\hat{y})}\big(B'(\hat{\rho},\hat{\om}) + P'(\hat{\rho},\hat{\om},\hat{p})\big)
\end{align*}
and it remains to show this is strictly positive for $\hat{\om}\in(\om_0,2-\ga)$. First, from~\eqref{E:rhoPrime1}--\eqref{E:omegaPrime1}, we have
\begin{align*}
	\frac{d}{dy}(\rho\om)\bigg|_{y=\hat{y}} = \bigg(\frac{\rho'}{\rho}+\frac{\om'}{\om}\bigg)\rho\om\bigg|_{y=\hat{y}} = -\frac{3(\hat{\om}-\om_0)}{\hat{y}\hat{\om}}\hat{\rho}\hat{\om},
\end{align*}
so that
\begin{align*}
	B'(\hat{\rho},\hat{\om}) = -\frac{3(\hat{\om}-\om_0)}{\hat{y}\hat{\om}}B(\hat{\rho},\hat{\om}).
\end{align*}
Moreover, from~\eqref{E:p},
\begin{align*}
	\frac{d}{dy}\bigg(\frac{p}{y^2\rho}\bigg)\bigg|_{y=\hat{y}} &= \bigg[\bigg(\frac{n-2}{y} + (\ga-1)\frac{\rho'}{\rho} + \frac{n}{3}\frac{(\rho\om)'}{\rho\om}\bigg)\frac{p}{y^2\rho}\bigg]_{y=\hat{y}}\\
	&= \frac{1}{y\om}\big((n-2)\hat{\om} - (n+3\ga-3)(\hat{\om}-\om_0)\frac{\hat{p}}{\hat{y}^2\hat{\rho}}\\
	&= \frac{1}{\hat{y}\hat{\om}}\big((3\ga-4)(2-\ga-\hat{\om}) - 3(\hat{\om}-\om_0)\big)\frac{\hat{p}}{\hat{y}^2\hat{\rho}},
\end{align*}
so that
\begin{align*}
	P'(\hat{\rho},\hat{\om},\hat{p}) = \bigg(\frac{(3\ga-4)(2-\ga-\hat{\om})}{\hat{y}\hat{\om}} - \frac{3(\hat{\om}-\om_0)}{\hat{y}\hat{\om}}\bigg)P(\hat{\rho},\hat{\om},\hat{p}).
\end{align*}
Thus, combining these derivatives and applying \eqref{E:ABP}, we obtain
\begin{align}
	-G(\hat{y})\om''(\hat{y}) &= \hat{y}\hat{\om}\big(B'(\hat{\rho},\hat{\om}) + P'(\hat{\rho},\hat{\om},\hat{p})\big)\notag\\
	&= 3(\hat{\om}-\om_0)A(\hat{\om}) + (3\ga-4)(2-\ga-\hat{\om})P(\hat{\rho},\hat{\om},\hat{p})\notag\\
	&= 3(\hat{\om}-\om_0)A(\hat{\om}) + (3\ga-4)(2-\ga-\hat{\om})\Big(1-\frac{\al}{2}\Big)^2f(\hat{\om})\frac{\hat{p}}{\hat{y}^2\hat{\rho}}.\label{E:omegaPrimePrime}
\end{align}
Thus if $f(\hat{\om})\geq0$, then we obtain $\om''(\hat{y})>0$ and derive the claimed contradiction. It is simple to see that $f(\om)\geq0$ is equivalent to $\om\geq\big(1+\frac{n}{3\ga}\big)\om_0$ and a calculation based on~\eqref{E:alpha(n)} shows that
\begin{align*}
	\om_0 < \bigg(1+\frac{n}{3\ga}\bigg)\om_0 = \frac{(3\ga-4)(2-\ga)(n+3\ga)}{3\ga(n+3\ga-6)} = \bigg(1-\frac{4n-6\ga}{3\ga(n+3\ga-6)}\bigg)(2-\ga) < 2-\ga
\end{align*}
for $n\geq4$ and $\ga\in\big(\frac{19}{12},\frac{11}{6}\big)$. For the remainder of the proof we thus assume $\om_0<\hat{\om}<\big(1+\frac{n}{3\ga}\big)\om_0$ so that $f(\hat{\om})<0$.

We now apply~\eqref{E:B4} and~\eqref{E:B5} to obtain from~\eqref{E:omegaPrimePrime}
\begin{align*}
	-G(\hat{y})\om''(\hat{y}) &> 3(\hat{\om}-\om_0)A(\hat{\om}) + (3\ga-4)(2-\ga-\hat{\om})\Big(1-\frac{\al}{2}\Big)^2f(\hat{\om})\frac{n+1}{n-2}\frac{\rho_0}{\om_0}\frac{\hat{\om}}{m_0}(\hat{\om}-\om_0)\\
	&= \bigg(3\bar{A}(\hat{\om}) + (3\ga-4)\Big(1-\frac{\al}{2}\Big)^2f(\hat{\om})\frac{n+1}{n-2}\frac{\rho_0}{\om_0}\frac{\hat{\om}}{m_0}\bigg)(2-\ga-\hat{\om})(\hat{\om}-\om_0)\\
	&=: F(\hat{\om};n,\ga)(2-\ga-\hat{\om})(\hat{\om}-\om_0),
\end{align*}
where $\bar{A}$ is defined as
\begin{align}
	\bar{A}(\om) := (2-\ga-\om)^{-1}A(\om) = \om - \bigg(1 - \frac{n-2}{2(3\ga-4)}\bigg)\om_0.\label{E:BarADef}
\end{align}
Simplifying $F$, we recall~\eqref{E:m0Def} to see
\begin{align*}
	F(\hat{\om};n,\ga) = 3\bar{A}(\hat{\om}) + \frac{3(n-2)(\ga-1)(11-6\ga)}{2n(3\ga-4)}\big(3\ga(\hat{\om}-\om_0)-n\om_0\big),
\end{align*}
where we have also used~\eqref{E:alpha(n)} to see
\begin{align*}
	1-\frac{\al}{2} = \frac{3(n-2)}{2(3\ga-4)}\om_0.
\end{align*}
Since $F$ is linearly increasing in $\om$, we thus have
\begin{align*}
	F(\hat{\om};n,\ga) \geq F(\om_0;n,\ga) = \frac{3(n-2)}{2(3\ga-4)}\big(1-(\ga-1)(11-6\ga)\big)\om_0 > 0
\end{align*}
for all $\ga\in\big(\frac{19}{12},\frac{11}{6}\big)$. Therefore $\om''(\hat{y})>0$ whenever $\om'(\hat{y})=0$. This concludes the proof.
\end{proof}

With this monotonicity proposition in hand, we now prove that the global-in-$y$ self-similar relative velocity $\om(y)$ generated by Theorem~\ref{T:GE} remains bounded above by $2-\ga$ and converges monotonically to this limiting value as $y\to\infty$. This is consistent with the expectation that the self-similar solution coincides in the far-field limit with the velocity associated with the exact far-field solution (compare Definition~\ref{D:FarField}).

\begin{proposition}[Sharp upper bound on $\om$]\label{P:omegaBoundSharp}
	Let $(\rho,\om)$ be the unique global-in-$y$ solution to the initial value problem~\eqref{E:rhoPrime1}--\eqref{E:omegaPrime1} given by Theorem~\ref{T:GE}. Then, for all $y\in(0,\infty)$, we have
	\begin{align}
		\om_0<\om(y)<2-\ga.
	\end{align}
\end{proposition}

\begin{proposition}[Asymptotic value of $\om$]\label{P:omegaLim}
	Let $(\rho,\om)$ be the unique global-in-$y$ solution to the initial value problem~\eqref{E:rhoPrime1}--\eqref{E:omegaPrime1} given by Theorem~\ref{T:GE}. Then we have
	\begin{align}
		\lim_{y\to\infty}\om(y) = 2-\ga.
	\end{align}
\end{proposition}

Before proving these propositions, we first establish two auxiliary properties of the relative velocity $\om$ through Lemmas~\ref{L:Aux1} and~\ref{L:Aux2}.

\begin{lemma}\label{L:Aux1}
	Let $(\rho,\om)$ be the unique global-in-$y$ solution to the initial value problem~\eqref{E:rhoPrime1}--\eqref{E:omegaPrime1} given by Theorem~\ref{T:GE}. Then, setting
	\begin{align}
		\om_1 := \bigg(1+\frac{n}{3\ga}\bigg)\om_0,\label{E:omega1Def}
	\end{align}
	 there exists a $y_1>0$ such that $\om(y)\in(\om_0,\om_1)$ for $y\in(0,y_1)$, $\om(y_1)=\om_1$ and $\om'(y)>0$ for $y\in(0,y_1]$.
\end{lemma}

\begin{proof}
Observe from Proposition~\ref{P:Mon} that it is sufficient to ensure that there exists a first point $y=y_1$ such that $\om(y)=\om_1$, as $\om(y)$ remains strictly monotone on the interval $(0,y_1]$ due to $\om_1\in(\om_0,2-\ga)$. To prove the claim, we thus suppose for a contradiction that $\om(y)<\om_1$ for all $y>0$ and recall from Proposition~\ref{P:Mon} that this ensures $\om'(y)>0$ for all $y>0$.

We begin by recalling from~\eqref{E:omegaPrimeID} that
\begin{align*}
	\om' = -\frac{y\om}{G}\big(A(\om) + B(\rho,\om) + P(\rho,\om,p)\big),
\end{align*}
where $A$, $B$, $P$ are defined in~\eqref{E:ADef}--\eqref{E:PDef} and, moreover, from the definitions of $P$ and $\om_1$, we have
\begin{align*}
	P = \Big(1-\frac{\al}{2}\Big)^2f(\om)\frac{p}{y^2\rho},
\end{align*}
where $f(\om)<0$ for $\om\in(\om_0,\om_1)$. Therefore, applying the bound~\eqref{E:B4} for $P$ and using the fact that $\lim_{y\to\infty}B=0$ from Proposition~\ref{P:rhoLim}, we obtain
\begin{align}
	\om' 
	&\geq \frac{1}{y\om}\bigg(A(\om) + \Big(1-\frac{\al}{2}\Big)^2\frac{\rho_0}{m_0}\om f(\om)\bigg) + o_{y\to\infty}\bigg(\frac{1}{y}\bigg),\label{E:omegaPrimeLower}
\end{align}
where we have also used $y^2\om^2>-G>0$. We define the quadratic polynomial
\begin{align*}
	q(\om) :=&\ A(\om) + \Big(1-\frac{\al}{2}\Big)^2\frac{\rho_0}{m_0}\om f(\om)\\
	=&\ -\om^2+\bigg(\ga-1-\frac{\al}{2}+3\om_0+3\ga\Big(1-\frac{\al}{2}\Big)^2\frac{\rho_0}{m_0}\bigg)\om + \Big(\ga-1-\frac{\al}{2}\Big)(2-\ga)\\
	-&\ (n+3\ga)\Big(1-\frac{\al}{2}\Big)^2\frac{\rho_0\om_0}{m_0}.
\end{align*}
Noting that $q$ attains its minimum over the set $[\om_0,\om_1]$ at either $\om_0$ or $\om_1$, we compute
\begin{align*}
	q(\om_1) = A(\om_1) = (2-\ga-\om)\bigg(\frac{n}{3\ga} + \frac{n-2}{2(3\ga-4)}\bigg)\om_0 > 0,
\end{align*}
where we have used $f(\om_1)=0$ to simplify the expression, and then 
\begin{align*}
	q(\om_0) = \frac{(n-2)^2}{2(n+1)(3\ga-4)^2}\big((n-4)+2(3\ga-4)^2+3(2-\ga)\big)\om_0^2 > 0,
\end{align*}
to see that there exists a $\bar{q}>0$ such that $q\geq\bar{q}$ on $[\om_0,\om_1]$. Thus, from~\eqref{E:omegaPrimeLower},
\begin{align*}
	\om' \geq \frac{\bar{q}}{\om_{\text{upper}}}\frac{1}{y} + o_{y\to\infty}\bigg(\frac{1}{y}\bigg),
\end{align*}
which after integration yields a contradiction to the assumption $\om(y)<\om_1$ for all $y>0$, thus proving the claim.
\end{proof}

\begin{lemma}\label{L:Aux2}
	Let $(\rho,\om)$ be the unique global-in-$y$ solution to the initial value problem~\eqref{E:rhoPrime1}--\eqref{E:omegaPrime1} given by Theorem~\ref{T:GE} and let $y_0$, $y_1$ and $\om_1$ be defined as in Proposition~\ref{P:Mon} and Lemma~\ref{L:Aux1}. Then there exists an $\om_*\in(\om_1,2-\ga)$ such that for all $y\in(0,y_0)$, whenever $\om(y)\in(\om_*,2-\ga)$, we have
	\begin{align}
		\frac{p}{y^2\rho} \leq \frac{1}{2}(2-\ga-\om).\label{E:pUpper}
	\end{align}
\end{lemma}

\begin{proof}
\emph{Step 1: Defining $\om_*$.}\\
We begin by first introducing the auxiliary notation
\begin{align}
	\om_d := \bigg(1+\frac{n-2}{3\ga-1}\bigg)\om_0,\label{E:omegadDef}
\end{align}
and then defining $\om_*$ by
\begin{align}
	\om_* := 2-\ga - \frac{2\rho_0\om_d}{m_0}.\label{E:omegaAstDef}
\end{align}
To show that
\begin{align}
	\om_1 < \om_* < 2-\ga\label{E:omegaAstBounds}
\end{align}
as claimed in the statement of the lemma, we first note that as $\om_d>0$, the upper bound $\om_*<2-\ga$ is obvious. To prove the lower bound, we expand $\om_d$, $\om_1$, $\rho_0$ and $m_0$ (defined in~\eqref{E:m0Def}) according to their definitions to find that this is equivalent to proving
\begin{align*}
	L(n,\ga) := \frac{2\ga(\ga-1)(n+3\ga-6)(n+3\ga-3)}{2n-3\ga} < \frac{(n+1)n(3\ga-1)(2-\ga)}{11-6\ga} =: R(n,\ga).
\end{align*}
It is clear that for $n\geq4$ and $\ga\in\big(\frac{19}{12},\frac{11}{6}\big)$, both $L$ and $R$ are strictly increasing with respect to $\ga$. Tedious, direct calculation then shows the three inequalities:
\begin{align*}
L\Big(4,\frac{21}{12}\Big) &< R\Big(4,\frac{19}{12}\Big), & L\Big(4,\frac{11}{6}\Big) &< R\Big(4,\frac{21}{12}\Big), & L\Big(6,\frac{11}{6}\Big) &< R\Big(6,\frac{19}{12}\Big),
\end{align*}
which combine to prove~\eqref{E:omegaAstBounds} for $n=4,6$ and $\ga\in\big(\frac{19}{12},\frac{11}{6}\big)$, as required.\\

\noindent\emph{Step 2: Prove~\eqref{E:pUpper} for $\om(y)\in[\om_0,\om_*]$.}\\
We compute from~\eqref{E:rhoPrime1}--\eqref{E:p},
\begin{align}
	\frac{d}{dy}\bigg(\frac{p}{y^2\rho}\bigg) &= \bigg(\frac{n-2}{y} + (\ga-1)\frac{\rho'}{\rho} + \frac{n}{3}\frac{(\rho\om)'}{\rho\om}\bigg)\frac{p}{y^2\rho}\notag\\
	&= \frac{1}{y\om}\big((n+3\ga-3)\om_0 - (3\ga-1)\om - (\ga-1)y\om'\big)\frac{p}{y^2\rho}\label{E:pDiv1}\\
	&= \frac{1}{y\om}\big((3\ga-4)(2-\ga-\om) - 3(\om-\om_0) - (\ga-1)y\om'\big)\frac{p}{y^2\rho},\label{E:pDiv2}
\end{align}
where the last line follows from~\eqref{E:alpha(n)}. Recalling that $\om'(y)>0$ for $y\in(0,y_0)$ by Proposition~\ref{P:Mon}, we see from~\eqref{E:pDiv1} that $\frac{p}{y^2\rho}$ is strictly decreasing for all $y\in(0,y_0)$ such that $\om(y)\geq\om_d$ and we note that $\om_d<\om_1$ by their respective definitions,~\eqref{E:omega1Def} and~\eqref{E:omegadDef}. Thus, employing~\eqref{E:B4}, which gives
\begin{align}
	\frac{p}{y^2\rho} < \frac{\rho_0}{m_0}\om,
\end{align}
we obtain
\begin{align*}
 \frac{p}{y^2\rho} < \frac{\rho_0\om_d}{m_0}
\end{align*}
for all $y\in(0,y_0)$. Thus, by the definition~\eqref{E:omegaAstDef} of $\om_*$, for all $y\in(0,y_0)$ such that $\om(y)\leq\om_*$, we have shown
\begin{align*}
	\frac{p}{y^2\rho} < \frac{1}{2}(2-\ga-\om_*) \leq \frac{1}{2}\big(2-\ga-\om(y)\big),
\end{align*}
that is, \eqref{E:pUpper} holds strictly for all such $y$.\\

\noindent\emph{Step 3: Prove that~\eqref{E:pUpper} holds for $\om(y)\in(\om_*,2-\ga)$.}\\
We now suppose that there exists a $\bar{y}\in(y_1,y_0)$, where $y_1$ is as defined in Lemma~\ref{L:Aux1}, such that $\om(\bar{y})\in(\om_*,2-\ga)$ and
\begin{align*}
	\frac{p(\bar{y})}{\bar{y}^2\rho(\bar{y})} = \frac{1}{2}\big(2-\ga-\om(\bar{y})\big).
\end{align*}
Returning to~\eqref{E:pDiv2}, we substitute~\eqref{E:omegaPrimeID} to see
\begin{multline}
	\frac{d}{dy}\bigg[\frac{p}{y^2\rho}-\frac{1}{2}(2-\ga-\om)\bigg]_{y=\bar{y}} = -\frac{y\om}{2G}(A + B + P)\bigg|_{y=\bar{y}}\\
	+ \frac{2-\ga-\om}{2y\om}\bigg((3\ga-4)(2-\ga-\om) - 3(\om-\om_0) + (\ga-1)\frac{y^2\om}{G}(A + B + P)\bigg)\bigg|_{y=\bar{y}}.\label{E:pUpperDiv}
\end{multline}
We observe that as $\ga>\frac{10+n}{9}>\frac{18+n}{15}$, we obtain from the definition~\eqref{E:omega1Def} of $\om_1$ that for all $\om\geq\om_1$, we have
\begin{align*}
	\om > (2-\ga-\om)(\ga-1).
\end{align*}
This implies in particular that
\begin{align*}
	-\frac{y\om}{2G} > -\frac{2-\ga-\om}{2y\om}(\ga-1)\frac{y^2\om}{G},
\end{align*}
and so, as $B<0$ by~\eqref{E:BDef}, we obtain from~\eqref{E:pUpperDiv} that
\begin{multline*}
	\frac{d}{dy}\bigg(\frac{p}{y^2\rho}-\frac{1}{2}(2-\ga-\om)\bigg)\bigg|_{y=\bar{y}} \leq -\frac{y\om}{2G}\big(A+P\big)\bigg|_{y=\bar{y}}\\
	+ \frac{2-\ga-\om}{2y\om}\bigg((3\ga-4)(2-\ga-\om) - 3(\om-\om_0) + (\ga-1)\frac{y^2\om}{G}\big(A + P\big)\bigg)\bigg|_{y=\bar{y}}.
\end{multline*}
Using the definition of $\bar{y}$, we see that $P$ simplifies at $\bar{y}$ as
\begin{align*}
	P(\rho,\om,p)\big|_{y=\bar{y}} = \Big(1-\frac{\al}{2}\Big)^2f(\om)\frac{2-\ga-\om}{2}\bigg|_{y=\bar{y}}.
\end{align*}
Therefore, we have obtained
\begin{align*}
	\frac{d}{dy}&\bigg(\frac{p}{y^2\rho}-\frac{1}{2}(2-\ga-\om)\bigg)\bigg|_{y=\bar{y}} \geq -\frac{y}{2\om G}(2-\ga-\om)H(\om),
\end{align*}
where, recalling $\bar{A}(\om)$ from~\eqref{E:BarADef}, $H$ is defined by
\begin{equation}
\begin{aligned}
	H(\om) &:= \om\big(\ga\om-(\ga-1)(2-\ga)\big)\bigg(\bar{A}(\om) + \frac{1}{2}\Big(1-\frac{\al}{2}\Big)^2f(\om)\bigg)\\
	&+ \big((3\ga-4)(2-\ga-\om) - 3(\om-\om_0)\big)\bigg(\om^2-\frac{\ga}{2}\Big(1-\frac{\al}{2}\Big)^2(2-\ga-\om)\bigg).\label{E:HDef}
\end{aligned}
\end{equation}

As $G<0$, in order to derive a contradiction and show~\eqref{E:pUpper}, it is sufficient to show $H(\om)<0$ for $\om\in(\om_*,2-\ga)$. We begin by noting that $H(\om)$ is cubic in $\om$ with the coefficient of $\om^3$ given by $-(2\ga-1)$ and so, to prove the desired inequality, it suffices to show that $H(0)<0$, $H(\om_d)>0$ and $H(\om_*)<0$, as we then obtain $H(\om)<0$ for all $\om\geq\om_*$. 

We have, substituting the definitions of $\om_0$ and $\al$ from~\eqref{E:alpha(n)},
\begin{align*}
	H(0) &= -\frac{1}{2}n(2-\ga)\Big(1-\frac{\al}{2}\Big)^2\om_0 < 0
\end{align*}
and, recalling $\om_d$ from~\eqref{E:omegadDef}, we also have
\begin{align*}
	H(\om_d) &= F_1(n,\ga)F_2(n,\ga)\om_0^2,
\end{align*}
where
\begin{align*}
	F_1(n,\ga) &:= \frac{(2-\ga)(9\ga^2-15\ga+6-n)}{(3\ga-1)(n+3\ga-6)},\\
	F_2(n,\ga) &:= \bigg(1+\frac{n-2}{3\ga-1}\bigg)\bigg(\frac{n-2}{3\ga-1} + \frac{n-2}{2(3\ga-4)}\bigg) - \frac{9(n-2)^2(2-\ga)(6\ga-n)}{8(3\ga-1)(3\ga-4)(n+3\ga-6)}.
\end{align*}
Clearly the denominator of $F_1$ is positive, and one may easily check for both $n=4$ and $n=6$ that the numerator is also positive for $\ga\in\big(\frac{19}{12},\frac{11}{6}\big)$, so that we obtain $F_1>0$. To see the positivity of $F_2(n,\ga)$, we note that
\begin{align*}
	F_2(n,\ga) = \frac{9(n-2)}{8(3\ga-1)^2(3\ga-4)(n+3\ga-6)}\bar{F}_2(n,\ga),
\end{align*}
where
\begin{align*}
	\bar{F}_2(n,\ga) := 4(\ga-1)(n+3\ga-3)(n+3\ga-6) - (n-2)(3\ga-1)(2-\ga)(6\ga-n)
\end{align*}
and observe:
\begin{align*}
	\bar{F}_2(4,\ga) &= 4(3\ga-2)\big((3\ga+1)(\ga-1) - (3\ga-1)(2-\ga)\big) > 0\\
	\bar{F}_2(6,\ga) &= 12(\ga-1)\big(3(\ga+1)\ga - 2(3\ga-1)(2-\ga)\big) > 0
\end{align*}
in the interval $\ga\in\big(\frac{19}{12},\frac{11}{6}\big)$. Thus $H(\om_d)>0$. As discussed above, it is therefore sufficient to show $H(\om_*)<0$ in order to show $H(\om)<0$ in the full interval $\om\in(\om_*,2-\ga)$. This is done using interval arithmetic, a rigorous computer-assisted form of proof, in Section~\ref{SS:A1}. An alternative approach, based on a further decomposition of $H(\om_*)$ and a fine splitting of the interval of $\ga$ is sketched in Section~\ref{SS:A2}. This completes the proof of the lemma.
\end{proof}

\begin{proof}[Proof of Proposition~\ref{P:omegaBoundSharp}]
Let $y_0$ and $\om_1$ be defined as in Proposition~\ref{P:Mon} and Lemma~\ref{L:Aux1} respectively.

If $\om(y)\leq\om_*$ for all $y>0$, we are done as $\om_*<2-\ga$. Otherwise, there exists a $y_2>0$ such that for all $y\in(y_2,y_0)$, we have $\om(y)\in(\om_*,2-\ga)$ by Proposition~\ref{P:Mon}. 

We now recall from~\eqref{E:omegaPrimeID} that
\begin{align*}
	\om' = -\frac{y\om}{G}\big(A(\om) + B(\rho,\om) + P(\rho,\om,p)\big),
\end{align*}
where $A$, $B$, $P$ are defined in~\eqref{E:ADef}--\eqref{E:PDef}. We have, from Lemma~\ref{L:Aux2},
\begin{align*}
	0 \leq P \leq \frac{1}{2}\Big(1-\frac{\al}{2}\Big)^2f(\om)(2-\ga-\om)
\end{align*}
for all $y\in(y_2,y_0)$, and hence,
\begin{align*}
	\om'(y) < -\frac{y\om}{G}(2-\ga-\om)\bigg(\bar{A}(\om) + \frac{1}{2}\Big(1-\frac{\al}{2}\Big)^2f(\om)\bigg),
\end{align*}
for all $y\in(y_2,y_0)$, where we have recalled the definition of $\bar{A}$ from~\eqref{E:BarADef}. As there exists a $c>0$ such that $|G(y)|\geq cy^2$ for $y>y_2$ (by Lemma~\ref{L:Supersonic}) and $\om$ is bounded (by Lemma~\ref{L:omegaUpper}), so that also $\bar{A}(\om)$ and $f(\om)$ are also bounded, this implies
\begin{align*}
	\big|\log(2-\ga-\om)'\big| \leq C
\end{align*}
and hence $y_0=\infty$, that is, $\om(y)<2-\ga$ for all $y\in(0,\infty)$.
\end{proof}

\begin{proof}[Proof of Proposition~\ref{P:omegaLim}]
Let $y_1$ and $\om_1$ be defined as in Lemma~\ref{L:Aux1}.

To show that $\om(y)\to2-\ga$ as $y\to\infty$, we assume for contradiction that
\begin{align*}
	\lim_{y\to\infty}\om = \om_\infty < 2-\ga,
\end{align*}
where we recall from Proposition~\ref{P:Mon} that $\om$ is monotone increasing and from Proposition~\ref{P:omegaBoundSharp} that $\om<2-\ga$. We recall again~\eqref{E:omegaPrimeID} and observe from the definition of $\om_1$ and~\eqref{E:PDef} that for $y>y_1$, $P(y)\geq0$. Now as $B(y)\to0$ as $y\to\infty$ due to Proposition~\ref{P:rhoLim}, there exists a $y_3>y_1$ such that for all $y>y_3$, we have
\begin{align*}
	|B(y)| &< \frac{1}{2}(2-\ga-\om_\infty)\bar{A}(\om_0).
\end{align*}
Thus, for $y>y_3$, we deduce from~\eqref{E:omegaPrimeID} in particular,
\begin{align*}
	\om' &> \frac{1}{y\om}\big(A(\om) + B(\rho,\om)\big) > \frac{2-\ga-\om_\infty}{2\om_\infty}\bar{A}(\om_0)\frac{1}{y},
\end{align*}
where we have used the fact that $A$ is an increasing function of $\om$. Integrating, we obtain a contradiction and hence $\om(y)\to2-\ga$ as $y\to\infty$.
\end{proof}

\begin{lemma}\label{L:Asymptotics}
	Let $(\rho,\om)$ be the unique global-in-$y$ solution to the initial value problem~\eqref{E:rhoPrime1}--\eqref{E:omegaPrime1} given by Theorem~\ref{T:GE}. Then there exists a constant $\rho_\infty>0$ such that we have the enhanced asymptotic behaviour:
	\begin{align}
		\rho(y) &= \rho_\infty y^{-\frac{2-\al}{2-\ga}} + O_{y\to\infty}\big(y^{-\frac32\frac{2-\al}{2-\ga}}\big), & \omega(y) &= 2-\ga + O_{y\to\infty}\big(y^{-\frac{1}{2}\frac{2-\al}{2-\ga}}\big).
	\end{align}
\end{lemma}

\begin{proof}
Using~\eqref{E:rhoPrime1}--\eqref{E:omegaPrime1} and $\om'\geq 0$ from Proposition~\ref{P:Mon}, we first observe the identity and estimate
\begin{align*}
	\frac{d}{dy}\bigg(\frac{p}{y\rho}\bigg) &= \bigg(\frac{n-1}{y} + (\ga-1)\frac{\rho'}{\rho} + \frac{n}{3}\frac{(\rho\om)'}{\rho\om}\bigg)\frac{p}{y\rho}\\
	&= \frac{1}{y\om}\big((n+3\ga-3)\om_0 - (3\ga-2)\om - (\ga-1)y\om'\big)\frac{p}{y\rho}\\
	&\leq \frac{1}{y\om}\big((n+3\ga-3)\om_0 - (3\ga-2)\om\big)\frac{p}{y\rho}\\
	&= \frac{1}{y\om}\big((n+3\ga-3)\om_0 - (3\ga-2)(2-\gamma) + o_{y\to\infty}(1)\big)\frac{p}{y\rho}\\
	&= -\bigg(\frac{\ga-\al}{2-\ga} + o_{y\to\infty}(1)\bigg)\frac{p}{y^2\rho},
\end{align*}
where we have recalled~\eqref{E:nDef} in the final identity. As $\al<\ga$ by~\eqref{E:Constraints}, we deduce
\begin{align*}
	\frac{p}{y\rho} = O_{y\to\infty}\big(y^{-\frac{\ga-\al}{2-\ga}}\big).
\end{align*}
Thus, returning to the ODE~\eqref{E:rhoPrime1}, we have found, for some $\tilde{\beta}>0$,
\begin{align*}
	y\frac{\rho'}{\rho} = \frac{y^2h}{G} + O_{y\to\infty}\big(y^{-\tilde{\beta}}\big) = -\frac{2-\al}{2-\ga} + O_{y\to\infty}\big(y^{-\tilde{\beta}}\big).
\end{align*}
Similarly, from the ODE~\eqref{E:omegaPrime1}, we find
\begin{align*}
	y\frac{(2-\ga-\om)'}{2-\ga-\om} = -\frac{1}{2}\frac{2-\al}{2-\ga} + o_{y\to\infty}(1).
\end{align*}
Returning these asymptotics into the ODE system~\eqref{E:rhoPrime1}--\eqref{E:omegaPrime1}, we deduce also the order of the next term in the asymptotics as claimed in the lemma.
\end{proof}

\subsection{Proof of the Main Theorem}

\begin{proof}[Proof of Theorem~\ref{T:Main}]
The proof is now a trivial consequence of Theorem~\ref{T:GE}, Proposition~\ref{P:rhoLim}, Proposition~\ref{P:omegaLim} and Lemma~\ref{L:Asymptotics}.
\end{proof}

\appendix

\section{Appendix}

\subsection{Interval Arithmetic}
\label{SS:A1}

As discussed in the proof of Lemma~\ref{L:Aux2} above, we require the inequality $H(\om_*)<0$ (recall~\eqref{E:HDef}) for $n=4,6$ and $\max\big\{\frac{19}{12},\frac{10+n}{9}\big\}<\ga<\frac{11}{6}$ to complete the proof. The purpose of this section is to provide the code necessary to establish this inequality rigorously via interval arithmetic. As the required use of this code is a simple one, requiring us only to check the sign of a single-variable polynomial on a fixed interval, we implement this in Python, using the packages numpy and mpmath, which has interval arithmetic facility.

Explicit expressions for $H(\om_*)$ for $n=4$ and $n=6$ are given respectively by:
\begin{align*}
	H(\om_*)\big|_{n=4} &= -\frac{1}{6750(3\ga-1)^3(3\ga-2)^3}\big(7830416 - 88689136\ga + 381102528\ga^2 - 755565356\ga^3\\
	&+ 589426507\ga^4 + 296797356\ga^5 - 1225677744\ga^6 + 1578233592\ga^7 - 1306770165\ga^8\\
	&+ 749610288\ga^9 - 291131982\ga^{10} + 72171000\ga^{11} - 10235160\ga^{12} + 629856\ga^{13}\big),\\
	H(\om_*)\big|_{n=6} &= -\frac{2}{9261\ga^3(3\ga-1)^3}\big(1185408 - 13646304\ga + 60843888\ga^2 - 132703484\ga^3\\
	&+ 149712106\ga^4 - 88073001\ga^5 + 23431590\ga^6 + 333231\ga^7 - 2459613\ga^8\\
	&+ 1620872\ga^9 - 821935\ga^{10} + 238122\ga^{11} - 33048\ga^{12} + 1728\ga^{13}\big).
\end{align*}
The necessary Python code to establish the negativity of each of these polynomials via interval arithmetic is as follows. The code gives a suitable definition for $H(\om_*)$ in the case $n=4$ or $n=6$, then subdivides the relevant interval of $\ga$ into subintervals of either length $10^{-8}$ (in the case $n=4$, giving 25000000 intervals) or $10^{-6}$ (in the case $n=6$, giving 55556 intervals), then verifies the necessary inequality on each subinterval.
\begin{verbatim*}
	import numpy as np
	import mpmath
	from mpmath import iv
\end{verbatim*}
\begin{verbatim*}
	def Hstar4(g):
	return iv.mpf(7830416) - iv.mpf(88689136)*g
	+ iv.mpf(381102528)*g**iv.mpf(2) - iv.mpf(755565356)*g**iv.mpf(3)
	+ iv.mpf(589426507)*g**iv.mpf(4) + iv.mpf(296797356)*g**iv.mpf(5)
	- iv.mpf(1225677744)*g**iv.mpf(6) + iv.mpf(1578233592)*g**iv.mpf(7)
	- iv.mpf(1306770165)*g**iv.mpf(8) + iv.mpf(749610288)*g**iv.mpf(9)
	- iv.mpf(291131982)*g**iv.mpf(10) + iv.mpf(72171000)*g**iv.mpf(11)
	- iv.mpf(10235160)*g**iv.mpf(12) + iv.mpf(629856)*g**iv.mpf(13)
	def Hstar6(g):
	return iv.mpf(1185408) - iv.mpf(13646304)*g
	+ iv.mpf(60843888)*g**iv.mpf(2) - iv.mpf(132703484)*g**iv.mpf(3)
	+ iv.mpf(149712106)*g**iv.mpf(4) - iv.mpf(88073001)*g**iv.mpf(5)
	+ iv.mpf(23431590)*g**iv.mpf(6) + iv.mpf(333231)*g**iv.mpf(7)
	- iv.mpf(2459613)*g**iv.mpf(8) + iv.mpf(1620872)*g**iv.mpf(9)
	- iv.mpf(821935)*g**iv.mpf(10) + iv.mpf(238122)*g**iv.mpf(11)
	- iv.mpf(33048)*g**iv.mpf(12) + iv.mpf(1728)*g**iv.mpf(13)
\end{verbatim*}
\begin{verbatim*}
	increment = iv.mpf(10**(-8))
	total = 25000000
	gint = iv.mpf(['19/12','1.5833333433333332'])
	counter = 0 
	for k in range(0,total):
	if(Hstar4(gint)>0):
	gint = gint + increment
	counter = counter + 1
	else:
	print("error at k =", k)
	print(counter)
	break
	print(counter)
	Returns: 25000000
\end{verbatim*}
\begin{verbatim*}
	increment = iv.mpf(10**(-6))
	total = 55556
	gint = iv.mpf(['16/9','1.7777787777777776'])
	counter = 0 
	for k in range(0,total):
	if(Hstar6(gint)>0):
	gint = gint + increment
	counter = counter + 1
	else:
	print("error at k =", k)
	print(counter)
	break
	print(counter)
	Returns: 55556
\end{verbatim*}

\subsection{Decomposition of $H$}
\label{SS:A2}

The purpose of this section is to sketch an alternative strategy for establishing the condition $H(\om_*)<0$ necessary for the proof of Lemma~\ref{L:Aux2} and proved via interval arithmetic in Section~\ref{SS:A1} above. The idea rests on a further decomposition of $H(\om_*)$ and a subdivision of the interval of $\ga$, and leads to checking a finite number of numerical inequalities, which must, however, be done rigorously.

We can make the decomposition
\begin{align*}
	H(\om_*) = H_+(n,\ga) - H_-(n,\ga),
\end{align*}
where:
\begin{align*}
	H_+(n,\ga) &= (2-\ga)C(n,\ga) + 3\ga(2-\ga-\om_*)\om_*^2\\
	&+ \frac{9(n-2)^2\ga}{8(3\ga-4)^2}\om_0^2(2-\ga-\om_*)\big(n\om_0+3(\om_*-\om_0)+4(2-\ga-\om_*)\big),\\
	H_-(n,\ga) &= \ga(2-\ga-\om_*)C(n,\ga) + \big(3(\om_*-\om_0)+4(2-\ga-\om_*)^2\big)\om_*^2\\
	&+ \frac{9(n-2)^2}{8(3\ga-4)^2}\om_0^2\big(n(2-\ga)\om_0+3\ga^2(2-\ga-\om_*)^2\big),\\
	C(n,\ga) &= (\om_*-\om_0)\om_* + \frac{n-2}{2(3\ga-4)}\om_0\om_* + \frac{27(n-2)^2\ga}{8(3\ga-4)^2}\om_0^2(\om_*-\om_0).
\end{align*}
The first aim is to show:
\begin{align}
	\frac{\pa}{\pa\ga}H_+(4,\ga) &< 0, & \frac{\pa}{\pa\ga}H_+(6,\ga) &< 0, &\frac{\pa}{\pa\ga}H_-(4,\ga) &< 0, & \frac{\pa}{\pa\ga}H_-(6,\ga) &< 0,\label{E:HDivs}
\end{align}
for $n=4,6$ and $\ga\in\big(\frac{19}{12},\frac{11}{6}\big)$, so that $H_+(n,\ga)$ and $H_-(n,\ga)$ can be realised as positive decreasing functions of $\ga$. After this, we check
\begin{align*}
	\min_{\ga\in I_k}\{H_-(n,\ga)\} > \max_{\ga\in I_k}\{H_+(n,\ga)\}
\end{align*}
for a sequence of subintervals $I_k$ covering $\big(\frac{19}{12},\frac{11}{6}\big)$. Now clearly $H_+(n,\ga)$ and $H_-(n,\ga)$ are positive, so to show~\eqref{E:HDivs}, it is sufficient to observe:
\begin{align*}
	\frac{\pa}{\pa\ga}\frac{9(n-2)^2\ga}{8(3\ga-4)^2}\om_0^2 &\propto \frac{\pa}{\pa\ga}\frac{\ga(2-\ga)^2}{(n+3\ga-6)^2} < \frac{1}{(n+3\ga-6)^2}\frac{\pa}{\pa\ga}\ga(2-\ga)^2 = -\frac{(3\ga-2)(2-\ga)}{(n+3\ga-6)^2} < 0,\\
	\frac{\pa}{\pa\ga}\frac{9(n-2)^2\ga}{8(3\ga-4)^2}\om_0^3 &\propto \frac{\pa}{\pa\ga}\frac{\ga(3\ga-4)(2-\ga)^3}{(n+3\ga-6)^3} < \frac{1}{(n+3\ga-6)^3}\frac{\pa}{\pa\ga}\ga(3\ga-4)(2-\ga)^3\\
	&= -\frac{(2-\ga)^2}{4(n+3\ga-6)^3}\big(5(12\ga-19)+3(11-6\ga)+(\ga-1)\big) < 0,\\
	\frac{\pa}{\pa\ga}\ga(2-\ga-\om_*) &\propto \frac{\pa}{\pa\ga}\bigg(1+\frac{n-2}{3\ga-1}\bigg)\ga(\ga-1)(11-6\ga) < \bigg(1+\frac{n-2}{3\ga-1}\bigg)\frac{\pa}{\pa\ga}\ga(\ga-1)(11-6\ga)\\
	&= -\bigg(1+\frac{n-2}{3\ga-1}\bigg)\big((11-6\ga)+2\ga(9\ga-14)\big) < 0,
\end{align*}
in addition to
\begin{align*}
	\frac{\pa}{\pa\ga}\om_* &= \frac{\pa}{\pa\ga}\bigg((2-\ga) - \frac{4(\ga-1)(11-6\ga)(n+3\ga-3)}{3(n+1)n(3\ga-1)}\bigg)\\
	&= -1 - \frac{4(\ga-1)(11-6\ga)(n+3\ga-3)}{3(n+1)n(3\ga-1)}\bigg(\frac{1}{\ga-1}-\frac{6}{11-6\ga}+\frac{3}{n+3\ga-3}-\frac{3}{3\ga-1}\bigg)\\
	&= -1 + \frac{4(12\ga-17)}{3(n+1)n}\bigg(1+\frac{n-2}{3\ga-1}\bigg) + \frac{4(n-2)(\ga-1)(11-6\ga)}{3(n+1)n(3\ga-1)^2}\\
	&< -1 +\frac{(3\ga+1)(12\ga-17)}{15(3\ga-1)} + \frac{2(\ga-1)(11-6\ga)}{15(3\ga-1)^2} < -\frac{1}{2}
\end{align*}
and
\begin{align*}
	\frac{\pa}{\pa\ga}(\om_*-\om_0) &= \frac{\pa}{\pa\ga}\bigg((2-\ga) - \frac{4(\ga-1)(11-6\ga)(n+3\ga-3)}{3(n+1)n(3\ga-1)} - \frac{(3\ga-4)(2-\ga)}{n+3\ga-6}\bigg)\\
	&< -\frac{1}{2} -\frac{\pa}{\pa\ga}\frac{(3\ga-4)(2-\ga)}{n+3\ga-6}\\
	&= -\frac{1}{2} - \frac{2}{n+3\ga-6} + \frac{2(3\ga-4)}{n+3\ga-6} + \frac{3(3\ga-4)(2-\ga)}{(n+3\ga-6)^2}\\
	&< -\frac{1}{2} - \frac{2}{3\ga} + \frac{2(3\ga-4)}{3\ga-2} + \frac{3(3\ga-4)(2-\ga)}{(3\ga-2)^2} < 0,
\end{align*}
with the signs of all five derivatives holding for $n=4,6$ and $\ga\in\big(\frac{19}{12},\frac{11}{6}\big)$. Thus we see that $H_+(n,\ga)$ and $H_-(n,\ga)$ are positive decreasing functions of $\ga$ for $n=4,6$ and $\ga\in\big(\frac{19}{12},\frac{11}{6}\big)$. This means for any subinterval $I_k=(\ga_k,\ga_{k+1}) \subset \big(\frac{19}{12},\frac{11}{6}\big)$, we have:
\begin{align*}
	\max_{\ga\in I_k}\{H_+(n,\ga)\} &= H_+(n,\ga_k), & \min_{\ga\in I_k}\{H_-(n,\ga)\} &= H_-(n,\ga_{k+1}).
\end{align*}
It may then be checked numerically that
\begin{align*}
	H_-(n,\ga_{k+1}) > H_+(n,\ga_k)
\end{align*}
for each $k\in\{1,\dots,512\}$, where:
\begin{align*}
	\ga_0 &= \max\bigg\{\frac{19}{12},\frac{10+n}{9}\bigg\}, & \ga_k &= \ga_0 + \frac{k-1}{512}\bigg(\frac{11}{6}-\ga_0\bigg).
\end{align*}

\end{document}